\documentclass[a4paper,10pt]{amsart} 
\usepackage[utf8]{inputenc}
\usepackage[english]{babel}
\usepackage{mathrsfs}
\usepackage{amsthm}
\usepackage{amsfonts}
\usepackage{amsmath}
\usepackage{amssymb}
\usepackage{faktor}
\usepackage{mathtools}
\usepackage{tikz}
\usepackage{esint}
\usepackage{centernot}
\usepackage{verbatim}
\usepackage[a4paper,top=3.5cm,bottom=3.5cm,left=2.5cm,right=2.5cm]{geometry}

\usepackage{cite}
\usepackage{color}
\definecolor{citegreen}{rgb}{0,0.8,0}
\definecolor{refred}{rgb}{0.8,0,0}
\usepackage[colorlinks, citecolor=citegreen, linkcolor=refred]{hyperref} 

\newtheorem{theorem}{Theorem}

\newtheorem{proposition}[theorem]{Proposition}
\theoremstyle{definition}
\newtheorem{remark}[theorem]{Remark}

\newtheorem{definition}[theorem]{Definition}

\newcommand{\R}{\mathbb R}\newcommand{\s}{\mathcal S}
\numberwithin{equation}{section} \numberwithin{theorem}{section}
\allowdisplaybreaks
\newcounter{stepctr}
{\end{list}}

\def\XXint#1#2#3{{\setbox0=\hbox{$#1{#2#3}{\int}$}
     \vcenter{\hbox{$#2#3$}}\kern-.5\wd0}}

\DeclareMathOperator{\Span}{Span}
\DeclareMathOperator{\SO}{SO}

\DeclareMathOperator{\Sc}{Sc}
\DeclareMathOperator{\Ric}{Ric}

\DeclareMathOperator{\Div}{div}
\DeclareMathOperator{\Esp}{Exp_p}
\DeclareMathOperator{\vol}{Vol_{m+1}}
\DeclareMathOperator{\area}{Vol_m}
\newcommand{\e}{\varepsilon}
\newcommand{\set}[1]{\{ {#1} \}}
\newcommand{\norm}[1]{\| {#1} \|}
\newcommand{\scal}[2]{\langle {#1} , {#2} \rangle}
\newcommand{\riem}[3]{R_p({#1},{#2}){#3}}

\DeclareUnicodeCharacter{2212}{-}

\title{Double Bubbles with High Constant Mean Curvatures in Riemannian Manifolds}
\author{Gianmichele Di Matteo, Andrea Malchiodi}
\address{Gianmichele Di Matteo, Karlsruhe Institute of Technology (KIT), Englerstrasse 2, 76131 Karlsruhe, Germany}
\address{Andrea Malchiodi, Scuola Normale Superiore, Piazza dei Cavalieri 7, 56126 Pisa, Italy}

\begin{document}

\begin{abstract}
We obtain  existence of  double bubbles of large and constant mean curvatures in Riemannian manifolds. These arise as perturbations of geodesic standard double bubbles centered at critical points of the ambient scalar curvature and aligned along eigen-vectors of the ambient Ricci  tensor. We also 
obtain general multiplicity results via Lusternik-Schnirelman theory, and extra ones in 
case of double bubbles whose opposite boundaries have the same mean curvature. 
\end{abstract}

\maketitle

\begin{center}
\end{center}

\section{Introduction}\label{sec.introduction}

The study of isoperimetric problems in different contexts has always been of great relevance 
both from the theoretical point of view, as well as from that of applications. On one hand, 
it has been one of the driving forces for the existence and regularity theory 
of minimal and constant mean curvature hypersurfaces, see  e.g. \cite{almgren76}. 
On the other, together with some variants, such problems are significant in different 
 models involving interfaces, see e.g. \cite{finn}. 
In this paper we will consider  \underline{closed smooth Riemannian manifolds} $(M^{m+1},g)$, that is compact without boundary,
both when recalling contributions from the literature and when stating our results, 
even though some of them could extend to other settings, such as the complete 
case or that of Euclidean domains: we will spend some extra words on these aspects later on.

Existence of isoperimetric regions in compact manifolds, within the class of sets with finite perimeter, 
is rather standard to achieve. Their regularity is a deeper question, in parallel with 
that of minimal surfaces, and it is well-known to generally hold only in  dimension $m < 7$, see e.g. \cite{morgan03}. 
There are though special situations in which regularity is guaranteed in all dimensions, 
as for example for isoperimetric sets with a small volume constraint - which might model droplets in a non-homogeneous environment - 
as shown in \cite{morganjohnson}
using Heintze-Karcher's inequality and Allard's regularity theorem. 
 In such a case, isoperimetric sets are known 
to be smooth graphs over some geodesic balls of small radii in the manifold: in \cite{druetpams} and \cite{nar} 
 the scalar curvature of $(M,g)$ is shown to play a role both in terms of the isoperimetric 
constant, as well as for the localization of  extremal sets (see also \cite{fall} for the case of manifolds with boundary). 

In the latter examples, critical domains have boundaries with large and constant mean curvature. A complementary aspect of the problem is to construct 
sets with such properties since, apart from isoperimetric sets, 
they might include  local minima or unstable 
critical points of the area functional under a small volume constraint. Their existence 
has been proved in e.g. \cite{ye} and \cite{pac} using the characterization of  Jacobi's operator 
on spheres and a finite-dimensional reduction of the problem (see also \cite{lm10}, \cite{lms20} 
for Willmore-type surfaces of small area). One might indeed look 
for candidate solutions as {\em pseudo-bubbles} - using the terminology from \cite{nar} - that 
solve the problem up to some Lagrange multiplier. In a final step, one then adjusts their 
center so that the Lagrange multipliers vanish as well: this is possible for example 
near non-degenerate critical points of the scalar curvature of $(M,g)$.

 \
 
 In this paper we are concerned with  counterparts of the latter results for {\em double-bubbles} in manifolds. 
 These are among the simplest cases of  {\em clusters}, namely 
 partitions  of the domain under interest into sets with prescribed volumes and so that the total 
 measure of the boundaries is minimal or extremal in some sense. Such models are important in 
 applications, as they describe for example foams, liquid crystals, magnetic domains and biological cells. 
 We refer the reader to the two books \cite{morgan09} and \cite{maggi} for a more detailed 
 introduction and for some relevant existence and regularity results.

Double bubbles (also referred to as {\em standard double bubbles}) in the Euclidean space $\R^{m+1}$ consist 
of two attached chambers with three boundary components that are caps of $m$-dimensional spheres, one of which 
is common to both chambers. In \cite{hut}, \cite{reichardt0} and \cite{reichardt}, as it was conjectured for some time, 
it was shown that these are solutions of the isoperimetric problem for two-chambers clusters. 
As it happened for isoperimetric sets of small volumes in manifolds, it is then natural to 
look for bubble clusters of large and constant (in each boundary component) mean curvature. However, differently from the 
single-chamber case, their characterization should be given by their orientation other than 
their location in the manifold (see also \cite{imm1}, \cite{imm2} for related issues concerning non-spherical 
Willmore surfaces of small area). Our first result answers precisely this question under generic local assumptions
on the metric. Some parts of the statement are not completely precise, but the meaning will be 
clear from the arguments of the proof.

 \begin{theorem}\label{th.nondeg}
 	Given a closed Riemannian manifold $(M^{m+1},G)$ and three numbers $H_0\ge 0$, $H_1,H_2>0$ such that $H_1=H_0+H_2$, 
 	let $p \in M$ be a non-degenerate critical point of the scalar curvature, and let  $\mu$ be an eigenvalue 
 	of $Ric_p$. Then there exists a constant $\rho_0>0$ such that the following holds: 
 	 	for any scale $\rho \in (0,\rho_0)$ there exists 
 	a smooth double bubble, whose sheets have constant mean curvatures $\rho^{-1}H_0$, $\rho^{-1}H_1$ and $\rho^{-1}H_2$ 
  and meet pairwise  at $120^{\circ}$-degrees. 
 	Such a double-bubble is contained in a geodesic ball  with radius of order $\rho$ centered at $p$, and it is aligned 
 	along a direction in the eigenspace of $Ric_p$ corresponding to the eigenvalue $\mu$.
 \end{theorem}

\begin{remark}
	We decided to work in a Riemannian setting, but a completely similar result would hold for 
	Euclidean domains in presence of inhomogeneous and/or non-isotropic conditions, as it might 	result in applications, see the above-mentioned references. 
	Also, instead of prescribing large values of the mean curvatures of the boundary 
	components, we could have chosen to fix small  boundary areas or  volumes enclosed by the chambers.  
\end{remark}

\begin{remark}\label{r:uniq}
	If the eigenvalue $\mu$ in Theorem \ref{th.nondeg} is simple, then there are exactly two solutions as in the statement 
	if $H_0 \neq 0$ and exacly one if $H_0 = 0$, see the final comments in the proof of the theorem. 
\end{remark}

The abstract strategy to prove the above result follows conceptually the approach for the 
construction of small CMC spheres, relying on a finite-dimensional reduction. On the other 
hand, we face several crucial differences. First of all, expansions of geometric quantities like 
volumes and perimeters for the perturbed surfaces are technically more involved due to the presence of tangential components in the perturbation: these are needed to deal with the structure of 
piecewise-smooth hypersurfaces, and more precisely to patch 
together the three smooth components, which should all satisfy a given PDE under suitable coupled boundary conditions.

We begin by constructing  {\em approximate solutions} to our problem, 
consisting in mapping small standard double bubbles in Euclidean space into the 
manifold via the exponential map from a given point. This process will of course 
affect the CMC condition on each component, but yet an asymptotic expansion in terms 
of the ambient curvature and their orientation can be worked out. The next step consists in deforming the 
approximate solutions into the so-called {\em pseudo-double bubbles}. By this we 
mean a family of solutions to the given problem up to some Lagrange multipliers. 
These appear by the fact that standard double-bubbles are degenerate minima of the two-chambers 
isoperimetric problem, since they can be arbitrarily translated or rotated. As it was recently proved in 
\cite{dim} though, such a degeneracy is the minimal possible, which implies that one can solve the problem when, roughly,
one restricts to the orthogonal complement of Euclidean isometries within the space of variations. 
In working out this procedure, which relies on  the contraction mapping theorem, we need to 
couple normal variations to tangential ones, in order to keep the structure of the 
triple points at the common boundary. For this step in particular, a careful choice of the 
tangential components is needed in order to achieve given boundary conditions and norm estimates at the same time. We opt for imposing additional fictitious conditions on the tangential components, by prescribing their divergence, so that we are led to consider an \emph{overdetermined non-elliptic problem}; this is shown to admit a unique solution employing methods from Hodge theory, originated earlier in the study of problems from fluid dynamics, see \cite{sch,aco,gir}.

The last crucial step in the proof consists in adjusting the parameters - centering and orientation - 
when choosing  
approximate solutions in order to fully solve the problem, including the degenerate directions. 
For this goal we exploit the variational character of the problem, as in \cite{ambmal}, showing the 
existence of a {\em reduced energy functional} $\Phi_\rho$, defined on the unit tangent bundle of $M$, 
whose critical points  yield solutions to our problem. An accurate expansion of this 
quantity, that depends on the fixed point argument in the previous step and is defined 
for $(p,\s) \in UTM$, the unit tangent bundle of $M$, yields the following result for $\rho$ sufficiently small 
\begin{equation}\label{eq.approximateasym}
	\norm{\Phi_\rho(p,\s)-\Sc(p)A(m,H_0,H_1,H_2)-\Ric(\s,\s)B(m,H_0,H_1,H_2)}_{C^k(UTM)} \le c_k \rho,
\end{equation}
where $A(m,H_0,H_1,H_2)$, $B(m,H_0,H_1,H_2)$ are positive constants and $c_k > 0$ 
depends on the given order of regularity $k$.  By extremizing the above quantity with respect to the 
parameters $p$ and $\s$, we finally obtain Theorem \ref{th.nondeg}.

\

The above expansion \eqref{eq.approximateasym} allows to obtain 
multiplicity results as well, both assuming the existence 
of a non-degenerate critical point of the scalar curvature, or with no assumptions at all. 
Our next result is stated as follows, where $cat$ stands for the Lusternik-Schnirelman 
category, see e.g. Chapter 9 in \cite{ambmal2}.

\begin{theorem}\label{th.mainasymmetric}
Given a compact Riemannian manifold $(M^{m+1},G)$ and three  numbers $H_0,H_1,H_2>0$ such that $H_1=H_0+H_2$, there exists a constant $\rho_0>0$ such that the following holds. For any scale $\rho \in (0,\rho_0)$, there exist at least $cat(UTM)$ $m$-dimensional smooth double bubbles, whose sheets have constant mean curvatures $\rho^{-1}H_0$, $\rho^{-1}H_1$,  $\rho^{-1}H_2$ and 
mutually  meet at $120^{\circ}$-degrees with each other. 

Moreover, if $p$ and $\mu$ are as in Theorem \ref{th.nondeg} and if $H_1 \neq H_2$, there are  at least two CMC double-bubbles as above 
 oriented along a direction in the eigenspace of $Ric_p$  corresponding to the eigenvalue $\mu$.
\end{theorem}

\begin{remark}
	In the papers \cite{renwei} and \cite{renwei2} the existence of configurations in the form of double bubbles 
	minimizing a free energy in a ternary system was proved in two-dimensional domains, even though 
	a precise characterization of minimizers is not given in terms of their location and orientation. 
	Our expansions, which hold in general dimension, also can determine such properties in generic cases. 
\end{remark}

Other multiplicity results can be found when the equalities $H_0=0$ and $H_1=H_2$ are imposed. Under these conditions, we obtain
\begin{equation}\label{eq.symmetry-Phi-rho}
	  \Phi_\rho(p,\s) = \Phi_\rho(p,-\s) \qquad \hbox{ for all } p \in M \hbox{ and } \s \in UT_p M, 
\end{equation}
which means that $\Phi_\rho$ is well-defined on the {\em projective} tangent bundle of $M$, denoted by $\mathbb{P} TM$, where opposite points in $UTM$ with the same base are identified. Moreover, the following expansion holds
\begin{equation}\label{eq.approximatesym}
	\norm{\Phi_\rho(p,\s)-\Sc(p)A^{sym}(m)-\Ric(\s,\s)B^{sym}(m)}_{C^k(\mathbb{P}TM)} \le c_k \rho^2,
\end{equation}
for some positive constants $A^{sym}(m)$, $B^{sym}(m)$ and $c_k$. Altogether,
by some general variational principles, see e.g. Chapter 10 in \cite{ambmal2}, symmetries lead to 
the existence of further critical points.

\begin{theorem}\label{th.mainsymmetric}
Given a compact Riemannian manifold $(M^{m+1},G)$ and a positive number $H>0$, there exists a constant $\rho_1>0$ such that the following holds. For any scale $\rho \in (0,\rho_1)$, there exist at least $cat(\mathbb{P} TM)$ symmetric double bubbles, whose sheets have constant mean curvatures $0$ and $\rho^{-1}H$, and 
whose sheets meet pairwise at $120^{\circ}$-degrees. 

If $p$ and $\mu$ are as in Theorem \ref{th.nondeg}, with $\mu$ of multiplicity $j > 1$, there are at least $j$ CMC double-bubbles as above 
oriented along a direction in the eigenspace of $Ric_p$  corresponding to the eigenvalue $\mu$. 
\end{theorem}

We remark that since $UTM$, $\mathbb{P}TM$ are fiber bundles with topologically non-trivial fibers, they might have a greater Lusternik-Schnirelman category than $M$, and in any case not smaller than that of $M$.

\begin{remark}\label{r:ganea}
	For {\em parallelizable manifolds} $M$, which is always the case in three dimensions by Stiefel's theorem (see \cite{mil}),  the tangent bundle is trivial 
	and therefore $UTM \simeq M \times S^{m}$ and $\mathbb{P}TM \simeq M \times \mathbb{RP}^m$. 
	For such manifolds 
	 {\em Ganea's conjecture} from 1971 stated that products of the type $ M \times S^m$ have category larger than 
	 that of $M$. While this fact is true when the dimension of $M$ is no more than five (see \cite{vandem}), 
	 it was in general disproved  in \cite{iwase}, with a minimum-dimensional example in \cite{stanley}. 
	 
	 In any case it is easy to see that $cat(\mathbb{P}TM) \geq cat(UTM) \geq cat(M)$. 
\end{remark}

\

The plan of the paper is the following. In Section \ref{sec.preliminary} we collect some 
known properties of Euclidean standard double bubbles, while in Section \ref{sec.geodesic} 
we perform asymptotic expansions of areas and volumes of their exponential maps 
over the manifold $M$.  Section \ref{sec.perturbed} is devoted to the perturbation, 
via normal and tangential variations, of the above surfaces and to their effects on 
the variations of first and second fundamental forms, as well as of their mean curvature. 
In Section \ref{sec.pseudodb} we construct {\em pseudo-bubbles} via a fixed point 
argument, while finally in Section \ref{sec.existence} we work out the variational 
strategy to find existence and multiplicity of critical points, relying also on the 
expansion of the reduced functional $\Phi_\rho$.

\

\begin{center}
	Acknowledgments
\end{center}

AM has been supported by the project {\em Geometric problems with loss of compactness} from the Scuola Normale Superiore. He is also member of GNAMPA as part of INdAM. The authors are grateful to Frank Morgan for some useful comments.

\section{Preliminary Results and  Notation}\label{sec.preliminary}
In this section we recall a few ingredients from the theory of double bubbles, with particular emphasis on the Euclidean space case, which will be extensively used throughout the paper.

Following \cite{hut}, we define a \emph{double bubble} in an ambient Riemannian manifold $(M^{m+1},G)$ as the union of the topological boundaries of two disjoint regions $B_1$ and $B_2$ with respective volumes $V_1$ and $V_2$. We will always assume $V_1 \le V_2$, and call a double bubble \emph{symmetric} if $V_1=V_2$ and \emph{asymmetric} if $V_1<V_2$. A piecewise smooth hypersurface $\Sigma \subset M$ is a \emph{smooth double bubble} if it consists of three compact, $m-$dimensional, smooth and orientable pieces $\Sigma^0$, $\Sigma^1$ and $\Sigma^2$, meeting along a common $(m-1)-$dimensional boundary $\Gamma$, and such that $\Sigma^1 \cup \Sigma^0=\partial B_1$ and $\Sigma^2 \cup \Sigma^0=\partial B_2$. The unit normal vector field along $\Sigma$, denoted by $N$, is chosen so that it points into $B_1$ along $\partial B_1$, and into $B_2$ along $\Sigma^2$. The associated second fundamental form and mean curvature, with our 
convention the trace of the second fundamental form,  are denoted respectively by $h$ and $H$. At points in $\Gamma$ these objects are not uniquely defined, but their value depends on the sheet used to compute them instead; we will therefore denote $N^{\sigma}$, $h^{\sigma}$ and $H^{\sigma}$ their restrictions to the sheet $\Sigma^{\sigma}$, where $\sigma=0,1,2$.

Let us now focus on the case in which the ambient manifold is the Euclidean space $\R^{m+1}$.
We call a smooth double bubble $\Sigma$ a \emph{standard double bubble} if $\Sigma^{\sigma}$ is a spherical cap for $\sigma=1$, or $2$, and  $\Sigma^0$ is a spherical cap or a flat ball in the asymmetric and symmetric case respectively, and the sheets $\Sigma^\sigma$ meet in an equi-angular way (that is at $120^\circ$-degrees) along an $(m-1)-$dimensional sphere $\Gamma$. According to our convention on the unitary normal vectors $N^\sigma$, we must have $N^1=N^0+N^2$. For any point $p\in \Gamma$, we can decompose the tangent space $T_p \Sigma^{\sigma}$ orthogonally as $T_p \Sigma^{\sigma}=T_p \Gamma \bigoplus \Span(\nu^\sigma)$ where $\nu^\sigma \in T_p \Sigma^{\sigma}$ is the unitary normal vector to $\Gamma$ at $p$ pointing inward $\Sigma^\sigma$. The condition on the sheets meeting in an equi-angular way can now be written as $\nu^0+\nu^1+\nu^2=0$.

We will call a standard double bubble $\Sigma$ \emph{centered} if $\Gamma$ is centered at the origin $0 \in \R^{m+1}$. We will denote by $\s$ the unitary normal vector to the hyper-plane containing $\Gamma$ and pointing towards $B_1$, and say that $\Sigma$ is \emph{aligned along} $\s$. When calling $\s$ the \emph{symmetry axis} of $\Sigma$, we are identifying $\s$ with its linear span; clearly, $\Sigma$ is rotationally symmetric along its symmetry axis $\s$. A centered standard double bubble is uniquely determined by its alignment vector $\s$ and the two volumes $V_1$ and $V_2$, or the radii $R_\sigma$ for $\sigma=0,1,2$. 

\begin{figure}[h]\label{fig.doublebubbles}
	\begin{minipage}[l]{0.48\linewidth}
	\hspace{-3.5cm}
\includegraphics[scale=0.45]{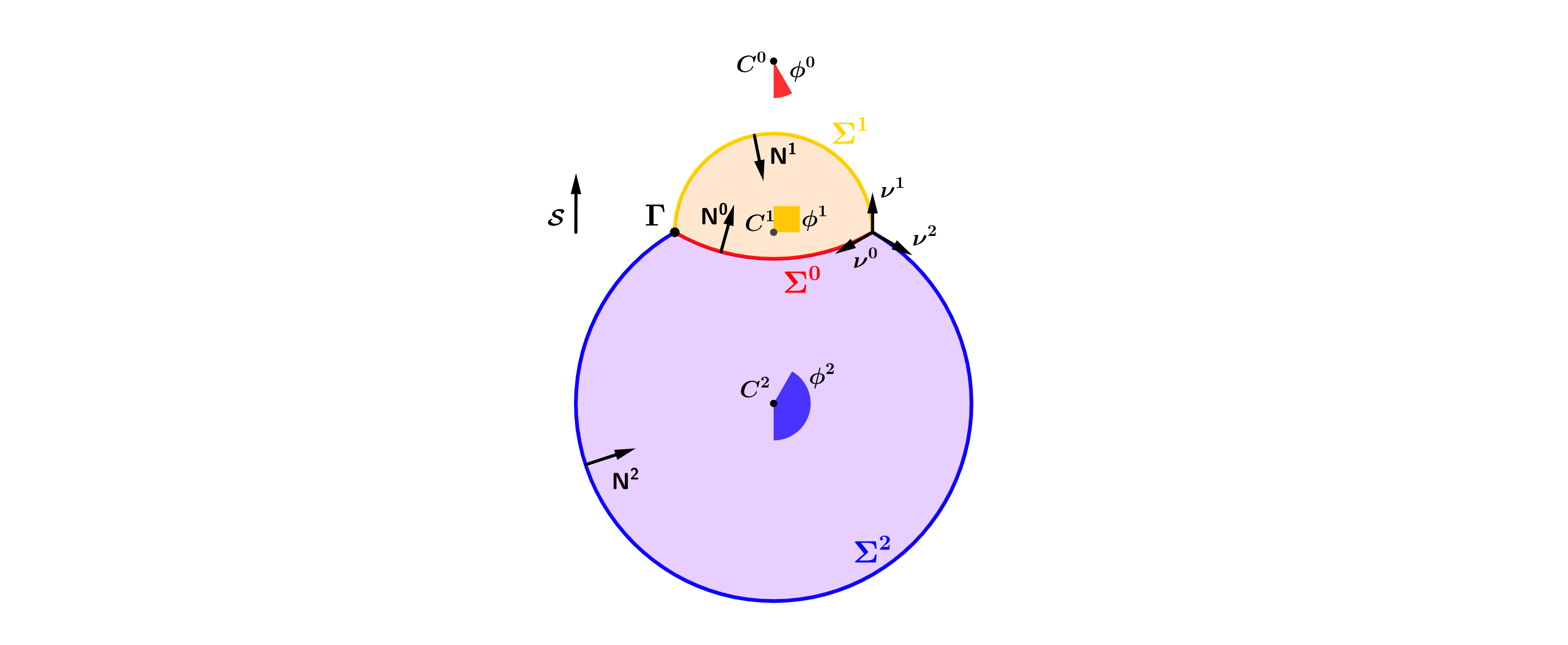}
\end{minipage} \hspace{1cm}
\begin{minipage}[r]{0.3\linewidth}
	\hspace{-2.5cm}
\includegraphics[scale=1.4]{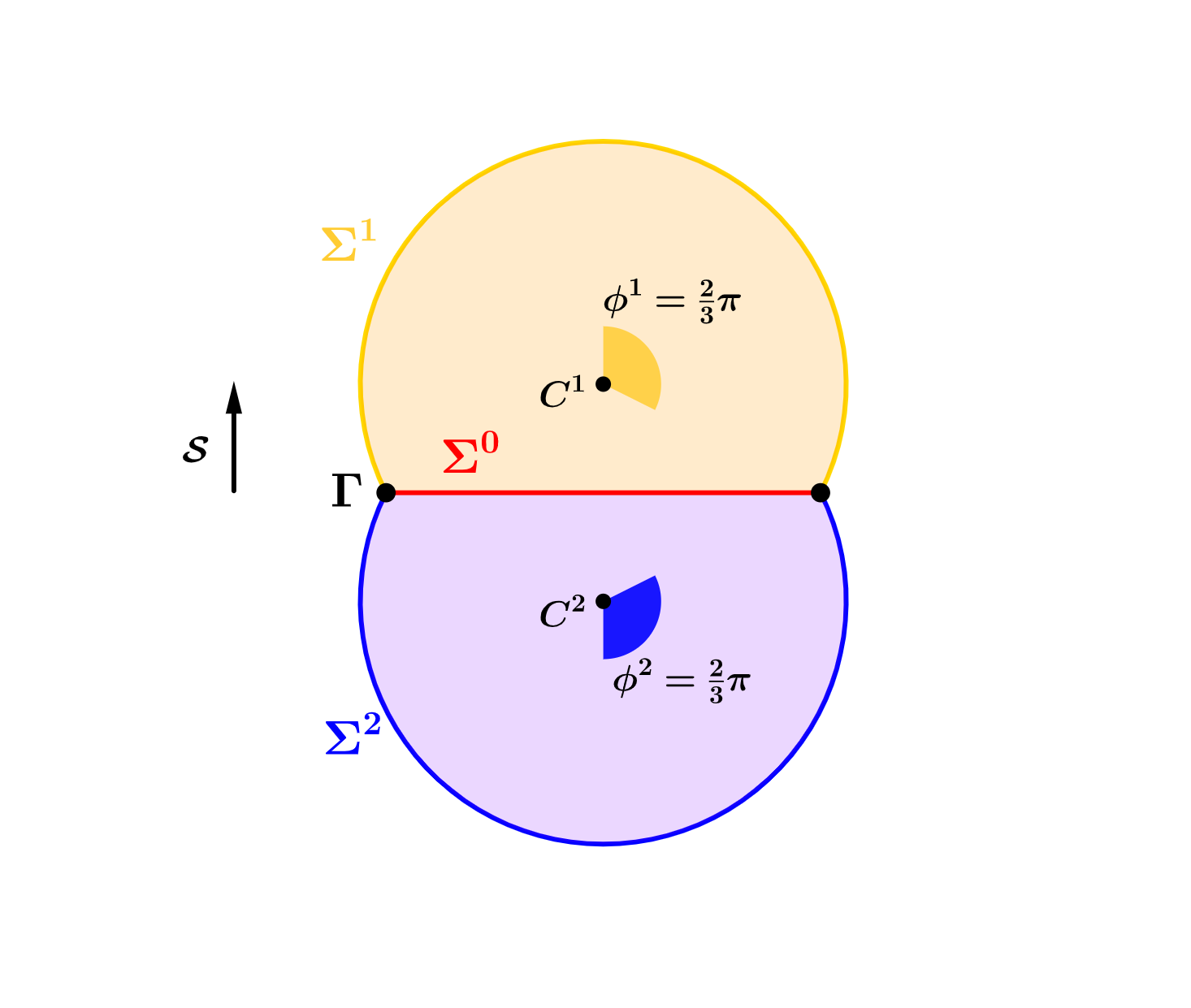} 
\end{minipage}
	\caption{A standard double bubble enclosing the regions $B_1$ (orange) and $B_2$ (violet) and a symmetric one}
\end{figure}

Consider now a centered asymmetric standard double bubble $\Sigma$ with symmetry axis $\s$. Let $\phi^\sigma$ be the angle between the pole of the cap of sphere $\Sigma^{\sigma}$ and $\Gamma$, and $R_{\sigma}$ be its radius $\Sigma^{\sigma} \subset S(C^\sigma,R_\sigma)$. Then, as it can be easily seen from volume-constrained variations, the (constant mean) curvatures $H_\sigma=m (R_\sigma)^{-1}$ of the different pieces verify the balance equation $H_1=H_0+H_2$. We adopt the convention that the mean curvature of a hyper-surface is given by the trace of its second fundamental form. Moreover, we can write the centers $C^\sigma$ as $C^\sigma=-R_\sigma \cos(\phi^\sigma) \s$, and we have the geometric equations
\begin{equation}
R_0\sin(\phi^0)=R_1\sin(\phi^1)=R_2\sin(\phi^2), \quad \phi^0+\phi^1=\tfrac{2}{3}\pi, \quad \phi^1+\phi^2=\tfrac{4}{3}\pi, \quad \phi^2+\phi^0=\tfrac{2}{3}\pi.
\end{equation}
For a symmetric standard double bubble the situation becomes even easier. Define $\phi^\sigma$ and $R_\sigma$ as above for $\sigma=1,2$, whereas define $R_0$ to be the radius of the $m-$dimensional ball $\Sigma^0$. Then $R_1=R_2=:R$, $R_0=\tfrac{\sqrt{3}}{2}R$, the mean curvatures satisfy the balance equations $H_1=H_2=m R^{-1}$ and $H_0=0$, the centers can be written as $C^1=-C^2=\tfrac{R}{2}\s$ and finally we obtain the geometric equation $\phi^1=\phi^2=\tfrac{2}{3}\pi$.

Finally, we denote by $g_{\sigma}$ and $h_\sigma$ respectively the metric and the second fundamental form on $\Sigma^\sigma$ induced by the embedding. Note that $h$ is either null or the constant  multiple of the identity $R^{-1} \cdot Id$.

\section{Normal Coordinates and Geodesic Double Bubbles}\label{sec.geodesic}
In this section we introduce the concept of geodesic double bubble in a Riemannian manifold, inspired by the standard definition of geodesic ball, and then derive some expansions for its area and enclosed volumes.
In order to do so, we fix a point $p$ in an $(m+1)-$dimensional Riemannian manifold $(M,G)$, and introduce normal coordinates on a neighbourhood $U$ of $p$: we consider an orthonormal basis $\{ E_\mu \}$, $\mu = 1,...,m+1$, of $T_p M$, and set $F \colon \mathbb{R}^{m+1} \longrightarrow M$ by $F(x) \coloneqq \Esp(x^\mu E_\mu)$, where $\Esp$ is the exponential map of $(M,G)$ at $p$ and summation on repeated indices is understood. This choice  induces coordinate vector fields $X_\mu = 
F_*(\partial_{x^\mu})$. Notice that $G_{\mu,\nu} \coloneqq G(X_\mu,X_\nu)=\delta_{\mu,\nu}$ at $p$, thus we will always endow $T_p M$ with the Euclidean scalar product, denoted by $\scal{\cdot}{\cdot}$. The importance of these coordinates relies in the following classical result (see \cite{wil}).
\begin{proposition}\label{prop.metricnormalcoord}
At the point $q = F(x)$, the metric the following expansion holds for  $\mu,\nu = 1, . . . , m + 1$
\begin{equation}\label{eq.metricnormalcoord}
G(X_\mu,X_\nu) =\delta_{\mu,\nu} + \tfrac{1}{3}\scal{\riem{\Xi}{E_\mu}{\Xi}}{E_\nu} + \tfrac{1}{6} \scal{\nabla_\Xi \riem{\Xi}{E_\mu}{\Xi}}{E_\nu} + \mathcal{O}(|x|^4), 
\end{equation}
where $R_p$ is the Riemann curvature tensor at the point $p$, and $\Xi = x^\mu E_\mu \in T_p M$.
\end{proposition}

We adopt the following convention on $\mathcal{O}(\cdot)$: for a real variable $t \in [0,1)$ and a natural number $N \in \mathbb{N}$, $\mathcal{O}(t^N)$ denotes a smooth function on $M \times (0,1)$, depending only on the geometry of the ambient manifold $M$, such that for every $j,k \in \mathbb{N}$ we have the following bound on the function and its derivatives
\begin{equation}
\left| \frac{d}{dt^j} \nabla^{(k)} \mathcal{O}(t^N) \right| (t,\cdot) \le C t^{N-j} \text{ on $M$,}
\end{equation}
for some constant $C>0$. In the following, we will always absorb $\mathcal{O}(t^M)$ in $\mathcal{O}(t^N)$ if $M \ge N$, which is coherent to our notation since $t <1$.

We identify the tangent space $T_p M$ with $\R^{m+1}$ through the linear isometry $\iota \colon \mathbb{R}^{m+1} \longrightarrow T_p M$ which sends $\partial_{x^\mu}$ to $E_{\mu}$; this expansion for the metric allows one to compute expansions for several geometric quantities of objects inside the normal neighbourhood $U$, and will therefore play a crucial role throughout the paper.
In the following, we will regularly identify points $Q$ in the unitary tangent bundle $UTM$ with their natural projections $(q,\s'):=(\pi_M (Q),\pi_{T_{\pi_M(Q)} M}(Q))$, so that $Q=(q,\s') \in M \times T M$. Moreover, given a centered standard double bubble $\Sigma \subset \R^{m+1}$ aligned along a vector $\s$, we identify it with its image $\Sigma \subset T_p M$ through the map $\iota$, so that $\s$ is identified with a vector $\s \in T_p M$.
\begin{definition}[Geodesic Double Bubbles]\label{def.geodesicdb}
For any point $(p,\s)\in UTM$, volumes $V_1\le V_2$ and a small enough scale $\rho>0$, we define the \emph{geodesic double bubble centered at $p$, aligned along $\s$ and at scale $\rho$}, as the smooth double bubble $\Sigma_{(p,\s),\rho}:=\Esp(\rho \Sigma)$, where $\Sigma \subset \R^{m+1}$ is the unique centered standard double bubble aligned along $\s$, enclosing volumes $V_1$ and $V_2$, identified as above with its image through the map $\iota$. We call a geodesic double bubble \emph{symmetric} or \emph{asymmetric} accordingly to whether $\Sigma$ is symmetric or asymmetric.
\end{definition}
It is clear from Definition \ref{def.geodesicdb} that if a geodesic double bubble $\Sigma_{(p,\s),\rho}$ is symmetric, then we must have $\Sigma_{(p,\s),\rho}=\Sigma_{(p,-\s),\rho}$, and hence the set of symmetric geodesic double bubbles at a fixed scale $\rho$ can be parametrized by points in the projective unitary tangent bundle $\mathbb{P}TM$, compare with Theorem \ref{th.mainsymmetric} in the introduction.  We remark that we could have fixed the mean curvatures $H_\sigma$'s of the sheets instead of the two enclosed volumes $V_1$ and $V_2$. Hence, from now on, we will assume that $\rho<\rho_1(M,H_0,H_1,H_2)$ is small enough so that $\rho \Sigma$ is contained in a normal neighbourhood. 
\subsection{Volumes of Geodesic Double Bubbles}\label{sub.volume}
Using the expansion \eqref{eq.metricnormalcoord} of the metric $G$ in normal coordinates, we can compute the volumes $(V_1)_{(p,\s),\rho}$ and $(V_2)_{(p,\s),\rho}$ enclosed by a geodesic double bubble $\Sigma_{(p,\s),\rho}$ as functions of the volumes $V_1$ and $V_2$ enclosed by $\Sigma \subset T_pM$ as in Definition \ref{def.geodesicdb}, the point $(p,\s)\in UTM$ and the scale $\rho$. Let us denote by $P^\sigma$ the volume enclosed by the spherical cap $\Sigma^\sigma$ and the $m$-dimensional ball $D:=co(\Gamma) \subset T_p M$, that is the convex hull of $\Gamma$; analogously, $P^\sigma_{(p,\s),\rho}$ is the region enclosed by $\Sigma_{(p,\s),\rho}^\sigma$ and the image of the disk $\rho \cdot \Esp(D)$ through the exponential map. Therefore, the enclosed volumes  $V_1$ and $V_2$ (respectively $(V_1)_{(p,\s),\rho}$ and $(V_2)_{(p,\s),\rho}$) verify
\begin{align}\label{eq.volumedecomposition}
V_1=|P^1|+|P^0|, &\quad V_2=|P^2|-|P^0|,\\
(V_1)_{(p,\s),\rho}=\vol(P^1_{(p,\s),\rho})+\vol(P^0_{(p,\s),\rho}), &\quad (V_2)_{(p,\s),\rho}=\vol(P^2_{(p,\s),\rho})-\vol(P^0_{(p,\s),\rho}).\nonumber
\end{align}
Here and elsewhere we denote by $|\cdot|_k$ and $\text{Vol}_k$ the $k$-dimensional Hausdorff measure of a set in $\R^{m+1}$ and $(M,G)$ respectively.
These formulas remain true in the symmetric case $V_1=V_2$, where $D=\Sigma^0$, once noticed that $|P^0|_{m+1}=\vol(P^0_{(p,\s),\rho})=0$. From \eqref{eq.metricnormalcoord}, one deduces the following expansion of the volume form in normal coordinates
\begin{equation}
d\mu_G=\sqrt{G} dx=\Big(1-\tfrac{1}{6} \Ric_{\mu \nu}x^\mu x^\nu +\mathcal{O}(|x|^3) \Big) dx.
\end{equation}
Here $\Ric$ denotes the Ricci tensor of the ambient metric $G$, and $\sqrt{G}$ the square root of the determinant of the matrix $G_{\mu,\nu}$. We will use this expansion to compute the volumes defined above through the formula 
\begin{equation}
\vol(P^\sigma_{(p,\s),\rho})=\int_{\rho P^\sigma}{\sqrt{G} dx}.
\end{equation}
Let us focus on the case $\sigma=1$, since the other cases are similar.
Extend $\s$ to an orthonormal basis, so that the matrix $A$ with columns $A:=(v_1|v_2|...|v_m|\s)\in \SO (m+1)$. Next we set $\bar{B}^1:=A^{-1}(B^1-C^1)$, so that $\bar{P}^1$ is the region enclosed by a spherical cap centered at the origin, with radius $R_1$, opening angle $\phi^1$, symmetry axis $e_{m+1}=(0,0,...,0,1)$ and pole $R_1 e_{m+1} \in \bar{P}^1$. Using the change of variables $y\coloneqq A^{-1}(\tfrac{x}{\rho}-C^1)$ we get
\begin{align}
\rho^{-(m+1)}\vol(P^1_{(p,\s),\rho})&=\int_{\bar{P}^1}{1-\tfrac{\rho^2}{6} \Ric(Ay+C^1,Ay+C^1) dy}+\mathcal{O}(\rho^3)\nonumber\\
&=\Big( 1-\tfrac{\rho^2}{6} \Ric(C^1,C^1) \Big) |\bar{P}^1|-\tfrac{\rho^2}{6}\int_{\bar{P}^1}{T_{\mu \nu}y^\mu y^\nu dy}-\tfrac{\rho^2}{3}\int_{\bar{P}^1}{W_\mu y^\mu dy}+\mathcal{O}(\rho^3),\label{eq.expansionP1}
\end{align}
where $T=A^T \Ric A$ and $W_\mu=(A^T \Ric)_{\mu \nu}(C^1)^\nu=(\Ric A)_{\mu \nu}(C^1)^\nu$. Using the symmetry of $\bar{P}^1$, and integrating over the level sets of the map $x \mapsto x^{m+1}$, we have
\begin{equation}
\begin{aligned}
\int_{\bar{P}^1}{y^\mu dy}&=0 \ \ \forall \mu=1,...,m; \\
\int_{\bar{P}^1}{y^{m+1} dy}&=\int_{0}^{\phi^1}{\omega_m (R_1 \sin(\phi))^{m+1}R_1 \cos(\phi) d\phi}=\omega_m R_1^{m+2}\tfrac{\sin^{m+2}(\phi^1)}{m+2}.
\end{aligned}
\end{equation}
Hence, we deduce
\begin{equation}\label{eq.linearP1}
\resizebox{0.93\hsize}{!}{ $
-\tfrac{\rho^2}{3}\int_{\bar{P}^1}{W_\mu y^\mu dy}=-\tfrac{\rho^2}{3}\omega_m R_1^{m+2}\tfrac{\sin^{m+2}(\phi^1)}{m+2}\Ric(C^1,\s)=\tfrac{\rho^2}{3}\omega_m R_1^{m+3}\tfrac{\sin^{m+2}(\phi^1)}{m+2}\cos(\phi^1) \Ric(\s,\s).$}
\end{equation}
In the last step we used the fact that the double bubble $\Sigma$ is centered to get $C^1=-R_1 \cos(\phi^1) \s$. In order to make the   quadratic term explicit, it is convenient to define $I_m(x) \coloneqq \int_0^x{\sin^m(t) dt}$. Using again the symmetry of $\bar{P}^1$, we obtain
\begin{equation}
\begin{aligned}
\int_{\bar{P^1}}{y^\mu y^\nu dy}&=0 \ \ \forall \mu \neq \nu \\
\int_{\bar{P^1}}{(y^{\mu})^2 dy}&=\tfrac{1}{m} \sum_{\mu=1}^m \int_{\bar{P^1}}{(y^{\mu})^2 dy}=\tfrac{1}{m} \int_{0}^{\phi^1}{R_1 \sin(\phi) \bigg( |S^{m-1}| \int_0^{R_1 \sin(\phi)}{r^2 \cdot r^{m-1} dr} \bigg) d\phi}\\
&=\tfrac{|S^{m-1}|}{m} \int_{0}^{\phi^1}{ \tfrac{R_1^{m+3}\sin^{m+3}(\phi)}{m+2} d\phi}=\omega_m R_1^{m+3} \tfrac{I_{m+3}(\phi^1)}{m+2} \ \ \ \forall \mu=1,...,m\\
\int_{\bar{P^1}}{(y^{m+1})^2 dy}&=\int_{0}^{\phi^1}{\omega_m (R_1 \sin(\phi))^{m+1} (R_1 \cos(\phi))^2 d\phi}=\omega_m R_1^{m+3}(I_{m+1}(\phi^1)-I_{m+3}(\phi^1)).
\end{aligned}
\end{equation}
Before substituting this in the above expression, we notice that an integration by parts easily implies
\begin{equation}\label{eq.Irecursion}
I_{m+1}(x)=\big( 1+\tfrac{1}{m+2} \big)I_{m+3}(x)+\tfrac{1}{m+2} \sin^{m+2}(x)\cos(x),
\end{equation}
and that we have $tr(T)=tr(\Ric(p))=\Sc(p)$ and $T_{m+1,m+1}=\Ric(\s,\s)$. Therefore we obtain
\begin{align}
-\tfrac{\rho^2}{6}\int_{\bar{P}^1}{T_{\mu \nu}y^\mu y^\nu dy}&=-\tfrac{\rho^2}{6} \bigg[ \sum_{\mu=1}^m \omega_m R_1^{m+3}\tfrac{I_{m+3}(\phi^1)}{m+2} T_{\mu \mu}+\omega_m R_1^{m+3} (I_{m+1}(\phi^1)-I_{m+3}(\phi^1)) T_{m+1,m+1}\bigg] \nonumber\\
&= -\tfrac{\rho^2}{6}\omega_m R_1^{m+3} \bigg[ \tfrac{I_{m+3}(\phi^1)}{m+2} \Sc(p)+ \tfrac{1}{m+2} \sin^{m+2}(\phi^1)\cos(\phi^1)\Ric(\s,\s)\bigg].\label{eq.quadraticP1}
\end{align}
Substituting \eqref{eq.linearP1} and \eqref{eq.quadraticP1} in \eqref{eq.expansionP1}, using again that $\Sigma$ is centered and \eqref{eq.Irecursion}, and plugging in the value $|\bar{P}^1|=\omega_m R_1^{m+1} I_{m+1}(\phi^1)$, we arrive at
\begin{equation}\label{eq.volumeP1}
\resizebox{0.92\hsize}{!}{ $
\rho^{-(m+1)}\vol(P^1_{(p,\s),\rho})= |P^1|_{m+1}-\tfrac{\rho^2}{6}\omega_m R_1^{m+3} \bigg[ \tfrac{I_{m+3}(\phi^1)}{m+2} \Sc(p)+ \Big(\tfrac{m+3}{m+2} I_{m+3}(\phi^1)-I_{m+1}(\phi^1) \sin^2(\phi^1) \Big)\Ric(\s,\s)\bigg]+\mathcal{O}(\rho^3)$}
\end{equation}
Analogous expansions for the volumes $\vol(P^\sigma_{(p,\s),\rho})$ for $\sigma=0,2$ can be obtained plugging in $R_\sigma$ and $\phi^\sigma$ instead of $R_1$ and $\phi^1$.
Let us remark that for $\sigma=0,2$ we need to reverse the orientation, which corresponds to the choice $A:=(v_1|...|v_m|-\s)$. In the symmetric case we put $\vol(B^0_{(p,\s),\rho})=0$, and recall that $\phi^1=\phi^2=\tfrac{2}{3}\pi$ as well as $C^1=R \tfrac{1}{2} \s=-C^2$.  Accordingly to \eqref{eq.volumedecomposition}, we now need to sum these expansions up to get
\begin{align}
&\rho^{-(m+1)}(V_1)_{(p,\s),\rho}=V_1 +\mathcal{O}(\rho^3)-\omega_m\frac{\rho^2}{6} \Big[ (R_1^{m+3} \tfrac{I_{m+3}(\phi^1)}{m+2} +R_0^{m+3} \tfrac{I_{m+3}(\phi^0)}{m+2})\Sc(p)\label{eq.geodesicvolume1}\\
+ \Ric(\s,\s) \Big( R_1^{m+3}&\big(\tfrac{m+3}{m+2} I_{m+3}(\phi^1)-I_{m+1}(\phi^1) \sin^2(\phi^1) \big) +R_0^{m+3} \big(\tfrac{m+3}{m+2} I_{m+3}(\phi^0)-I_{m+1}(\phi^0) \sin^2(\phi^0) \big) \Big) \Big], \nonumber \\ 
&\rho^{-(m+1)}(V_2)_{(p,\s),\rho}=V_2 -\omega_m\frac{\rho^2}{6} \Big[ (R_2^{m+3} \tfrac{I_{m+3}(\phi^2)}{m+2} -R_0^{m+3} \tfrac{I_{m+3}(\phi^0)}{m+2})\Sc(p) \label{eq.geodesicvolume2}\\
+ \Ric(\s,\s) \Big(R_2^{m+3}&\big(\tfrac{m+3}{m+2} I_{m+3}(\phi^2)-I_{m+1}(\phi^2) \sin^2(\phi^2) \big)  -R_0^{m+3} \big(\tfrac{m+3}{m+2} I_{m+3}(\phi^0)-I_{m+1}(\phi^0) \sin^2(\phi^0) \big) \Big)\Big]. \nonumber
\end{align}
In the symmetric case, we obtain the formulas ($V:=V_1=V_2$)
\begin{equation}\label{eq.geodesicvolumesym} 
\resizebox{0.92\hsize}{!}{ $\rho^{-(m+1)}(V_1)_{(p,\s),\rho}=\rho^{-(m+1)}(V_2)_{(p,\s),\rho}=V-\tfrac{\rho^2}{6}\omega_m \bigg[ R^{m+3} \tfrac{I_{m+3}(\tfrac{2 \pi}{3})}{m+2} \Sc(p)+ R^{m+3}\Big(\tfrac{m+3}{m+2} I_{m+3}(\tfrac{2}{3}\pi)- \tfrac{3}{2}I_{m+1}(\tfrac{2}{3}\pi)\Big)\Ric(\s,\s)\bigg]+\mathcal{O}(\rho^3), $}
\end{equation}
\begin{remark}\label{rem.volumesymmetricbetter}
As a matter of fact, in the symmetric case one can see that the term of order $\rho^3$ in the expansion of $\rho^{-(m+1)}((V_1)_{(p,\s),\rho}+(V_2)_{(p,\s),\rho})$ is zero, since we are integrating an odd function of $\Theta$ (i.e. $\Theta \mapsto \nabla_\Theta \Ric(\Theta,\Theta)$) over the double bubble $\Sigma$, which is invariant under the transformation $\Theta \mapsto -\Theta$.
\end{remark}
\subsection{Areas of Geodesic Double Bubbles}\label{sub.area}
In this subsection we carry out some of the computations which lead to an expansion for the surface area of a geodesic double bubble. The plan is to exploit a good parametrization of a sheet $\Sigma_{(p,\s),\rho}^\sigma$, whose volume element is both explicit and easy to integrate over the hypersurface. To this aim, let us consider a fixed sheet $\Sigma_{(p,\s),\rho}^\sigma=Exp_p(\rho \Sigma^\sigma)$. We will consider the case $\sigma=1$, as before. Set $A:=(v_1 | v_2 | ... | v_m | \s) \in SO(m+1)$ exactly as we did in the previous subsection. Define a parametrization $\Theta : S^{m-1}\times [0,\phi^1] \longrightarrow \hat{\Sigma}$, by $\Theta:=A \circ \bar{\Theta}(\theta,\phi)+C^1$ and $\bar{\Theta}(\theta,\phi):=R_1 (\sin(\phi) \theta,\cos(\phi))$.

Restricting the expansion \eqref{eq.metricnormalcoord} to the point $Exp_p(\rho \Theta(\theta,\phi))$ we get
\begin{equation}
G_{\mu,\nu}=\delta_{\mu, \nu}+\tfrac{\rho^2}{3} \scal{\riem{\Theta}{E_\mu}{\Theta}}{E_\nu} + \mathcal{O}(\rho^3).
\end{equation}
Therefore, calling $\Theta_i$ the tangent vectors of $\Sigma$ induced by $\Theta$, the first fundamental form $g_{(p,\s),\rho}$ of $\Sigma_{(p,\s),\rho}^1$ is given by ($i,j=1,...,m$)
\begin{equation}
(g_{(p,\s),\rho})_{i,j}=\rho^2 (\scal{\Theta_i}{\Theta_j}+\tfrac{\rho^2}{3} \scal{\riem{\Theta}{\Theta_i}{\Theta}}{\Theta_j} + \mathcal{O}(\rho^3)).
\end{equation}
We can rewrite the above expansion  introducing the roto-translated Riemann $(0,4)-$tensor $U$ defined by $U(u_1,u_2,u_3,u_4):=\scal{\riem{A u_1+C^1}{A u_2}{(A u_3+C^1)}}{A u_4}$,  as
\begin{equation}
(g_{(p,\s),\rho})_{i,j}=\rho^2 (\scal{\bar{\Theta}_i}{\bar{\Theta}_j}+\tfrac{\rho^2}{3} U(\bar{\Theta},\bar{\Theta}_i,\bar{\Theta},\bar{\Theta}_j) + \mathcal{O}(\rho^3)).
\end{equation}
The volume element expression follows easily from the Taylor expansion of the determinant, as well as exploiting our explicit parametrization
\begin{equation}
(\sqrt{g})_{(p,\s),\rho}=\rho^{m} R_1^m \sin^{m-1}(\phi) \sqrt{g^{S^{m-1}}} \Big(1+\tfrac{\rho^2}{6} tr(U)(\bar{\Theta},\bar{\Theta})+ \mathcal{O}(\rho^3) \Big).
\end{equation}
By the definition of $U$ we deduce
\begin{equation}
tr(U)(\bar{\Theta},\bar{\Theta})=\Ric(A\bar{\Theta}+C^1,A\bar{\Theta}+C^1)=T(\bar{\Theta},\bar{\Theta})+2 W_\mu(\bar{\Theta})+\Ric(C^1,C^1),
\end{equation}
where, as before, $T=A^T \Ric A$ and $W_\mu=(A^T \Ric)_{\mu, \nu}(C^1)^\nu$.

We can integrate the expansion for $(\sqrt{g})_{(p,\s),\rho}$ over $S^{m-1} \times [0,\phi^1]$ similarly to what we have done in the previous section; after a lengthy calculation, we finally arrive to
\begin{footnotesize}
\begin{equation}\label{eq.areaasymmetric1}
\begin{aligned}
&\area(\Sigma_{(p,\s),\rho}^1)=\int_0^{\phi^1} \int_{S^{m-1}} \rho^{m} R_1^m \sin^{m-1}(\phi) \sqrt{g^{S^{m-1}}} \Big(1+\tfrac{\rho^2}{6}(T(\bar{\Theta},\bar{\Theta})+2 W_\mu(\bar{\Theta})+\Ric(C^1,C^1))+ \mathcal{O}(\rho^3) \Big)\\
&=\rho^m |\Sigma^1|_m+\tfrac{\rho^{m+2}}{6} R_1^{m+2} \omega_m \Big[ I_{m+1}(\phi^1)\Sc(p)+(m \cos^2(\phi^1)I_{m-1}(\phi^1)-\sin^m(\phi^1)\cos(\phi^1))\Ric(\s,\s) \Big]+ \mathcal{O}(\rho^{m+3}).
\end{aligned}
\end{equation}
\end{footnotesize}
To get the expansions for the $m-$dimensional volumes of $\Sigma_{(p,\s),\rho}^0$ (in the asymmetric case) and $\Sigma_{(p,\s),\rho}^2$ (always), one has to substitute $R_1$ and $\phi^1$ with $R_\sigma$ and $\phi^\sigma$ with $\sigma=0, 2$ respectively.

In the symmetric case, we choose a radial parametrization $\Theta:S^{m-1}\times [0,\tfrac{\sqrt{3}}{2}R] \longrightarrow \Sigma^0$ for the sheet $\Sigma_{(p,\s),\rho}^0$, where $\Theta(\theta,r)=A \circ (r \theta)$ and $A$ is as above; we remark that since the double bubble $\Sigma$ is centered, the term of order $\rho^2$ will be purely homogeneous of degree $2$, simplifying a bit the calculation. Arguing as above, we obtain
\begin{equation}
\begin{aligned}
&\area(\Sigma_{(p,\s),\rho}^0)=\rho^m |\Sigma^0|_m+  \tfrac{\rho^{m+2}}{6(m+2)} (\tfrac{\sqrt{3}}{2}R)^{m+2} \omega_m \big(\Sc(p)-2\Ric(\s,\s) \big)+ \mathcal{O}(\rho^{m+3}).
\end{aligned}
\end{equation}
Finally, we write the expansion for the area of the entire geodesic double bubble in the asymmetric case
\begin{footnotesize}
\begin{equation}\label{eq.areaasymmetric}
\begin{aligned}
&\area(\Sigma_{(p,\s),\rho})=\rho^m |\Sigma|_m+\tfrac{\rho^{m+2}}{6}\omega_m \Big(  R_0^{m+2} I_{m+1}(\phi^0)+R_1^{m+2} I_{m+1}(\phi^1)+R_2^{m+2} I_{m+1}(\phi^2)\Big)\Sc(p)\\
&+\tfrac{\rho^{m+2}}{6}\omega_m \Big( R_0^{m+2}(m \cos^2(\phi^0)I_{m-1}(\phi^0)-\sin^m(\phi^0)\cos(\phi^0))+ R_1^{m+2}(m \cos^2(\phi^1)I_{m-1}(\phi^1)-\sin^m(\phi^1)\cos(\phi^1))\\
&+ R_2^{m+2}(m \cos^2(\phi^2)I_{m-1}(\phi^2)-\sin^m(\phi^2)\cos(\phi^2))\Big)\Ric(\s,\s)+ \mathcal{O}(\rho^{m+3}), 
\end{aligned}
\end{equation}
\end{footnotesize}
and in the symmetric case
\begin{equation}\label{eq.areasymmetric}
\resizebox{0.92\hsize}{!}{ $\area(\Sigma_{(p,\s),\rho})=\rho^m |\Sigma|_m+ \tfrac{\rho^{m+2}}{6} \omega_m R^{m+2} \Big( \big( \tfrac{1}{m+2}(\tfrac{\sqrt{3}}{2})^{m+2}+2 I_{m+1}(\tfrac{2}{3} \pi)\big)\Sc(p)+\big(\tfrac{m}{2}I_{m-1}(\tfrac{2}{3}\pi)+\tfrac{2m+1}{2m+4}(\tfrac{\sqrt{3}}{2})^{m}\big)\Ric(\s,\s)\Big)+ \mathcal{O}(\rho^{m+3}) .$}
\end{equation}
\begin{remark}\label{rem.areasymmetricbetter}
The same argument based on symmetry of Remark \ref{rem.volumesymmetricbetter} allows us to substitute $\mathcal{O}(\rho^{m+3})$ with $\mathcal{O}(\rho^{m+4})$.
\end{remark}
\subsection{Approximate Equi-angularity of Geodesic Double Bubbles}\label{sub.equiangularity} 
Concluding the section on geodesic double bubbles, we show that their sheets meet at $120^\circ$-degrees  up to a second order error in the scale $\rho$. Let $\nu_{p,\s}^\sigma$ be the inner conormal vector field defined on $\Gamma_{(p,\s),\rho}$ and pointing towards $\Sigma_{(p,\s),\rho}^\sigma$. Introducing the vectors $\tilde{\nu}_{p,\s}^\sigma:=\Esp( \rho \cdot \nu^\sigma)$, we can use the expansion \eqref{eq.metricnormalcoord} to show that $\nu_{p,\s}^\sigma=\tilde{\nu}_{p,\s}^\sigma/\norm{\tilde{\nu}_{p,\s}^\sigma}_G+\mathcal{O}(\rho^2)$. Consider any point $q \in \Gamma_{(p,\s),\rho}$ and an arbitrary unitary vector $r\in UT_qM$ and put $\hat{r}:=\text{Exp}_q(r)$. Gauss' Lemma guarantees that
\begin{align*}
G(\tilde{\nu}_{p,\s}^0+\tilde{\nu}_{p,\s}^1+\tilde{\nu}_{p,\s}^2,\hat{r})=0.
\end{align*}
By the arbitrarity of $r$ we must have $\tilde{\nu}_{p,\s}^0+\tilde{\nu}_{p,\s}^1+\tilde{\nu}_{p,\s}^2=0$ at $q$. 
Appealing again to the expansion for the metric $G$ given in \eqref{eq.metricnormalcoord}, it is not hard to see that $\norm{\tilde{\nu}_{p,\s}^\sigma}_G=1+\mathcal{O}(\rho^2)$, so we deduce
\begin{equation}\label{eq.geodesicequiangular}
\nu_{p,\s}^0+\nu_{p,\s}^1+\nu_{p,\s}^2(q)=:\mathfrak{e}_{p,\s}(q)=\mathcal{O}(\rho^2).
\end{equation}
One can show that $\mathfrak{e}_{p,\s}$ is a purely geometric function on $\Gamma_{(p,\s),\rho}$ depending smoothly on the metric $G$, the scale $\rho$ and on the point $(p,\s) \in UTM$. However, we will always consider $\mathfrak{e}_{p,\s}$ as a generic term of order $\mathcal{O}(\rho^2)$ in the sequel.
\section{Perturbed Double Bubbles}\label{sec.perturbed}
In this section we will study the geometry of perturbations of geodesic double bubbles, with the scope of computing expansions for the mean curvatures of their sheets; these formulas will play a crucial role in the proof of Theorems \ref{th.mainasymmetric} and \ref{th.mainsymmetric}. Due to the presence 
of a singular set, we will need to combine normal and tangent variations.
\subsection{The Class of Perturbations}\label{sub.class}
Consider a standard double bubble $\Sigma$ in $\mathbb{R}^{m+1}$ with the sheets meeting along a common boundary $\Gamma$. We will use the same conventions as in Section \ref{sec.preliminary}.
Let us define a perturbation $\phi_{w,Y} \colon \Sigma \longrightarrow \mathbb{R}^{m+1}$ by 
\begin{equation}\label{eq.definitionperturbation}
\phi_{w,Y} \colon x \longmapsto x + w(x) N(x) + Y(x), 
\end{equation}
where in each sheet $\Sigma^\sigma$, $w=w_\sigma \in C^{2,\alpha}(\Sigma^\sigma)$, $N=N^\sigma$ and $Y=Y_\sigma \in C^{1,\alpha}(\Sigma^\sigma;  T\Sigma^\sigma)$.
\begin{figure}[h]\label{fig.admissible}
	\begin{minipage}[l]{0.48\linewidth}
\includegraphics[scale=0.55]{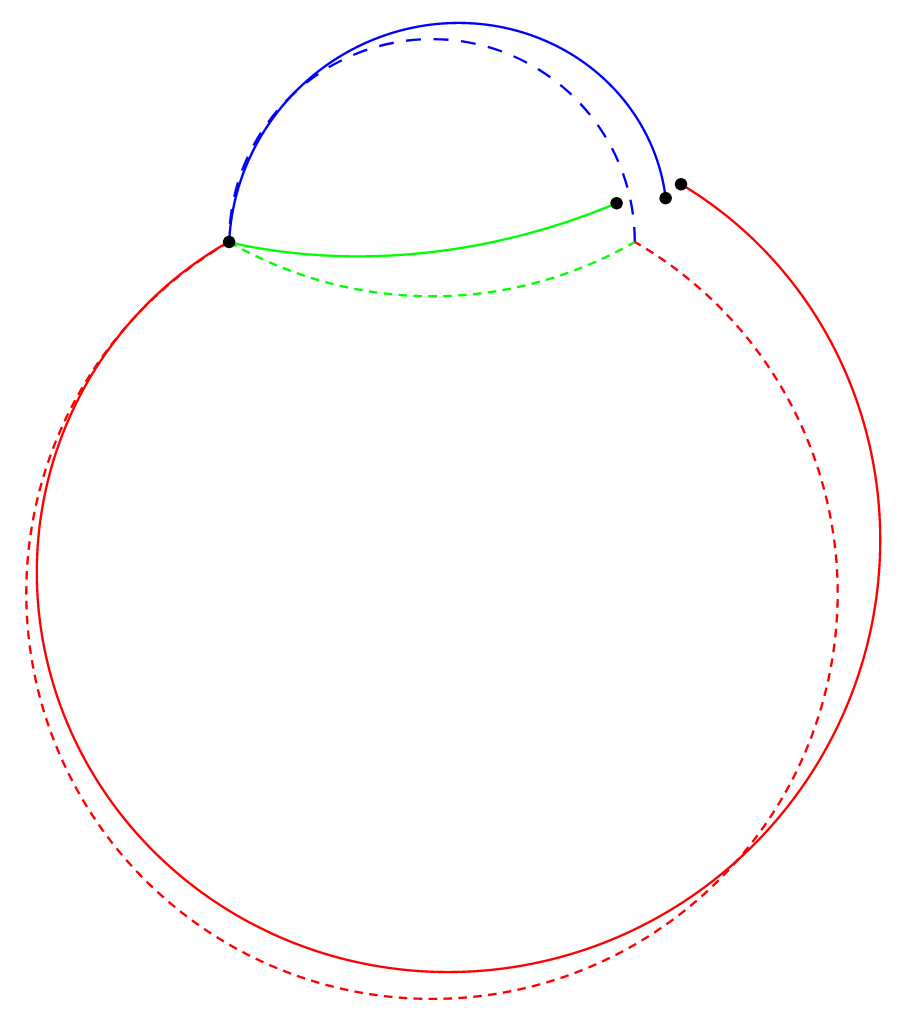}
\end{minipage} 
\begin{minipage}[r]{0.3\linewidth}
\includegraphics[scale=0.21]{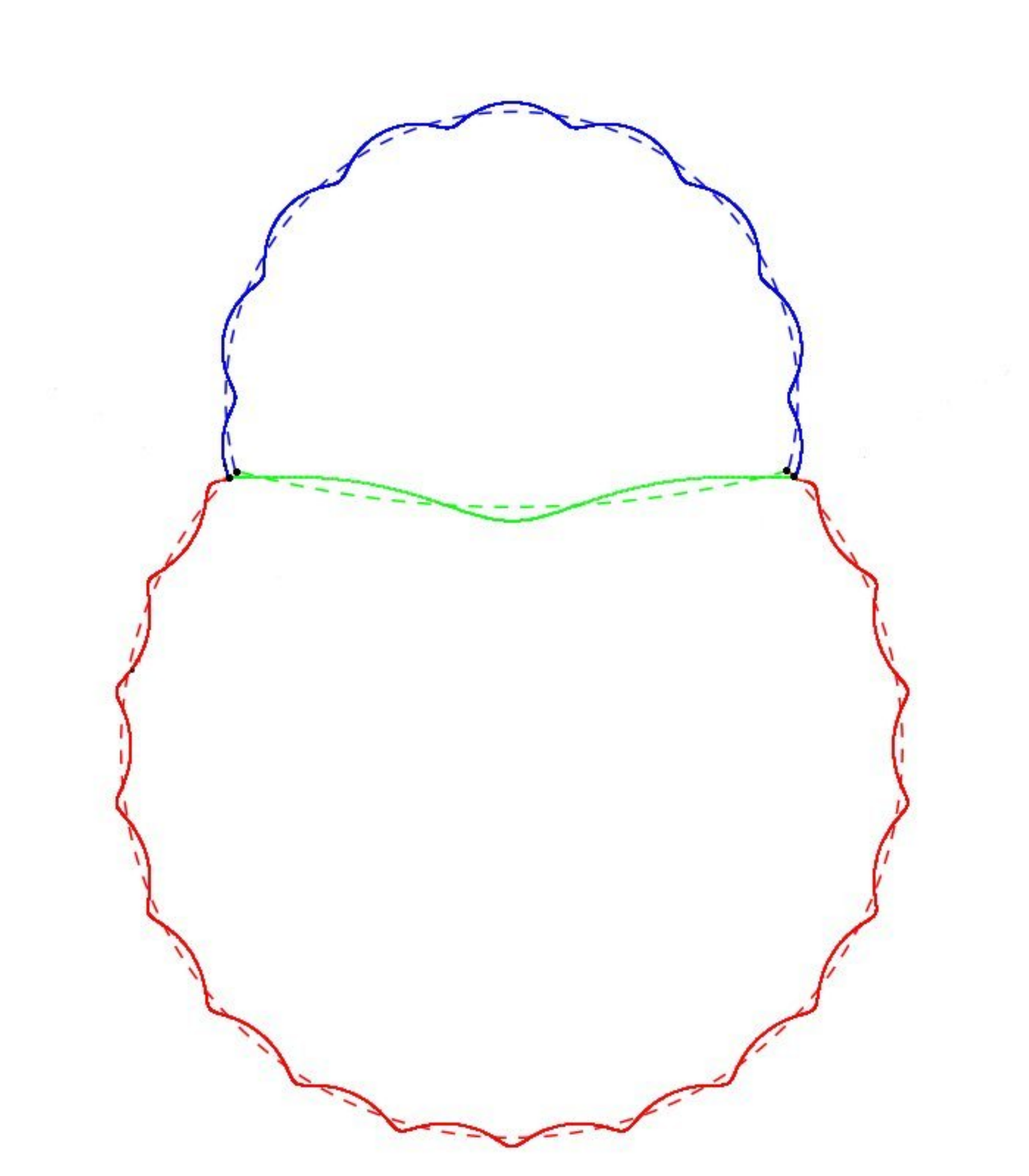}
\end{minipage}
\caption{Non-admissible and admissible perturbations}
\end{figure}
For this function to be well-defined, we have to impose that the images of the points of $\Gamma$ through $\phi_{w,Y}$ do not depend on the sheet used to compute them, that is we need
\begin{equation}
x + w_0(x) N^0(x) + Y_0(x) = x + w_1(x) N^1(x) + Y_1(x) = x + w_2(x) N^2(x) + Y_2(x),
\end{equation}
whenever $x \in \Gamma$. Clearly, this equation is equivalent to the following
\begin{equation}\label{eq.junction}
w_0(x) N^0(x) + Y_0(x) = w_1(x) N^1(x) + Y_1(x) =  w_2(x) N^2(x) + Y_2(x).
\end{equation}
Let us split  $Y_\sigma(x)$ as $Y_\sigma(x)= (Y_\sigma(x))^\Gamma + u_\sigma(x) \nu^\sigma (x)$, where  $\cdot ^\Gamma$ stands for the orthogonal projection of a vector onto the space $T_x \Gamma$ and $\nu^\sigma (x)\in T_x \Sigma^\sigma$ is the inward pointing unit vector normal to $T_x \Gamma$.
Projecting equation \eqref{eq.junction} onto $T_x \Gamma$ we get 
\begin{equation}\label{eq.junctionGamma}
(Y_0(x))^\Gamma = (Y_1(x))^\Gamma = (Y_2(x))^\Gamma.
\end{equation}
Substituting in \eqref{eq.junction}, we obtain
\begin{equation}\label{eq.junctionGammaperp}
w_0(x)N^0(x) + u_0(x)\nu^0(x) = w_1(x)N^1(x) + u_1(x)\nu^1(x) = w_2(x)N^2(x) + u_2(x)\nu^2(x).
\end{equation}
Taking the scalar products of the latter with $\nu^1$ and $N^1$, and using the fact that the sheets of $\Sigma$ meet in an equi-angular way, we deduce (dropping the $x$-dependence)
\begin{equation*}
\begin{cases}
w_1 = \tfrac{1}{2} w_2 - \tfrac{\sqrt{3}}{2} u_2 = \tfrac{1}{2} w_0 + \tfrac{\sqrt{3}}{2} u_0 \\ 
u_1 = - \tfrac{\sqrt{3}}{2} w_2 - \tfrac{1}{2} u_2 = \tfrac{\sqrt{3}}{2} w_0 - \tfrac{1}{2} u_0.
\end{cases}
\end{equation*}
This system is equivalent to the following, where we are writing everything in terms of $w_0$ and $w_2$: 
\begin{align}
w_1 &= w_0 + w_2 \label{eq.boundaryconditionw}\\
u_0=\tfrac{1}{\sqrt{3}}w_0+\tfrac{2}{\sqrt{3}}w_2; \quad u_1&=\tfrac{1}{\sqrt{3}}w_0-\tfrac{1}{\sqrt{3}}w_2; \quad u_2=-\tfrac{2}{\sqrt{3}}w_0-\tfrac{1}{\sqrt{3}}w_2,\label{eq.boundaryconditionu}
\end{align}
which in turn  is equivalent to \eqref{eq.junctionGammaperp}. Furthermore, 
\eqref{eq.junctionGammaperp} and \eqref{eq.junctionGamma} imply \eqref{eq.junction}, so we will use \eqref{eq.junctionGamma}, \eqref{eq.boundaryconditionw} and \eqref{eq.boundaryconditionu} because these are more suitable conditions for the Dirichlet problem we will consider in Section \ref{sec.pseudodb}.
\begin{definition}\label{def.admissibleperturbationdefinition}
We define the class of \emph{admissible pairs} as 
\begin{align*}
\mathcal{C}_{amm} \coloneqq \{ &(w,Y) \in \big( C^{2,\alpha}(\Sigma) \cap C^{0,\alpha}(\overline{\Sigma}) \big) \times \big( C^{1,\alpha}(\Sigma; \mathbb{R}^{m+1})\cap C^{0,\alpha}(\overline{\Sigma}; \mathbb{R}^{m+1}) \cap T\Sigma \big) \mid (w,Y) \ \text{verifies} \\
 &\eqref{eq.junctionGamma}, \eqref{eq.boundaryconditionw}, \eqref{eq.boundaryconditionu}, \text{ and }\norm{w}_{C^{2,\alpha}},\norm{Y}_{C^{1,\alpha}} \le \delta \},
\end{align*}
for some $\delta$ small enough. Moreover,  define the class of \emph{admissible perturbations} as $\{ \phi_{w,Y} \mid (w,Y) \in \mathcal{C}_{amm} \}$.
\end{definition}
Here $\delta$ is chosen small enough so that the image $\phi_{w,Y}(\Sigma)$ is homeomorphic to $\Sigma$.
When imposing a perturbed double bubble to have sheets with (almost) constant mean curvature, we will need to solve a Dirichlet problem, whose leading term is given by a second-order elliptic operator involving only the normal components $w_\sigma$'s of the perturbation considered, the so-called \emph{Jacobi operator}. It is then clear that such a problem is under-determined (after having fixed a tangential component $Y$), when solving for a general normal component $w$ of an admissible pair $(w,Y)$, and we need to impose other two conditions on the $w_\sigma$'s. A good choice for these two conditions (in Euclidean ambient) is inspired by the analysis in \cite{dim}, and consists in imposing the perturbation $\phi_{w,Y}$ to preserve the condition on the sheets to meet in an equi-angular way, so the inner conormal vectors $\nu^\sigma_{w,Y}$ should verify
\begin{equation}\label{eq.euclideanequiangularity}
\nu^0_{w,Y}+\nu^1_{w,Y}+\nu^2_{w,Y}=0.
\end{equation}
It can be shown that this equation is non-linear and involves only $w, Y$ and \emph{their first derivatives} with respect to the reference inner normal vectors $\nu^\sigma$'s. This subtlety reveals crucial in solving the boundary value problem we will deal with, as well as in gaining the regularity of the solution. However, it will be more convenient for us to impose an equi-angularity condition \emph{directly} for the perturbed surface in the manifold, see equation \eqref{eq.perturbedequiangularity}; indeed, this choice will allow us to get rid of a troublesome boundary term appearing in the proof of one of the key steps in the proof of the main Theorem \ref{th.nondeg}, namely Proposition \ref{prop.criticalpointCMC} we refer the reader to Section \ref{sec.existence} for more details. Appealing to the results in \cite{dim}, we are able to produce a unique solution to the associated problem with \emph{linearized equi-angularity} condition, and then see our original problem, under the condition \eqref{eq.perturbedequiangularity}, as a perturbation of this linearized version. In order to introduce this linearization, we set
\begin{equation}\label{eq.definitionq}
q_0:=\tfrac{1}{\sqrt{3}}(H_1+H_2), \quad q_1:=\tfrac{1}{\sqrt{3}}(H_0-H_2), \quad q_2:=-\tfrac{1}{\sqrt{3}}(H_1+H_0).
\end{equation}
Projecting the equation on the vectors $N^0$ and $N^2$, it can then be shown (see \cite{hut}) that the condition described above is equivalent to the system
\begin{equation}\label{eq.linearisedequiangularity} 
\begin{cases}
\tfrac{\partial w_0}{\partial \nu^0} +q_0 w_0 +\tfrac{\partial w_1}{\partial \nu^1} +q_1 w_1 =0 \quad \text{on} \ \Gamma;\\
\tfrac{\partial w_1}{\partial \nu^1} +q_1 w_1 +\tfrac{\partial w_2}{\partial \nu^2} +q_2 w_2 =0 \quad \text{on} \ \Gamma.
\end{cases}
\end{equation}
\begin{figure}[h]\label{fig.equiangularity}
\begin{minipage}[l]{0.48\linewidth}
\hspace{-0.5cm}
\includegraphics[scale=3]{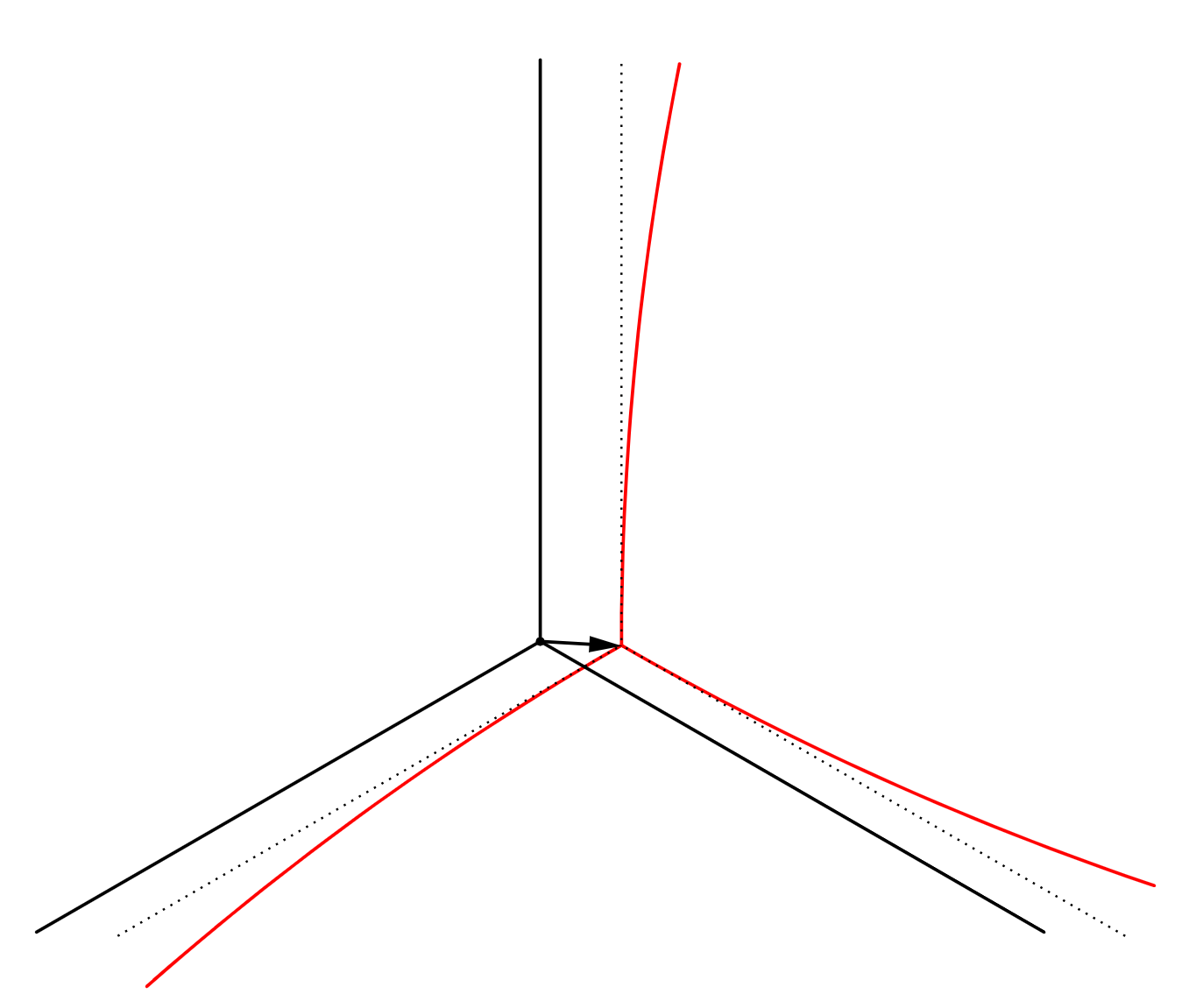}
\end{minipage} \hspace{0.5cm}
\begin{minipage}[r]{0.3\linewidth}
\hspace{-2.3cm}
\includegraphics[scale=0.2]{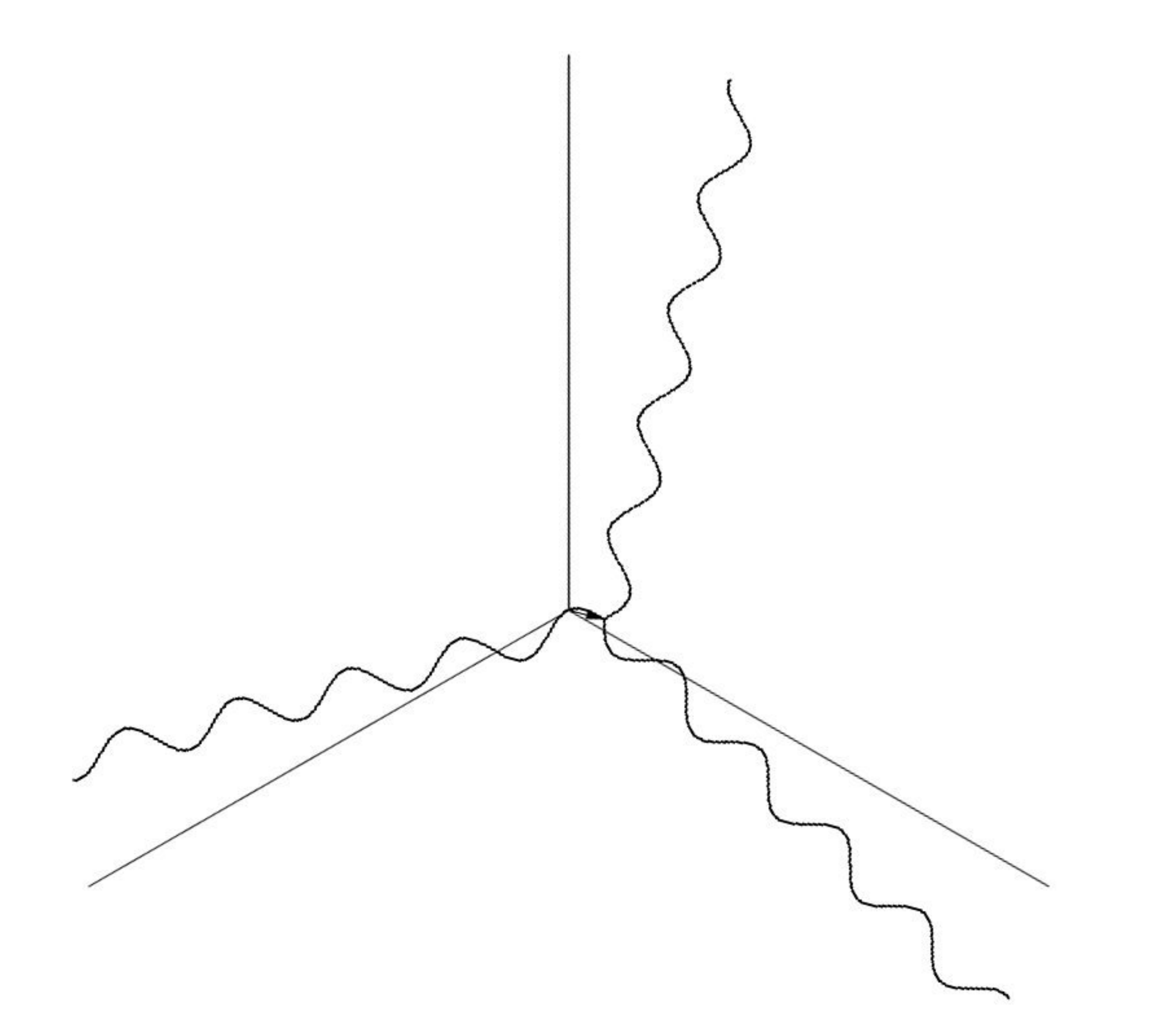}
\end{minipage}
\caption{The linearized equi-angularity condition and the equi-angularity condition}
\end{figure}
 We remark that these equations do not depend on the tangential component $Y$. System \eqref{eq.linearisedequiangularity} furnishes two additional Robin-type conditions which makes our  \emph{linearized} problem for the normal components $w_\sigma$'s well-determined, i.e. we get existence and uniqueness of a solution to our elliptic problem, \emph{for any fixed tangential component} $Y$. Therefore, we will need to fix some suitable $Y$, verifying the boundary conditions \eqref{eq.junctionGamma} and \eqref{eq.boundaryconditionu} and also the norm bound $\norm{Y}_{C^{1,\alpha}} \le \delta$. We have great freedom of choice for this $Y$ but, as already mentioned in the introduction, we will be naturally led to impose additional constraints. We refer the reader to Section \ref{sec.pseudodb} for further details.
\begin{remark}
It is worth mentioning that we are not assuming the perturbations just introduced to preserve the volumes enclosed by the double bubble in consideration; as already remarked in \cite{dim}, the absence of this constraint reflects a stronger stability valid for standard double bubbles proved in \cite{slob}, on which the results in \cite{dim} and those presented here ultimately rely upon.
\end{remark}
\subsection{First Fundamental Form for Perturbed Geodesic Double Bubbles}\label{sub.first}
Throughout the rest of this section, we will consider a given point $(p,\s) \in UTM$ and a centered standard double bubble $\Sigma \subseteq T_p M$ aligned along it.  Recall the orthonormal basis of $T_pM$ given by the vectors $E_\mu$'s from the previous section. Let us denote by $\Theta=\Theta^\mu E_\mu:U' \longrightarrow \Sigma$ a parametrization of $\Sigma$, defined on an open set $U' \subset \R^m$, and let $z=(z^i)_{i=1,...,m} \in U'$.

For a function $f$ (or a vector field $V$) defined on $\Sigma$, we write the first-order derivative $f_i(z) = \partial_{z^i}f(z) $ (resp. $(V_i(z))^\mu = \partial_{z^i}(V(z)^\mu)$), and similarly for higher-order derivative. In particular, the parametrization $\Theta$ induces coordinate vector fields $\Theta_i:=\Theta_i^\mu E_\mu$ on $\Sigma$.

As in \cite{pac}, it is useful to adopt the following convention: any expression of the form $L (w,Y)$ denotes a linear combination of the functions $w$ and $Y$, together with their derivatives up to order $2$ for $w$ and $1$ for $Y$ with respect to the vector fields $\Theta_i$ . 
The coefficients of $L$ depend smoothly on $\rho$ and $(p,\s)$ and, for all $k \in \mathbb{N}$, there exists a constant $c_k > 0$ independent of $\rho \in (0,1)$ and $(p,\s) \in UTM$ such that
\begin{equation}\label{eq.propertyL}
\norm{L(w,Y)}_{C^{k,\alpha}(\Sigma)} \le c_k \norm{(w,Y)}_{C^{k+2,\alpha}(\Sigma)\times C^{k+1,\alpha}(\Sigma;\R^{m+1})}.
\end{equation}
Similarly, given $a \in \mathbb{N}$, any expression of the form $Q^{(a)} (w,Y)$ denotes a nonlinear operator in the functions $w$ and $Y$, together with their derivatives with respect to the vector fields $\Theta_i$ up to order $2$ and $1$ respectively. The coefficients of the Taylor expansion of $Q^{(a)} (w,Y)$ in powers of $w$ and $Y$ and their partial derivatives depend smoothly on $\rho$ and $(p,\s)$ and, given $k \in \mathbb{N}$, there exists a constant $d_k > 0$ independent of $\rho \in (0, 1)$, $(p,\s) \in UTM$ such that $Q^{(a)} (0,\cdot) = Q^{(a)}(\cdot,0) = 0$ and
\begin{equation}\label{eq.propertyQ}
\begin{aligned}
&\norm{Q^{(a)}(w_2,Y_2) - Q^{(a)}(w_1,Y_1)}_{C^{k,\alpha}(\Sigma)} \le d_k\norm{(w_2-w_1,Y_2-Y_1)}_{C^{k+2,\alpha}(\Sigma) \times C^{k+1,\alpha}(\Sigma;\R^{m+1})} \\
& \quad \quad \quad \quad \quad \cdot ( \norm{(w_1,Y_1)}_{C^{k+2,\alpha}(\Sigma)\times C^{k+1,\alpha}(\Sigma;\R^{m+1})} +\norm{(w_2,Y_2)}_{C^{k+2,\alpha}(\Sigma)\times C^{k+1,\alpha}(\Sigma;\R^{m+1}) }  )^{a-1},
\end{aligned}
\end{equation}
provided that $\norm{(w_i,Y_i)}_{C^{k+2,\alpha}(\Sigma) \times C^{k+1,\alpha}(\Sigma;\R^{m+1})}$ are small enough. For our purposes, we will consider mostly nonlinearities of the type $\rho^k Q^{(2)}$ for $k \ge 0$, so we will frequently absorb terms of the form $\rho^j Q^{(a)}$ in it when $a \ge 2$ and $j \ge k$ (there is no mistake in doing it, because $\rho^j Q^{(a)}$ easily verify the inequality defining $ \rho^k Q^{(2)}$ in that case). Similarly, we will absorb terms of the form $\rho^j L$ into $\rho^k L$ whenever $j \ge k$. A typical example of $Q^{(a)}$ can be a homogeneous polynomial of degree $a$, in $w,Y$ and their derivatives up to order $2$ and $1$ respectively (e.g. $w_jRic(Y,Y_i)= Q^{(3)}(w,Y)$). 
\begin{remark}\label{rem.Ogeometric}
We remark that the absorption convention on $L(w,Y)$ and $Q^{(a)}(w,Y)$ just introduced, guarantees that any term of the form $\mathcal{O}(\rho^k)$ is \textit{purely geometric}, in the sense that it is independent of $w$ and $Y$. This is important for our fixed point argument in Section 5.
\end{remark}

From now on, we fix an admissible pair $(w,Y)\in \mathcal{C}_{amm}$ and its induced perturbation $\phi_{w,Y}$. Let us focus on {the interior part of one sheet} $\Sigma^\sigma$, and drop the index $\sigma$; then $\Sigma$ is either a spherical cap included in a sphere $S(C,R)$ or a disk. In the first case, we deduce that $T_\Theta \Sigma = \Span (\Theta_i) = (\Theta-C)^\perp$ where $i=1,...,m$ and that the second fundamental form of it with respect to the inward normal is given by $\tfrac{1}{R}Id$. Since $Y$ is tangent to $\Sigma$ we must have $\scal{C-\Theta}{Y}=0$. Moreover, we can decompose the derivatives of the vector field $Y$ as $Y_i=\nabla^{\Sigma}_i Y+\tfrac{\scal{Y}{\Theta_i}}{R}N$.

In the second case, we have $T_\Theta \Sigma = \Span (\Theta_i) = \s^\perp$, in particular the normal vector $N=\s$ does not depend on $\Theta$. The second fundamental form is trivial, and since $Y$ is tangent we must have $Y_i=\nabla^{\Sigma}_i Y$.
\vspace{0.5cm}

{Let us recall the coordinate vector fields $X_\mu$'s defined as in Section \ref{sec.geodesic}.} Defining the vector fields $\mathcal{C}:=C^\mu X_\mu$, $\mathcal{N}:=N^\mu X_\mu$, $\Upsilon:=\Theta^\mu X_\mu$, $\Upsilon_i:=\Theta_i^\mu X_\mu$, $\Omega \coloneqq Y^\mu X_\mu$, $\Omega_i \coloneqq Y_i^\mu X_\mu$, we will use their coordinates with respect to the local frame $X_\mu$ together with the formula in Proposition \ref{prop.metricnormalcoord} to deduce the needed expansions.

We want to compute an expansion for the first fundamental form $\mathring{g}$ of the perturbed hypersurface $\Sigma_{(p,\s),\rho}(w,Y):= \Esp (\rho \cdot \phi_{w,Y}(\Sigma))=: M_\rho (\Sigma)$, in terms of the first fundamental form $g$ of $\Sigma$. Notice that $M_\rho(\Theta(z))=\Esp(\rho(\Theta+w(z)N(\Theta(z))+Y(\Theta(z))))$; we will omit the $z-$dependence in what follows.
Firstly, consider the case in which $\Sigma$ is a spherical cap; in order to find a basis of the tangent space of $\Sigma_{(p,\s),\rho}(w,Y)$ at a point $q$, we take the push-forward of the basis $(\Theta_i)_{i=1,...,m}$ of $T_{M^{-1}_{\rho}(q)} \Sigma$ through $M_\rho$: 
\begin{align}\label{eq.tangentvectors}
 Z_i(q) = \rho(\Upsilon_i (1-\tfrac{w}{R}) + \tfrac{w_i}{R}(\mathcal{C}-\Upsilon) + \Omega_i)(M^{-1}_{\rho}(q)).
\end{align}
Using equation \eqref{eq.metricnormalcoord} with $\Xi = \Esp^{-1}(M_\rho (\Theta(z))) =\rho ((1- \tfrac{w}{R}) \Theta+\tfrac{w}{R}C+Y)$, we obtain the following expansion for the ambient metric $G$ at an arbitrary point $q=\Esp(\Xi)=M_\rho(z) \in \Sigma_{(p,\s),\rho}(w,Y)$
\begin{align*}
G(X_\mu,X_\nu) =&\delta_{\mu,\nu} + \tfrac{1}{3}\scal{\riem{\rho ((1- \tfrac{w}{R}) \Theta +\tfrac{w}{R}C+Y)}{E_\mu}{\rho ((1- \tfrac{w}{R}) \Theta+\tfrac{w}{R}C+Y)}}{E_\nu} \\
&+ \tfrac{1}{6} \scal{\nabla_\Xi \riem{\Xi}{E_\mu}{\Xi}}{E_\nu} \big|_{\Xi = \rho ((1- \tfrac{w}{R}) \Theta +\tfrac{w}{R}C+Y)} {+\rho^4L(w,Y)+\rho^4 Q^{(2)}(w,Y)}+ \mathcal{O}(\rho^4) .
\end{align*}
For convenience of the reader, we separate the addenda by their order in $\rho$:
\begin{align*}
  0^{th} \ \text{order:}& \quad \delta_{\mu,\nu}; \quad { 1^{st} \ \text{order:} \quad 0}; \\ 
  2^{nd} \ \text{order:}& \quad \tfrac{1}{3}\scal{\riem{\rho ((1- \tfrac{w}{R}) \Theta+\tfrac{w}{R}C +Y)}{E_\mu}{\rho ((1- \tfrac{w}{R}) \Theta+\tfrac{w}{R}C+Y)}}{E_\nu} \\
  =& \tfrac{1}{3} \rho^2 \scal{\riem{(1- \tfrac{w}{R}) \Theta +\tfrac{w}{R}C+Y}{E_\mu}{((1- \tfrac{w}{R}) \Theta+\tfrac{w}{R}C+Y)}}{E_\nu} \\
  =& \tfrac{1}{3} \rho^2 \Big[ (1-\tfrac{w}{R})^2 \scal{\riem{\Theta}{E_\mu}{\Theta}}{E_\nu} + \tfrac{w^2}{R^2} \scal{\riem{C}{E_\mu}{C}}{E_\nu} + \scal{\riem{Y}{E_\mu}{Y}}{E_\nu} \\ 
  &+  \tfrac{w}{R}(1-\tfrac{w}{R}) \big[ \scal{\riem{\Theta}{E_\mu}{C}}{E_\nu} + \scal{\riem{C}{E_\mu}{\Theta}}{E_\nu} \big] + (1-\tfrac{w}{R}) \big[ \scal{\riem{\Theta}{E_\mu}{Y}}{E_\nu} \\ 
  &+ \scal{\riem{Y}{E_\mu}{\Theta}}{E_\nu} \big] + \tfrac{w}{R} \big[ \scal{\riem{C}{E_\mu}{Y}}{E_\nu} + \scal{\riem{Y}{E_\mu}{C}}{E_\nu} \big] \Big].
\end{align*}
For the third order, we have: 
\begin{align*}
  &3^{rd} \ \text{order:} \quad \tfrac{1}{6} \scal{\nabla_\Xi \riem{\Xi}{E_\mu}{\Xi}}{E_\nu} \big|_{\Xi = \rho ((1- \tfrac{w}{R}) \Theta +\tfrac{w}{R}C +Y)} \\
  &= \tfrac{1}{6} \rho^3 \scal{\nabla_\Theta \riem{\Theta}{E_\mu}{\Theta}}{E_\nu} + \rho^3 L(w,Y) {+\rho^3 Q^{(2)}(w,Y)}.
\end{align*}
The expansion for $\mathring{g}$ can now be easily deduced  computing $\mathring{g}_{i,j} \coloneqq G(Z_i,Z_j)$
\begin{align*}
  &\rho^{-2} \mathring{g}_{i,j} = \rho^{-2} G(Z_i,Z_j) \\
  &= G(\Upsilon_i (1-\tfrac{w}{R})+\tfrac{w_i}{R}(\mathcal{C}-\Upsilon)+\Omega_i,\Upsilon_j (1-\tfrac{w}{R})+\tfrac{w_j}{R}(\mathcal{C}-\Upsilon)+\Omega_j) \\
  &= (\Theta_i(1-\tfrac{w}{R})+\tfrac{w_i}{R}(C-\Theta)+Y_i)^\mu (\Theta_j(1-\tfrac{w}{R})+\tfrac{w_j}{R}(C-\Theta)+Y_j)^\nu G(X_\mu,X_\nu).
\end{align*}
As before, we separate the addenda by their order in $\rho$:
\begin{align*}
  0^{th} \ \text{order:}& \quad (\Theta_i(1-\tfrac{w}{R})+\tfrac{w_i}{R}(C-\Theta)+Y_i )^\mu (\Theta_j(1-\tfrac{w}{R})+\tfrac{w_j}{R}(C-\Theta)+Y_j)^\nu \delta_{\mu,\nu} \\
  =& (1-\tfrac{w}{R})^2 \scal{\Theta_i}{\Theta_j} + w_i w_j + \scal{Y_i}{Y_j}+(1-\tfrac{w}{R})\scal{Y_i}{\Theta_j} \\
  &+(1-\tfrac{w}{R})\scal{\Theta_i}{Y_j} + w_i \scal{N}{Y_j}+ w_j \scal{Y_i}{N}\\
  =& (1-\tfrac{w}{R})^2 g_{i,j}+ w_i w_j + g(\nabla_i^{\Sigma} Y,\nabla_j^{\Sigma} Y)+\tfrac{1}{R^2}g(\Theta_i,Y)g(Y,\Theta_j) \\
  &+(1-\tfrac{w}{R}) g(\nabla_i^{\Sigma} Y,\Theta_j)+(1-\tfrac{w}{R})g(\Theta_i,\nabla_j^{\Sigma} Y) + \tfrac{w_i}{R} g(Y,\Theta_j)+ \tfrac{w_j}{R} g(\Theta_i,Y).
\end{align*}
{Once again, the first order is null.} For the second order, using the previous $L-Q^{(a)}$'s formalism {(and absorbing $Q^{(a)}$ into the quadratic term $Q^{(2)}$) }we obtain
\begin{align*}
 2^{nd} \ \text{order:}&= \tfrac{1}{3} \rho^2 (1-\tfrac{w}{R})^2 \biggl\{ (1-\tfrac{w}{R})^2 \Big[ \scal{(\riem{\Theta}{\Theta_i}{\Theta}}{\Theta_j} + \tfrac{w_j}{R}\scal{\riem{\Theta}{\Theta_i}{\Theta}}{C} \\
 &+ \tfrac{w_i}{R}\scal{\riem{\Theta}{C}{\Theta}}{\Theta_j}+ \scal{\riem{\Theta}{\Theta_i}{\Theta}}{Y_j} + \scal{\riem{\Theta}{Y_i}{\Theta}}{\Theta_j} \Big]\\
&+\tfrac{w}{R} \Big[  \scal{\riem{\Theta}{\Theta_i}{C}}{\Theta_j}+ \scal{\riem{C}{\Theta_i}{\Theta}}{\Theta_j} \Big]+\Big[ \scal{\riem{\Theta}{\Theta_i}{Y}}{\Theta_j} \\
 &+ \scal{\riem{Y}{\Theta_i}{\Theta}}{\Theta_j} \Big] + Q^{(2)}(w,Y) \biggr\}= \tfrac{1}{3} \rho^2 \biggl\{ (1-\tfrac{w}{R})^4 \scal{\riem{\Theta}{\Theta_i}{\Theta}}{\Theta_j} \\
&+  \Big[ \tfrac{w_j}{R}\scal{\riem{\Theta}{\Theta_i}{\Theta}}{C} + \tfrac{w_i}{R}\scal{\riem{\Theta}{C}{\Theta}}{\Theta_j} \Big]+ \Big[ \scal{\riem{\Theta}{\Theta_i}{\Theta}}{Y_j} \\
&+ \scal{\riem{\Theta}{Y_i}{\Theta}}{\Theta_j} \Big]+\tfrac{w}{R} \Big[  \scal{\riem{\Theta}{\Theta_i}{C}}{\Theta_j}+ \scal{\riem{C}{\Theta_i}{\Theta}}{\Theta_j} \Big]\\
&+\Big[ \scal{\riem{\Theta}{\Theta_i}{Y}}{\Theta_j} + \scal{\riem{Y}{\Theta_i}{\Theta}}{\Theta_j} \Big] + Q^{(2)}(w,Y) \biggr\}.
\end{align*}
The $3^{rd}$-order term is simpler than the previous one, because we seek for less information:
\begin{align*}
  3^{rd} \ \text{order:}&= \tfrac{1}{6} \rho^3 \scal{\nabla_\Theta \riem{\Theta}{\Theta_i}{\Theta}}{\Theta_j} + \rho^3 L(w,Y) {+\rho^3 Q^{(2)}(w,Y)}.
\end{align*}
Finally, we arrive at the following analogue to Lemma $2.1$ in \cite{pac}.
\begin{proposition}[Perturbed First Fundamental Form - Spherical Case]\label{prop.perturbedfirst}
The expansion for the first fundamental form of the perturbed spherical cap is given by
 \begin{align}
 \rho^{-2} (1-\tfrac{w}{R})^{-2} &\mathring{g}_{i,j} = g_{i,j}+g(\nabla^\Sigma_i Y,\Theta_j)+g(\Theta_i,\nabla^\Sigma_j Y)+ \mathcal{Q}(w,Y) + \tfrac{1}{3} \rho^2 \biggl\{ (1-\tfrac{w}{R})^4 \scal{\riem{\Theta}{\Theta_i}{\Theta}}{\Theta_j} \nonumber \\
&+  \Big[ \tfrac{w_j}{R}\scal{\riem{\Theta}{\Theta_i}{\Theta}}{C} + \tfrac{w_i}{R}\scal{\riem{\Theta}{C}{\Theta}}{\Theta_j} \Big]+ \Big[ \scal{\riem{\Theta}{\Theta_i}{\Theta}}{Y_j} + \scal{\riem{\Theta}{Y_i}{\Theta}}{\Theta_j} \Big] \label{eq.perturbedfirst} \\
&+\tfrac{w}{R} \Big[  \scal{\riem{\Theta}{\Theta_i}{C}}{\Theta_j}+ \scal{\riem{C}{\Theta_i}{\Theta}}{\Theta_j} \Big]+\Big[ \scal{\riem{\Theta}{\Theta_i}{Y}}{\Theta_j} + \scal{\riem{Y}{\Theta_i}{\Theta}}{\Theta_j} \Big] \biggr\} \nonumber \\
&+ \tfrac{1}{6} \rho^3 \scal{\nabla_\Theta \riem{\Theta}{\Theta_i}{\Theta}}{\Theta_j} + \rho^3 L(w,Y) + \rho^2 Q^{(2)}(w,Y)+ \mathcal{O}(\rho^4), \nonumber
  \end{align}
where $\mathcal{Q}(w,Y) = (1-\tfrac{w}{R})^{-2} \big[ w_i w_j + g(\nabla_i^{\Sigma} Y,\nabla_j^{\Sigma} Y)+\tfrac{1}{R^2}g(\Theta_i,Y)g(Y,\Theta_j)+\tfrac{w}{R} g(\nabla_i^{\Sigma} Y,\Theta_j)+\tfrac{w}{R}g(\Theta_i,\nabla_j^{\Sigma} Y) + \tfrac{w_i}{R} g(Y,\Theta_j)+ \tfrac{w_j}{R} g(\Theta_i,Y) \big]= Q^{(2)} (w,Y)$.
\end{proposition}
Let us remark that, in the case $Y=0$ and $C=0$, we recover the  expansion in  Lemma $2.1$ of \cite{pac}, except for the rescaling $w \leadsto \tfrac{w}{R}$.
\vspace{0.3cm}

For later purposes, we give an explicit expansion for the inverse of the metric 
\begin{align}
  \mathring{g}^{i,j} =&  \rho^{-2} (1-\tfrac{w}{R})^{-2} \bigg\{ g^{i,j} + g^{i,k}\Big[ -g(\nabla_k^{\Sigma} Y,\Theta_m)-g(\Theta_k,\nabla_m^{\Sigma}Y)+Q^{(2)}(w,Y)-\tfrac{1}{3} \rho^2 \scal{\riem{\Theta}{\Theta_k}{\Theta}}{\Theta_m} \nonumber\\
  &+\rho^2 L(w,Y)+ \mathcal{O}(\rho^3) \Big]g^{m,j} \bigg\}= \rho^{-2} (1-\tfrac{w}{R})^{-2} g^{i,j} -\rho^{-2} g^{i,k}(g(\nabla_k^{\Sigma} Y,\Theta_m)+g(\Theta_k,\nabla_m^{\Sigma}Y))g^{m,j} \label{eq.inversemetric}\\
  &-\tfrac{1}{3}g^{i,k} \scal{\riem{\Theta}{\Theta_k}{\Theta}}{\Theta_m}g^{m,j}+L(w,Y)+\rho^{-2}Q^{(2)}(w,Y)+\mathcal{O}(\rho).\nonumber
\end{align}
In  case we are perturbing a flat disk, the $i-$th coordinate tangent vector induced by a parametrization $\Theta$ is $Z_i:= \rho(\Upsilon_i+w_i \mathcal{N}+\Omega_i)$, so  analogous computations yield the expansion for the metric stated in the following proposition.
\begin{proposition}[Perturbed First Fundamental Form - Disk Case]
The expansion for the first fundamental form of the perturbed disk is given by
\begin{align}
 \rho^{-2} \mathring{g}_{i,j} =&\delta_{i,j}+g(\nabla_i^{\Sigma} Y,\Theta_j)+g(\Theta_i,\nabla_j^{\Sigma} Y)+w_i w_j +g(\nabla_i^{\Sigma}Y,\nabla_j^{\Sigma} Y)+\tfrac{1}{3} \rho^2 \bigg\{ \scal{\riem{\Theta}{\Theta_i}{\Theta}}{\Theta_j} \nonumber \\
&+w [\scal{\riem{N}{\Theta_i}{\Theta}}{\Theta_j}+\scal{\riem{\Theta}{\Theta_i}{N}}{\Theta_j}]+[\scal{\riem{\Theta}{\Theta_i}{Y}}{\Theta_j}+\scal{\riem{Y}{\Theta_i}{\Theta}}{\Theta_j}] \label{eq.perturbedfirstdisk}\\
&+[w_i \scal{\riem{\Theta}{N}{\Theta}}{\Theta_j}+w_j \scal{\riem{\Theta}{\Theta_i}{\Theta}}{N}]+[\scal{\riem{\Theta}{\Theta_i}{\Theta}}{\nabla_j^{\Sigma} Y}+\scal{\riem{\Theta}{\nabla_i^{\Sigma} Y}{\Theta}}{\Theta_j}]\bigg\}\nonumber \\
&+ \tfrac{1}{6} \rho^3 \scal{\nabla_\Theta \riem{\Theta}{\Theta_i}{\Theta}}{\Theta_j} +\rho^3 L(w,Y)+ \rho^2 Q^{(2)}(w,Y)+ \mathcal{O}(\rho^4).\nonumber
  \end{align}
\end{proposition}
As already done above, we explicit an expansion for the metric's inverse:
\begin{equation}
\begin{aligned}\label{eq.inversemetricdisk}
  \mathring{g}^{i,j} =& \rho^{-2} \delta^{i,j} -\rho^{-2} \delta^{i,k}(g(\nabla_k^{\Sigma} Y,\Theta_m)+g(\Theta_k,\nabla_m^{\Sigma}Y))\delta^{m,j}-\tfrac{1}{3}\delta^{i,k} \scal{\riem{\Theta}{\Theta_k}{\Theta}}{\Theta_m}\delta^{m,j}\\
  &+L(w,Y)+\rho^{-2}Q^{(2)}(w,Y)+\mathcal{O}(\rho).
\end{aligned}
\end{equation}
\subsection{Second Fundamental Form for Perturbed Geodesic Double Bubbles}\label{sub.second}
In order to obtain the mean curvature of the different sheets of the perturbed double bubble considered above, we first of all need to compute expansions for their second fundamental forms, which we are about to present. The calculation is quite lengthy and more involved compared to the analogous one in \cite{pac}, due to both the presence of the center $C$ of the double bubble and  the tangential component $Y$. If on one hand the presence of the center cannot be neglected since there cannot be a common one for the three sheets at once, on the other hand its location  will contribute only by an amount of order $\rho^2$ or higher; this should not be surprising, as the second fundamental form of a hypersurface in $\R^{m+1}$ is invariant under translation. Moreover, we will see how the tangential component $Y$ will appear in the expansion of the second fundamental form as a Lie derivative (\emph{at its lowest order}), and therefore how the mean curvature will not depend (again, at the lowest orders) on it.
\vspace{0.3cm}

Let us begin with some expansions related to the unit normal vector field  $\mathring{\mathcal{N}}$ to the hypersurface $\Sigma_{(p,\s),\rho}(w,Y)$. We firstly focus our attention to the spherical cap case, postponing the flat disk one. Since this hypersurface is a small perturbation of a spherical cap, one expects the inner normal vector to be close to the inner vector toward the center. This heuristic justifies the choice of searching a normal vector of the form $\mathring{ \mathcal{M}} \coloneqq \mathcal{C}-\Upsilon + a^j Z_j$ and then renormalize it, with the (small) coefficients $a^j$ chosen so that $G(\mathring{ \mathcal{M}},Z_j)=0$ for all $j=1,...,m$. Let us explicit what condition the $a^j$'s have to verify at a point $q=M_\rho(z) \in \Sigma_{(p,\s),\rho}(w,Y)$:
\begin{align}
  0=&G(Z_i, \mathring{ \mathcal{M}}) = G(Z_i,\mathcal{C}-\Upsilon + a^j Z_j)= G(Z_i,\mathcal{C}-\Upsilon) + a^j \mathring{g}_{i,j} \nonumber\\
  =& a^j \mathring{g}_{i,j}+ \rho G((1-\tfrac{w}{R}) \Upsilon_i+\tfrac{w_i}{R}(\mathcal{C}-\Upsilon)+\Omega_i, \mathcal{C}-\Upsilon) \nonumber\\
  =& a^j \mathring{g}_{i,j} + \rho (1-\tfrac{w}{R}) \Big( \tfrac{\rho^2}{3} \scal{\riem{\Theta}{\Theta_i}{\Theta}}{C}+\rho^2L(w,Y){+\rho^2 Q^{(2)}(w,Y)}+ \mathcal{O}(\rho^3) \Big) \nonumber\\
&+ \rho \tfrac{w_i}{R} \Big( R^2 {+\rho^2 L(w,Y)+\rho^2 Q^{(2)}(w,Y)}+ \mathcal{O}(\rho^2) \Big)+ \rho \Big(g(\Theta_i,Y) +\rho^2 L(w,Y) {+\rho^2 Q^{(2)}(w,Y)}+ \mathcal{O}(\rho^3) \Big) \nonumber\\
&=a^j \mathring{g}_{i,j}+\rho w_i R+\rho g(\Theta_i,Y) +\tfrac{\rho^3}{3}  \scal{\riem{\Theta}{\Theta_i}{\Theta}}{C} \label{eq.Morthogonal}+\rho^3 L(w,Y) {+\rho^3 Q^{(2)}(w,Y)}+ \mathcal{O}(\rho^4), 
\end{align}
where we have used Proposition \ref{prop.metricnormalcoord}. Appealing again to the same proposition, and to the orthogonality just imposed, we evaluate the squared norm of $\mathring{ \mathcal{M}}$:
\begin{align*}
  G(\mathring{ \mathcal{M}},\mathring{ \mathcal{M}})=& G(\mathcal{C}-\Upsilon + a^i Z_i, \mathring{ \mathcal{M}}) = G(\mathcal{C}-\Upsilon,\mathcal{C}-\Upsilon) + G(\mathcal{C}-\Upsilon,Z_j) a^j \\
  =& R^2 + \tfrac{\rho^2}{3} \scal{\riem{\Theta}{C}{\Theta}}{C} + \rho^2 L(w,Y){+\rho^2 Q^{(2)}(w,Y)}+ \mathcal{O}(\rho^3) + \Big[\rho (w_j R + g(Y,\Theta_j)) \\
  &+ \tfrac{\rho^3}{3} \scal{\riem{\Theta}{C}{\Theta}}{\Theta_j} + \rho^3 L(w,Y){+\rho^3 Q^{(2)}(w,Y)} + \mathcal{O}(\rho^4) \Big] a^j.
\end{align*}
Thanks to \eqref{eq.inversemetric} and \eqref{eq.Morthogonal} we can approximate the coefficients $a^j$ as follows
\begin{equation}
\begin{aligned}
a^j &= \mathring{g}^{i,j}[\rho L(w,Y) {+\rho^3 Q^{(2)}(w,Y)}+ \mathcal{O}(\rho^3)]= [\rho L(w,Y) {+\rho^3 Q^{(2)}(w,Y)}+ \mathcal{O}(\rho^3)]\cdot \\
&\cdot [\mathcal{O}(\rho^{-2}){+\rho^{-2}L(w,Y)+\rho^{-2} Q^{(2)}(w,Y)}]= \rho^{-1} L(w,Y){+\rho^{-1} Q^{(2)}(w,Y)}+\mathcal{O}(\rho), 
\end{aligned}
\end{equation}
to get
\begin{align}
  G(\mathring{ \mathcal{M}},\mathring{ \mathcal{M}})=&R^2 + \tfrac{\rho^2}{3} \scal{\riem{\Theta}{C}{\Theta}}{C}+\rho^2 L(w,Y){+\rho^2 Q^{(2)}(w,Y)}+ \mathcal{O}(\rho^3) \nonumber \\
&+ \Big[\rho(w_j R+g(Y,\Theta_j))+ \tfrac{\rho^3}{3} \scal{\riem{\Theta}{C}{\Theta}}{\Theta_j} + \rho^3 L(w,Y) {+\rho^3 Q^{(2)}(w,Y)}+ \mathcal{O}(\rho^4) \Big] \cdot \nonumber\\
& (\rho^{-1} L(w,Y){+\rho^{-1} Q^{(2)}(w,Y)}+\mathcal{O}(\rho)) \label{eq.Mnorm}\\
  =& R^2 + \tfrac{\rho^2}{3} \scal{\riem{\Theta}{C}{\Theta}}{C} + \rho^2 L(w,Y)+ Q^{(2)}(w,Y) + \mathcal{O}(\rho^3).\nonumber 
\end{align}
We deduce the following expansion
\begin{equation}
 G(\mathring{ \mathcal{M}},\mathring{ \mathcal{M}})^{-\tfrac{1}{2}} =R^{-1}\Big[1-\tfrac{\rho^2}{6R^2} \scal{\riem{\Theta}{C}{\Theta}}{C} + \rho^2 L(w,Y)+ Q^{(2)}(w,Y) + \mathcal{O}(\rho^3) \Big].
\end{equation}
Thus we define the unitary normal vector $\mathring{\mathcal{N}} \coloneqq  G(\mathring{ \mathcal{M}},\mathring{ \mathcal{M}})^{-\tfrac{1}{2}} \mathring{ \mathcal{M}}$. Similar computations yield, in the flat disk case, the following equation for the coefficients $a^j$, where $\mathring{ \mathcal{M}}=\mathcal{N}+a^j Z_j$:
\begin{equation}
  0=a^j \mathring{g}_{i,j}+\rho w_i+\tfrac{\rho^3}{3} \scal{\riem{\Theta}{\Theta_i}{\Theta}}{N}+ \rho^3 L(w,Y) {+\rho^3 Q^{(2)}(w,Y)}+ \mathcal{O}(\rho^4),
\end{equation}
as well as the one for the inverse of the norm
\begin{equation}
G(\mathring{ \mathcal{M}},\mathring{ \mathcal{M}})^{-\tfrac{1}{2}} =1-\tfrac{\rho^2}{6} \scal{\riem{\Theta}{N}{\Theta}}{N}+ \rho^2 L(w,Y)+ Q^{(2)}(w,Y) + \mathcal{O}(\rho^3) .
\end{equation}
Once again we put $\mathring{\mathcal{N}} \coloneqq  G(\mathring{ \mathcal{M}},\mathring{ \mathcal{M}})^{-\tfrac{1}{2}} \mathring{ \mathcal{M}}$.
\vspace{0.3cm}

An expansion for the second fundamental form of the {perturbed hyper-surface is presented in the following theorem, for the case of a spherical cap contained in $S(C,R)$.}
\begin{theorem}[Perturbed Second Fundamental Form - Spherical Case]\label{th.perturbedsecond}
The second fundamental form $\mathring{h}_{i,j}$ of $\Sigma_{p,\rho}(w,Y)$ has the following expansion:
\begin{equation}\label{eq.perturbedsecond}
   \begin{aligned}
      \mathring{h}_{i,j} =&\tfrac{\rho}{R}(1-\tfrac{w}{R}) g_{i,j}+\tfrac{\rho}{R}( g(\Theta_i,\nabla_j^{\Sigma} Y)+g(\nabla_i^{\Sigma} Y,\Theta_j))+ \rho Hess^{\Sigma} w+\tfrac{\rho^3}{6R}\mathscr{S}_{i,j}(\Theta,C)\\
&+\rho^3 L(w,Y)+\rho Q^{(2)}(w,Y)+\mathcal{O}(\rho^4),
     \end{aligned}
\end{equation}
where $\mathscr{S}_{i,j}$ is the symmetric $(0,2)-$tensor given by
\begin{equation}\label{eq.definitionS}
\resizebox{0.92\hsize}{!}{ $\mathscr{S}_{i,j}(\Theta,C):=4 \scal{\riem{\Theta}{\Theta_i}{\Theta}}{\Theta_j}-2\scal{\riem{C}{\Theta_i}{\Theta}}{\Theta_j}-2\scal{\riem{\Theta}{\Theta_i}{C}}{\Theta_j}+ R^{-2}\scal{\riem{\Theta}{C}{\Theta}}{C} g_{i,j}.$ }
\end{equation}
\end{theorem}
  \begin{proof}
    Define $\mathring{k}_{i,j} \coloneqq - G(\nabla_{Z_i} \mathring{ \mathcal{M}}, Z_j) = G(\nabla_{Z_i} (\Upsilon-\mathcal{C}), Z_j) - G(\nabla_{Z_i} (a^k Z_k) , Z_j)$. We treat the two summands separately. 
    For the second summand we recall \eqref{eq.Morthogonal}: 
    \begin{align*}
      &a^k G(Z_k,Z_j) = a^k \mathring{g}_{k,j} = -\rho w_j R - \rho g(Y, \Theta_j) -\tfrac{\rho^3}{3} \scal{\riem{\Theta}{C}{\Theta}}{\Theta_j} +\rho^3 L(w,Y){+\rho^3 Q^{(2)}(w,Y)}+ \mathcal{O}(\rho^4);
    \end{align*}
    applying to it $\partial_{z^i}$ and using the compatibility of the metric:
    \begin{align*}
      &G(\nabla_{Z_i} (a^k Z_k) , Z_j) + a^k G(Z_k,\nabla_{Z_i} Z_j) = \partial_{z^i} \Big( G( a^k Z_k,Z_j) \Big) =-\rho w_{i,j} R - \rho g(\nabla_{i,j}^\Sigma \Theta,Y)-\rho g(\nabla_i^{\Sigma} Y,\Theta_j) \\
&-\tfrac{\rho^3}{3} \mathscr{T}_{i,j}(\Theta,C)+\rho^3 L_p(w,Y) {+\rho^3 Q^{(2)}(w,Y)}+ \mathcal{O}(\rho^4),
    \end{align*}
    where $\mathscr{T}_{i,j}(\Theta,C) \coloneqq \scal{\riem{\Theta_i}{C}{\Theta}}{\Theta_j} + \scal{\riem{\Theta}{C}{\Theta_i}}{\Theta_j} + \scal{\riem{\Theta}{C}{\Theta}}{ \Theta_{i,j}}$. This holds if and only if (remark that the ambient connection extends the tangential one)
    \begin{align*}
      &-G(\nabla_{Z_i} (a^k Z_k) , Z_j) = a^k G(Z_k,\nabla_{Z_i} Z_j)+ \rho w_{i,j} R +\rho g(\nabla_{i,j}^\Sigma \Theta,Y)+\rho g(\nabla_i^{\Sigma} Y,\Theta_j)+ \tfrac{\rho^3}{3} \mathscr{T}_{i,j}(\Theta,C)\\
      &+ \rho^3 L(w,Y) {+\rho^3 Q^{(2)}(w,Y)}+ \mathcal{O}(\rho^4) =  a^k G(Z_k, \mathring{\Gamma}^l_{i,j} Z_l)+ \rho w_{i,j} R +\rho g(\nabla_{i,j}^\Sigma \Theta,Y)+\rho g(\nabla_i^{\Sigma} Y,\Theta_j)\\
&+ \tfrac{\rho^3}{3} \mathscr{T}_{i,j}(\Theta,C)+ \rho^3 L(w,Y) {+\rho^3 Q^{(2)}(w,Y)}+ \mathcal{O}(\rho^4) = a^k \mathring{g}_{k,l} \mathring{\Gamma}^l_{i,j}+ \rho w_{i,j} R +\rho g(\nabla_{i,j}^\Sigma \Theta,Y)+\rho g(\nabla_i^{\Sigma} Y,\Theta_j)\\
&+ \tfrac{\rho^3}{3} \mathscr{T}_{i,j}(\Theta,C)+ \rho^3 L(w,Y) {+\rho^3 Q^{(2)}(w,Y)}+ \mathcal{O}(\rho^4) = \Big( -\rho w_l R -\rho g(\Theta_l,Y) - \tfrac{\rho^3}{3} \scal{\riem{\Theta}{C}{\Theta}}{\Theta_l} \\
&+\rho^3 L(w,Y){+\rho^3 Q^{(2)}(w,Y)}+\mathcal{O}(\rho^4) \Big) \mathring{\Gamma}^l_{i,j}+ \rho w_{i,j} R +\rho g(\nabla_{i,j}^\Sigma \Theta,Y)+\rho g(\nabla_i^{\Sigma} Y,\Theta_j)+ \tfrac{\rho^3}{3} \mathscr{T}_{i,j}(\Theta,C)\\
&+ \rho^3 L(w,Y) {+\rho^3 Q^{(2)}(w,Y)}+ \mathcal{O}(\rho^4) =\rho g(\nabla_i^{\Sigma} Y,\Theta_j)+ \rho R \Big( w_{i,j} -  \mathring{\Gamma}^l_{i,j} w_l \Big)+ \rho 	\scal{\nabla_{i,j}^\Sigma \Theta-\mathring{\Gamma}^l_{i,j} \Theta_l}{Y} \\
&+ \tfrac{\rho^3}{3} \mathscr{T}'_{i,j}(\Theta,C)+ \rho^3 L_p(w,Y) {+\rho^3 Q^{(2)}(w,Y)}+ \mathcal{O}(\rho^4), 
 \end{align*}
where $\mathscr{T}_{i,j}'(\Theta,C)$ is given by
\begin{equation}
\scal{\riem{\Theta_i}{C}{\Theta}}{\Theta_j} + \scal{\riem{\Theta}{C}{\Theta_i}}{\Theta_j} + \scal{\riem{\Theta}{C}{\Theta}}{ \Theta_{i,j}} - \scal{\riem{\Theta}{C}{\Theta}}{\Theta_l} \mathring{\Gamma}^l_{i,j}.
\end{equation}
For the first summand, following the argument in \cite{pac}, one can connect it to a ``radial" derivative of the metric. In their case $\partial_\rho \mathring{g}_{i,j}$ furnishes the desired result, since the sphere under interest is centered at the origin. We instead choose to consider a direction radial with respect to the center $C$, since this will substantially simplify the calculations: the approaches are equivalent. Let us define 
    \begin{equation*}
      \tilde G (s,z) \coloneqq \Esp \bigg(\rho \Big(C+s\big(\Theta(z)-C+w(z)N(\Theta(z)) + Y(\Theta(z)) \big)\Big) \bigg),
    \end{equation*}
for $s \in (0,1]$, and notice that $\Sigma_{(p,\s),\rho}(w,Y)= \tilde G (1,\Sigma)$. The radial vector we consider is
    \begin{align*}
      Z_0 & \coloneqq d \tilde G(s,z) [\partial_s]  \mid_{s=1}= d(\Esp)[\Theta-C+ wN +Y]= d(\Esp)[ \rho ((1- \tfrac{w}{R}) (\Theta-C) +Y)] \\
      &= \rho[(1-\tfrac{w}{R}) (\Upsilon-\mathcal{C}) + \Omega].
    \end{align*}
Notice that
\begin{equation}
[Z_0,Z_i]=\tilde G_* ([\partial_s,\partial_{z^i}])=0.
\end{equation}
    By the previous analysis: 
    \begin{align*}
      \mathring{k}_{i,j} =& G(\nabla_{Z_i} (\Upsilon- \mathcal{C}), Z_j) +\rho g(\nabla_i^{\Sigma} Y,\Theta_j)+ \rho R \Big( w_{i,j} -  \mathring{\Gamma}^l_{i,j} w_l \Big)+ \rho \scal{\nabla_{i,j}^\Sigma \Theta-\mathring{\Gamma}^l_{i,j} \Theta_l}{Y} + \tfrac{\rho^3}{3} \mathscr{T}'_{i,j}(\Theta,C)\\
&+ \rho^3 L_p(w,Y) {+\rho^3 Q^{(2)}(w,Y)}+ \mathcal{O}(\rho^4).
    \end{align*}
     It is easy to prove that $\mathring{k}_{i,j} = \mathring{k}_{j,i}$, since  we need only the orthogonality of $\mathring{ \mathcal{M}}$ and $T \Sigma_{p,\rho}(w,Y)$; moreover, the Christoffel symbols are symmetric (the induced connection is Levi-Civita), hence the first summand on the right-hand side of the latter equation must compensate the presence of the asymmetric summands.  In order to see how, we compute (we use $[Z_i,Z_j]=0$ for every $i,j=0,...,m$):
\begin{small}
\begin{align*}
&G(\nabla_{Z_i} Z_0, Z_j)- G(Z_i,\nabla_{Z_j} Z_0)=\partial_i (G(Z_0,Z_j))-\partial_j (G(Z_i,Z_0))-G(Z_0,\nabla_{Z_i} Z_j)+G(\nabla_{Z_j}Z_i, Z_0) \\
&=\rho \partial_i (G((1-\tfrac{w}{R}) (\Upsilon-\mathcal{C}) + \Omega,Z_j))-\rho \partial_j (G(Z_i,(1- \tfrac{w}{R}) (\Upsilon-\mathcal{C}) + \Omega)\\
&=\rho^2 \partial_i [(1-\tfrac{w}{R})\scal{\Theta-C}{(1-\tfrac{w}{R}) \Theta_j+\tfrac{w_j}{R}(C-\Theta)+Y_j}+g(Y,(1-\tfrac{w}{R})\Theta_j)+Q^{(2)}(w,Y)-\tfrac{\rho^2}{3} \scal{\riem{\Theta}{C}{\Theta}}{\Theta_j}]\\
&-\rho^2 \partial_j [(1-\tfrac{w}{R})\scal{(1-\tfrac{w}{R}) \Theta_i+\tfrac{w_i}{R}(C-\Theta)+Y_i}{\Theta-C}+g((1-\tfrac{w}{R})\Theta_i,Y)+Q^{(2)}(w,Y)-\tfrac{\rho^3}{3} \scal{\riem{\Theta}{\Theta_i}{\Theta}}{C}]\\
&+\rho^4 L(w,Y)+\mathcal{O}(\rho^5)=\tfrac{\rho^4}{3}(\mathscr{T}_{j,i}(\Theta,C)-\mathscr{T}_{i,j}(\Theta,C))+\rho^4 L(w,Y)+\rho^2Q^{(2)}(w,Y)+\mathcal{O}(\rho^5).
\end{align*}
\end{small}
We now need to compute the derivative in $s$ of the metric. In order to do so, we first need to explicit the first fundamental form of $\tilde G(s,\Sigma)$, and then differentiate at $s=1$. With a calculation along the lines to the one we developed in the previous subsection, we obtain
\begin{small}
\begin{align*}
\mathring{g}_{i,j}(s)=& \rho^2 s^2  [(1-\tfrac{w}{R})^2 g_{i,j}+g(\nabla_i^{\Sigma} Y,\Theta_j)+g(\Theta_i,\nabla_j^{\Sigma} Y)+\tfrac{\rho^2}{3}(s^2 \scal{\riem{\Theta}{\Theta_i}{\Theta}}{\Theta_j}+s(1-s)\scal{\riem{\Theta}{\Theta_i}{C}}{\Theta_j}\\
&+s(1-s)\scal{\riem{C}{\Theta_i}{\Theta}}{\Theta_j}+(1-s)^2\scal{\riem{C}{\Theta_i}{C}}{\Theta_j})]+\rho^4 L(w,Y)+\rho^2Q^{(2)}(w,Y)+\mathcal{O}(\rho^5).
\end{align*}
\end{small}
Taking the $s-$derivative and evaluating at $s=1$ we get
\begin{small}
\begin{align*}
\partial_s \mathring{g}_{i,j}(s)\mid_{s=1}=&2 \rho^2  [(1-\tfrac{w}{R})^2 g_{i,j}+g(\nabla_i^{\Sigma} Y,\Theta_j)+g(\Theta_i,\nabla_j^{\Sigma} Y)+\tfrac{\rho^2}{3}\scal{\riem{\Theta}{\Theta_i}{\Theta}}{\Theta_j}]+\rho^4 L(w,Y)+\rho^2Q^{(2)}(w,Y)\\
&+\mathcal{O}(\rho^5)+\tfrac{\rho^4}{3}[2 \scal{\riem{\Theta}{\Theta_i}{\Theta}}{\Theta_j}-\scal{\riem{\Theta}{\Theta_i}{C}}{\Theta_j}-\scal{\riem{C}{\Theta_i}{\Theta}}{\Theta_j}].
\end{align*}
\end{small}
 We can now rewrite the derivative in $s$ of the metric as follows
\begin{small}
     \begin{align*}
       \partial_s \mathring{g}_{i,j}(s)\mid_{s=1} =& \partial_s \big( G(Z_i,Z_j) \big)(s)\mid_{s=1}= G(\nabla_{Z_0} Z_i , Z_j) + G(Z_i , \nabla_{Z_0} Z_j) =G(\nabla_{Z_i} Z_0 , Z_j) + G(Z_i , \nabla_{Z_j} Z_0)\\
&=2G(\nabla_{Z_i} Z_0 , Z_j) -\tfrac{\rho^4}{3}(\mathscr{T}_{j,i}(\Theta,C)-\mathscr{T}_{i,j}(\Theta,C))+\rho^4 L(w,Y)+\rho^2 Q^{(2)}(w,Y)+\mathcal{O}(\rho^5)\\
&=2\partial_i (G(Z_0 , Z_j) )-2G(Z_0,\nabla_{Z_i} Z_j) -\tfrac{\rho^4}{3}(\mathscr{T}_{j,i}(\Theta,C)-\mathscr{T}_{i,j}(\Theta,C))+\rho^4 L(w,Y)+\rho^2 Q^{(2)}(w,Y)\\
&+\mathcal{O}(\rho^5)=2\rho (1-\tfrac{w}{R}) G(\nabla_{Z_i} (\Upsilon- \mathcal{C}), Z_j) -2 \rho \tfrac{w_i}{R}G(\Upsilon-\mathcal{C},Z_j)+2 \rho \partial_i(G(\Omega,Z_j))\\
&-2 \rho G(\Omega,\nabla_{Z_i} Z_j)-\tfrac{\rho^4}{3}(\mathscr{T}_{j,i}(\Theta,C)-\mathscr{T}_{i,j}(\Theta,C))+\rho^4 L(w,Y)+\rho^2 Q^{(2)}(w,Y)+\mathcal{O}(\rho^5)\\
&=2\rho (1-\tfrac{w}{R}) G(\nabla_{Z_i} (\Upsilon- \mathcal{C}), Z_j)+2 \rho^2 \partial_i(g(Y,\Theta_j))-2 \rho^2 \mathring{\Gamma}^l_{i,j} g(Y,\Theta_l)-\tfrac{\rho^4}{3}(\mathscr{T}_{j,i}(\Theta,C)-\mathscr{T}_{i,j}(\Theta,C))\\
&+\rho^4 L(w,Y)+\rho^2 Q^{(2)}(w,Y)+\mathcal{O}(\rho^5)=2\rho (1-\tfrac{w}{R}) G(\nabla_{Z_i} (\Upsilon- \mathcal{C}), Z_j)+2\rho^2g(\nabla_i^{\Sigma} Y,\Theta_j)\\
&+2 \rho^2 \scal{Y}{\Theta_{i,j}^{\Sigma}-\mathring{\Gamma}^l_{i,j} \Theta_l}-\tfrac{\rho^4}{3}(\mathscr{T}_{j,i}(\Theta,C)-\mathscr{T}_{i,j}(\Theta,C))+\rho^4 L(w,Y)+\rho^2 Q^{(2)}(w,Y)+\mathcal{O}(\rho^5).
     \end{align*}
\end{small}
We can now rearrange this identity and use the expansion we got before for the derivative of the metric to obtain
\begin{small}
\begin{align*}
&G(\nabla_{Z_i} (\Upsilon- \mathcal{C}), Z_j)=\tfrac{1}{2}\rho^{-1}(1-\tfrac{w}{R})^{-1} \bigg\{ \partial_s \mathring{g}_{i,j}(s)\mid_{s=1}-[2\rho^2g(\nabla_i^{\Sigma} Y,\Theta_j)+2 \rho^2 \scal{Y}{\Theta_{i,j}^{\Sigma}- \mathring{\Gamma}^l_{i,j} \Theta_l}\\
&-\tfrac{\rho^4}{3}(\mathscr{T}_{j,i}(\Theta,C)-\mathscr{T}_{i,j}(\Theta,C))+\rho^4 L(w,Y)+\rho^2 Q^{(2)}(w,Y)+\mathcal{O}(\rho^5)] \bigg\}\\
&=\tfrac{1}{2}\rho^{-1}(1-\tfrac{w}{R})^{-1}\bigg\{ 2 \rho^2(1-\tfrac{w}{R})^2 g_{i,j}+2 \rho^2 g(\Theta_i,\nabla_j^{\Sigma} Y)+\tfrac{\rho^4}{3}(4 \scal{\riem{\Theta}{\Theta_i}{\Theta}}{\Theta_j}-\scal{\riem{\Theta}{\Theta_i}{C}}{\Theta_j}\\
&-\scal{\riem{C}{\Theta_i}{\Theta}}{\Theta_j}) -2 \rho^2 \scal{Y}{\Theta_{i,j}^{\Sigma}-\mathring{\Gamma}^l_{i,j} \Theta_l}+\tfrac{\rho^4}{3}(\mathscr{T}_{j,i}(\Theta,C)-\mathscr{T}_{i,j}(\Theta,C))+\rho^4 L(w,Y)+\rho^2 Q^{(2)}(w,Y)+\mathcal{O}(\rho^5) \bigg\}\\
&= \rho(1-\tfrac{w}{R}) g_{i,j}+\rho g(\Theta_i,\nabla_j^{\Sigma} Y)- \rho \scal{Y}{\Theta_{i,j}^{\Sigma}-\mathring{\Gamma}^l_{i,j} \Theta_l}+\tfrac{\rho^3}{6}[4 \scal{\riem{\Theta}{\Theta_i}{\Theta}}{\Theta_j}-\scal{\riem{\Theta}{\Theta_i}{C}}{\Theta_j}\\
&-\scal{\riem{C}{\Theta_i}{\Theta}}{\Theta_j}+ \mathscr{T}_{j,i}(\Theta,C)-\mathscr{T}_{i,j}(\Theta,C) ]+\rho^3 L(w,Y)+\rho Q^{(2)}(w,Y)+\mathcal{O}(\rho^4).
\end{align*}
\end{small}
Resuming, we get
\begin{align*}
&\mathring{k}_{i,j}= G(\nabla_{Z_i} (\Upsilon-\mathcal{C}), Z_j) - G(\nabla_{Z_i} (a^k Z_k) , Z_j)=\rho(1-\tfrac{w}{R}) g_{i,j}+\rho g(\Theta_i,\nabla_j^{\Sigma} Y)- \rho \scal{Y}{\Theta_{i,j}^{\Sigma}-\mathring{\Gamma}^l_{i,j} \Theta_l}\\
&+\tfrac{\rho^3}{6}[4 \scal{\riem{\Theta}{\Theta_i}{\Theta}}{\Theta_j}-\scal{\riem{\Theta}{\Theta_i}{C}}{\Theta_j}-\scal{\riem{C}{\Theta_i}{\Theta}}{\Theta_j}+ \mathscr{T}_{j,i}(\Theta,C)-\mathscr{T}_{i,j}(\Theta,C) ]\\
&+\rho g(\nabla_i^{\Sigma} Y,\Theta_j)+ \rho R \Big( w_{i,j} -  \mathring{\Gamma}^l_{i,j} w_l \Big)+ \rho \scal{\nabla_{i,j}^\Sigma \Theta-\mathring{\Gamma}^l_{i,j} \Theta_l}{Y} + \tfrac{\rho^3}{3} [\mathscr{T}_{i,j}(\Theta,C)-\scal{\riem{\Theta}{C}{\Theta}}{\Theta_l}\mathring{\Gamma}^l_{i,j}]\\
&+\rho^3 L(w,Y)+\rho Q^{(2)}(w,Y)+\mathcal{O}(\rho^4)=\rho(1-\tfrac{w}{R}) g_{i,j}+\rho g(\Theta_i,\nabla_j^{\Sigma} Y)+\rho g(\nabla_i^{\Sigma} Y,\Theta_j)+ \rho R \Big( w_{i,j} -  \mathring{\Gamma}^l_{i,j} w_l \Big)\\
&+\tfrac{\rho^3}{6}[4 \scal{\riem{\Theta}{\Theta_i}{\Theta}}{\Theta_j}-\scal{\riem{\Theta}{\Theta_i}{C}}{\Theta_j}-\scal{\riem{C}{\Theta_i}{\Theta}}{\Theta_j}+ \mathscr{T}_{i,j}(\Theta,C)+\mathscr{T}_{j,i}(\Theta,C) \\
&-2\scal{\riem{\Theta}{C}{\Theta}}{\Theta_l}\mathring{\Gamma}^l_{i,j}]+\rho^3 L(w,Y)+\rho Q^{(2)}(w,Y)+\mathcal{O}(\rho^4).
\end{align*}
Employing the definition of $\mathscr{T}_{i,j}(\Theta,C)$ and the basic symmetries of the Riemann tensor, we get that the term in front of $\tfrac{\rho^3}{6}$ is given by
\begin{equation*}
\mathscr{T}^{''}_{i,j}(\Theta,C):=4 \scal{\riem{\Theta}{\Theta_i}{\Theta}}{\Theta_j}-2\scal{\riem{C}{\Theta_i}{\Theta}}{\Theta_j}-2\scal{\riem{\Theta}{\Theta_i}{C}}{\Theta_j}+2\scal{\riem{\Theta}{C}{\Theta}}{\Theta_{i,j}^{\Sigma}-\Theta_l\mathring{\Gamma}^l_{i,j}}.
\end{equation*}
We now  have to connect the Christoffel's symbols. We get: 
\begin{small}
     \begin{align*}
       &\mathring{\Gamma}^k_{i,j}=\tfrac{1}{2} \mathring{g}^{k,m} \Big( \tfrac{\partial \mathring{g}_{m,i}}{\partial z_j} + \tfrac{\partial \mathring{g}_{m,j}}{\partial z_i} - \tfrac{\partial \mathring{g}_{i,j}}{\partial z_m} \Big)= \tfrac{1}{2} \mathring{g}^{k,m} \rho^2 \biggl( \tfrac{\partial}{\partial z_j} \Big( (1-\tfrac{w}{R})^2 \Big[ g(\Theta_m,\Theta_i) +L(w,Y){+Q^{(2)}(w,Y)} \Big]+ \mathcal{O}(\rho^2) \Big) \\
       &+ \tfrac{\partial}{\partial z_i} \Big( (1-\tfrac{w}{R})^2 \Big[ g(\Theta_m,\Theta_j) +L(w,Y) {+Q^{(2)}(w,Y)}\Big]+ \mathcal{O}(\rho^2) \Big)\\
&-\tfrac{\partial}{\partial z_m} \Big( (1-\tfrac{w}{R})^2 \Big[ g(\Theta_i,\Theta_j) +L(w,Y) {+Q^{(2)}(w,Y)}\Big]+ \mathcal{O}(\rho^2) \Big) \biggr) \\
&=\tfrac{1}{2} \mathring{g}^{k,m} \rho^2 (1-\tfrac{w}{R})^2 \Big[ \Big( \tfrac{\partial}{\partial z_j} (g(\Theta_m,\Theta_i)) +\tfrac{\partial}{\partial z_i} (g(\Theta_m,\Theta_j))-\tfrac{\partial}{\partial z_m} (g(\Theta_i,\Theta_j)) \Big) + L(w,Y) {+Q^{(2)}(w,Y)}+ \mathcal{O}(\rho^2) \Big].
     \end{align*}
\end{small}
     Observe that $\mathring{g}^{k,m} = (1-\tfrac{w}{R})^{-2} \rho^{-2} g^{k,m} + \rho^{-2}L_p(w,Y) {+\rho^{-2} Q^{(2)}(w,Y)}+ \mathcal{O}(1)$ by \eqref{eq.inversemetric}, so we easily deduce
\begin{small}
     \begin{align*}
\mathring{\Gamma}^k_{i,j}=&\tfrac{1}{2} [ (1-\tfrac{w}{R})^{-2} \rho^{-2} g^{k,m} + \rho^{-2}L_p(w,Y) {+\rho^{-2} Q^{(2)}(w,Y)}+ \mathcal{O}(1)] \rho^2 (1-\tfrac{w}{R})^2 \Big[ \Big( \tfrac{\partial}{\partial z_j} (g(\Theta_m,\Theta_i)) +\tfrac{\partial}{\partial z_i} (g(\Theta_m,\Theta_j))\\
&-\tfrac{\partial}{\partial z_m} (g(\Theta_i,\Theta_j)) \Big)+ L(w,Y) {+ Q^{(2)}(w,Y)}+ \mathcal{O}(\rho^2) \Big]=\Gamma^k_{i,j}+ L(w,Y) {+ Q^{(2)}(w,Y)}+ \mathcal{O}(\rho^2).
     \end{align*}
\end{small}
Inserting in the above expression  we get
     \begin{align*}
\mathring{k}_{i,j} =& G(\nabla_{Z_i} (\Upsilon-C), Z_j) - G(\nabla_{Z_i} (a^k Z_k) , Z_j)=\rho(1-\tfrac{w}{R}) g_{i,j}+\rho g(\Theta_i,\nabla_j^{\Sigma} Y)+\rho g(\nabla_i^{\Sigma} Y,\Theta_j)  \\ & + \rho R (Hess^{\Sigma} w)_{i,j}
+\tfrac{\rho^3}{6}\mathscr{T}^{'''}_{i,j}(\Theta,C)+\rho^3 L(w,Y)+\rho Q^{(2)}(w,Y)+\mathcal{O}(\rho^4),
     \end{align*}
where, using Weingarten's relation and since the second fundamental form of $\Sigma$ with respect to the normal vector $N=\tfrac{1}{R}(C-\Theta)$ is $\tfrac{1}{R}g$, we have set
\begin{small}
\begin{align*}
\mathscr{T}^{'''}_{i,j}(\Theta,C)&:=4 \scal{\riem{\Theta}{\Theta_i}{\Theta}}{\Theta_j}-2\scal{\riem{C}{\Theta_i}{\Theta}}{\Theta_j}-2\scal{\riem{\Theta}{\Theta_i}{C}}{\Theta_j}+2\scal{\riem{\Theta}{C}{\Theta}}{(Hess^{\Sigma} \Theta)_{i,j}}\\
&=4 \scal{\riem{\Theta}{\Theta_i}{\Theta}}{\Theta_j}-2\scal{\riem{C}{\Theta_i}{\Theta}}{\Theta_j}-2\scal{\riem{\Theta}{\Theta_i}{C}}{\Theta_j}+2 R^{-2}\scal{\riem{\Theta}{C}{\Theta}}{C}g_{i,j}.
\end{align*}     
\end{small}
To conclude the proof, we recall that 
     \begin{equation*}
       G(\mathring{ \mathcal{M}},\mathring{ \mathcal{M}})^{-\tfrac{1}{2}} = \tfrac{1}{R} - \tfrac{\rho^2}{6 R^3} \scal{\riem{\Theta}{C}{\Theta}}{C} +\rho^2 L(w,Y) +Q^{(2)}(w,Y) + \mathcal{O}(\rho^3);
     \end{equation*}     
     remark that the expansion of $\mathring{k}_{i,j}$ remains unchanged when multiplying by the last three terms of the latter quantity, and changes a little  with the first. Moreover, for the multiplication by the second term, we can see that $\mathring{k}_{i,j} = \rho g(\Theta_i,\Theta_j) + \rho L(w,Y) {+\rho Q^{(2)}(w,Y)}+ \mathcal{O}(\rho^3)$, so we easily get
\begin{small}
     \begin{align*}
 \mathring{h}_{i,j} =& G(\mathring{ \mathcal{M}},\mathring{ \mathcal{M}})^{-\tfrac{1}{2}} \mathring{k}_{i,j} =[\tfrac{1}{R} - \tfrac{\rho^2}{6 R^3} \scal{\riem{\Theta}{C}{\Theta}}{C} +\rho^2 L(w,Y) +Q^{(2)}(w,Y) + \mathcal{O}(\rho^3)] \cdot [\rho(1-\tfrac{w}{R}) g_{i,j}\\
&+\rho g(\Theta_i,\nabla_j^{\Sigma} Y)+\rho g(\nabla_i^{\Sigma} Y,\Theta_j)+ \rho R (Hess^{\Sigma} w)_{i,j}+\tfrac{\rho^3}{6}\mathscr{T}^{'''}_{i,j}(\Theta,C)+\rho^3 L(w,Y)+\rho Q^{(2)}(w,Y)+\mathcal{O}(\rho^4)]\\
=&\tfrac{\rho}{R}(1-\tfrac{w}{R}) g_{i,j}+\tfrac{\rho}{R}( g(\Theta_i,\nabla_j^{\Sigma} Y)+g(\nabla_i^{\Sigma} Y,\Theta_j))+ \rho (Hess^{\Sigma} w)_{i,j}+\tfrac{\rho^3}{6R}\mathscr{S}_{i,j}(\Theta,C)+\rho^3 L(w,Y)\\
&+\rho Q^{(2)}(w,Y)+\mathcal{O}(\rho^4),
     \end{align*}
\end{small}
where
\begin{equation*}
\resizebox{0.92\hsize}{!}{ $\mathscr{S}_{i,j}(\Theta,C):=4 \scal{\riem{\Theta}{\Theta_i}{\Theta}}{\Theta_j}-2\scal{\riem{C}{\Theta_i}{\Theta}}{\Theta_j}-2\scal{\riem{\Theta}{\Theta_i}{C}}{\Theta_j}+ R^{-2}\scal{\riem{\Theta}{C}{\Theta}}{C} g_{i,j},$ }
\end{equation*}
exactly as in \eqref{eq.definitionS}, concluding the proof.    
  \end{proof}
As already noticed before, the tangential component of the perturbation appears at  first order as a Lie derivative in the expansion of the second fundamental form, exactly as expected. Moreover, we can recover the expansion in Lemma $2.3$ of \cite{pac} by setting $C=0$, $R=1$ and $Y=0$.

Finally, the mean curvature of the perturbed manifold $\Sigma_{(p,\s),\rho}(w,Y)$ is obtained taking the trace of its second fundamental form. We 
first recall our convention from Remark \ref{rem.Ogeometric}. 

\begin{theorem}[Perturbed Mean Curvature - Spherical Case]\label{th.Hperturbed}
We have the following expansion for the mean curvature $\mathring{H}$ of $\Sigma_{(p,\s),\rho}(w,Y)$:
\begin{equation}\label{eq.Hperturbed}
\begin{aligned}
\rho R \mathring{H} =&m+(R\Delta_{\Sigma}w+\tfrac{m}{R}w)+\tfrac{\rho^2}{3}\Ric(\Theta,\Theta) - \tfrac{2 \rho^2}{3} \Ric(C,\Theta) + \tfrac{\rho^2 m}{6 R^2} \scal{\riem{\Theta}{C}{\Theta}}{C}\\
&+ \rho^2 L(w,Y) + Q^{(2)}(w,Y) + \mathcal{O}(\rho^3).
    \end{aligned}
\end{equation}

  \begin{proof}
 From \eqref{eq.inversemetric} and \eqref{eq.perturbedsecond} we deduce
\begin{small}
  \begin{align*}
  \rho R \mathring{H}=&\rho \mathring{g}^{i,j} R \mathring{h}_{i,j}= \Big[  \rho^{-1} (1-\tfrac{w}{R})^{-2} g^{i,j}+\rho^{-1}Q^{(2)}(w,Y) -\rho^{-1} g^{i,k}(g(\nabla_k^{\Sigma} Y,\Theta_m)+g(\Theta_k,\nabla_m^{\Sigma}Y))g^{m,j}\\
  &-\tfrac{\rho}{3}g^{i,k} \scal{\riem{\Theta}{\Theta_k}{\Theta}}{\Theta_m}g^{m,j}+\rho L(w,Y)+\mathcal{O}(\rho^2) \Big]\cdot \Big[ \rho(1-\tfrac{w}{R}) g_{i,j}+\rho( g(\Theta_i,\nabla_j^{\Sigma} Y)+g(\nabla_i^{\Sigma} Y,\Theta_j))\\
&+ \rho R (Hess^{\Sigma} w)_{i,j}+\tfrac{\rho^3}{6}\mathscr{S}_{i,j}(\Theta,C)+\rho^3 L(w,Y)+\rho Q^{(2)}(w,Y)+\mathcal{O}(\rho^4) \Big]=m+(R\Delta_{\Sigma}w+\tfrac{m}{R}w)\\
&+\tfrac{\rho^2}{3}\Ric(\Theta,\Theta) - \tfrac{2 \rho^2}{3} \Ric(C,\Theta) + \tfrac{\rho^2 m}{6 R^2} \scal{\riem{\Theta}{C}{\Theta}}{C}+ \rho^2 L(w,Y) + Q^{(2)}(w,Y) + \mathcal{O}(\rho^3), 
\end{align*}
\end{small}
yielding the desired result. 
  \end{proof}
\end{theorem}
Notice that the tangential component $Y$ contributes only at high order, as anticipated above.
For what concerns the flat disk case, we obtain the following expansions for the second fundamental form and the mean curvature of the perturbed surface.
\begin{theorem}[Perturbed Extrinsic Geometry - Disk Case]\label{th.Hperturbeddisk}
The second fundamental form $\mathring{h}_{i,j}$ of the perturbed disk $\Sigma_{p,\rho}(w,Y)$ has the following expansion:
\begin{equation}\label{eq.hperturbeddisk}
   \begin{aligned}
      \mathring{h}_{i,j} =&\rho (Hess^{\Sigma} w)_{i,j}- \tfrac{\rho^3}{3} \Big( \scal{\riem{N}{\Theta_i}{\Theta}}{\Theta_j}+\scal{\riem{\Theta}{\Theta_i}{N}}{\Theta_j}\Big)\\
      &+ \rho^3 L(w,Y)+\rho Q^{(2)}(w,Y)+ \mathcal{O}(\rho^4).
   \end{aligned}
\end{equation}
Moreover, its mean curvature $\mathring{H}$ verifies
 \begin{equation}\label{eq.Hperturbeddisk}
 \rho \mathring{H}=\Delta_{\Sigma}w-\tfrac{2 \rho^2}{3}\Ric(\Theta,N)+\rho^2 L(w,Y)+Q^{(2)}(w,Y)+\mathcal{O}(\rho^3).
 \end{equation}
\end{theorem}
\subsection{Volumes Enclosed by Perturbed Geodesic Double Bubbles}\label{sub.perturbedvolume}
In this subsection we aim to find expansions for the two volumes enclosed by the perturbed geodesic double bubble $\Sigma_{p,\rho}(w,Y)$. We will consider the symmetric difference between the perturbed spherical sector and the non-perturbed one, and then add the expansions obtained in Section \ref{sec.geodesic} for geodesic double bubbles. The main advantage of this method is to solve a non-uniqueness problem. Indeed, we cannot a-priori know how the bottom of the region $P^\sigma$ enclosed by each spherical cap is deformed since the perturbation is defined only on the spherical surface (see Figure \ref{fig.nonunique} below), unless we are in the symmetric case; therefore we cannot compute the volumes of the three perturbed regions separately in an unique way. On the other hand, these contributions will clearly balance each other when computing the volumes $(V_1)_{p,\rho}(w,Y)$ and $(V_2)_{p,\rho}(w,Y)$, since the exponential map cannot create empty chambers. {Recall we are assuming that the perturbation is admissible in the terminology from in Subsection \ref{sub.class}.}
\begin{figure}[h]\label{fig.nonunique}
 \centering
\includegraphics[scale=2.0]{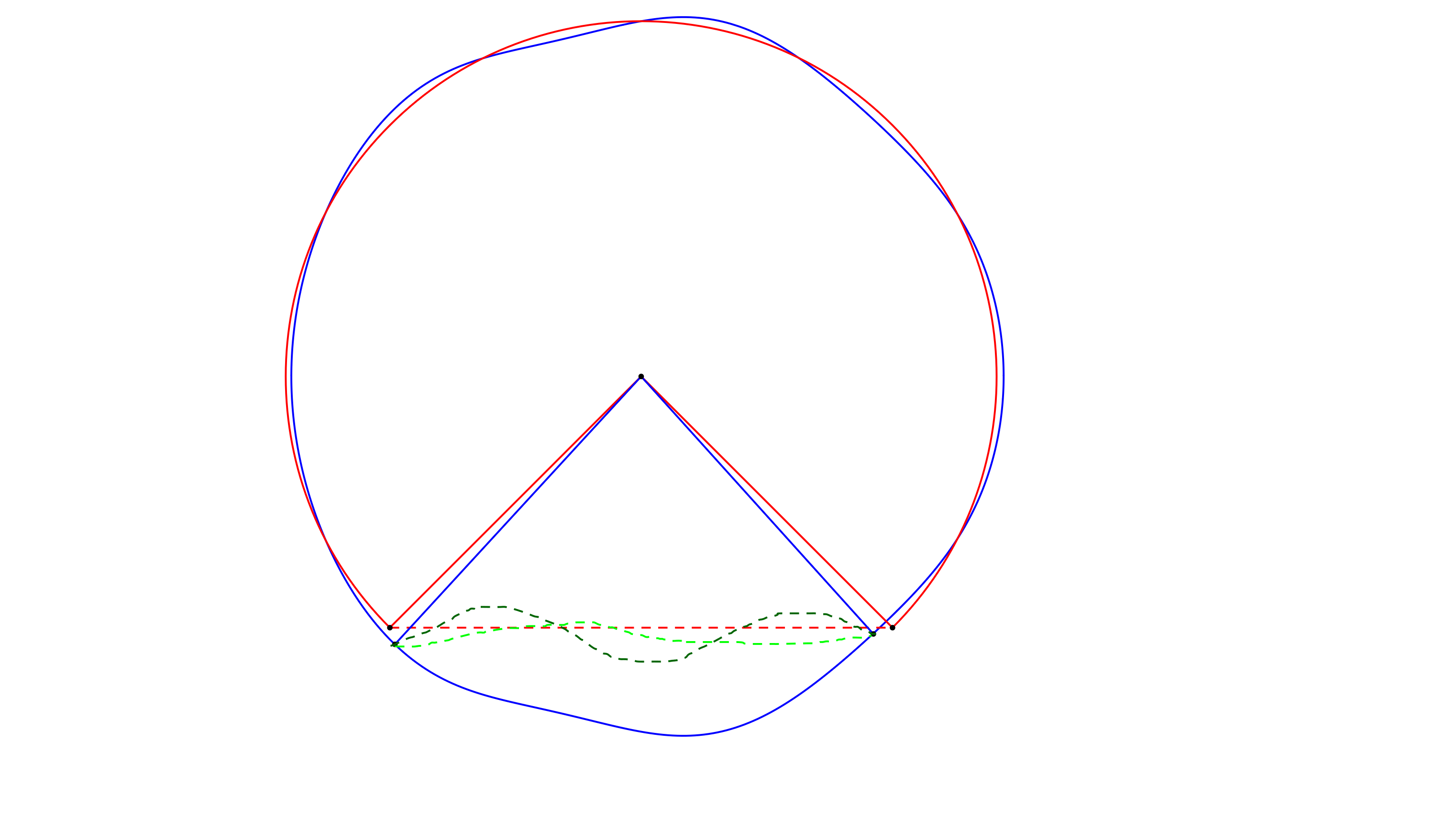}
\caption{Three different ways of closing the perturbed cap}
\end{figure}

For the remainder of this subsection it is convenient to set $\Sigma^\sigma(w,Y):=\phi_{w,Y}(\Sigma^\sigma)$. Define the (perturbed) spherical sector $Sec(\Sigma^\sigma)$ (respectively $Sec(\Sigma^\sigma(w,Y))$) to be the set $\{ x=C^\sigma+s (\Theta-C^\sigma) \in T_p M \mid \Theta \in \Sigma^\sigma, s \in [0,1] \}$ (resp. $\{ x=C^\sigma+s (\Theta-C^\sigma) \in T_p M \mid \Theta \in \Sigma^\sigma(w,Y), s \in [0,1] \}$). Let us denote its image through the exponential map by $Sec(\Sigma_{(p,\s),\rho}^\sigma):=\Esp(\rho \cdot Sec(\Sigma^\sigma))$ (resp. $Sec(\Sigma_{(p,\s),\rho}^\sigma(w,Y)):=\Esp(\rho \cdot Sec(\Sigma^\sigma(w,Y)))$).
Notice that the presence of a possibly non-trivial tangential component $Y$ of the perturbation at the boundary $\Gamma$ might change the volume of the perturbed sector. Let us start with the asymmetric case: from the above discussion we obtain the following equations
\begin{equation*}
\begin{cases}
\resizebox{0.95\hsize}{!}{ $
(V_1)_{(p,\s),\rho}(w,Y)=(V_1)_{(p,\s),\rho}+(\vol( Sec(\Sigma^1_{(p,\s),\rho}(w,Y)))-\vol(Sec(\Sigma^1_{(p,\s),\rho})))+(\vol(Sec(\Sigma^0_{(p,\s),\rho}(w,Y)))-\vol(Sec(\Sigma^0_{(p,\s),\rho}))),$}\\
\resizebox{0.95\hsize}{!}{ $(V_2)_{(p,\s),\rho}(w,Y)=(V_2)_{(p,\s),\rho}+(\vol(Sec(\Sigma^2_{(p,\s),\rho}(w,Y)))-\vol(Sec(\Sigma^2_{(p,\s),\rho})))-(\vol(Sec(\Sigma^0_{(p,\s),\rho}(w,Y)))-\vol(Sec(\Sigma^0_{(p,\s),\rho}))).$}
\end{cases}
\end{equation*}
Therefore, we reduced the problem to the computation of an expansion for $\vol( Sec(\Sigma^\sigma_{(p,\s),\rho}(w,Y)))-\vol( Sec(\Sigma^\sigma_{(p,\s),\rho}))$. We focus our attention to the case $\sigma=1$, since the other ones are analogous, and we drop the index $\sigma$ in order to keep the notation short. The volume of the perturbed sector is computed as follows
\begin{align*}
&\vol( Sec(\Sigma_{(p,\s),\rho}(w,Y)))=\int_{Sec(\Sigma_{(p,\s),\rho}(w,Y))}1=\rho^{m+1} \int_{Sec(\Sigma(w,Y))} \Big( 1-\tfrac{\rho^2}{6}\Ric(x,x)+ \mathcal{O}(\rho^3) \Big) dx\\
=& \rho^{m+1} \bigg\{ \Big( \tfrac{R}{m+1} Area(\Sigma)-\int_{\Sigma} w-R\tfrac{div_\Sigma(Y)}{m+1}+Q^{(2)}(w,Y) d\mu_{\Sigma}\Big)\\
&-\tfrac{\rho^2}{6} \int_\Sigma \Big[ \Big( \int_0^1 \Ric(C+s(\Theta-C+wN+Y),C+s(\Theta-C+wN+Y)) s^m ds \Big) \cdot\\
&\cdot |\Theta-C+wN+Y| \big(1-\tfrac{w}{R} \big)^m (1+div_\Sigma(Y)+Q^{(2)}) \Big] d\mu_{\Sigma}+\mathcal{O}(\rho^3){+\rho^3 L(w,Y)}{+\rho^3 Q^{(2)}(w,Y)} \bigg\}\\
=& \vol( Sec(\Sigma_{(p,\s),\rho}))+\rho^{m+1} \Big( -\int_{\Sigma} w-R\tfrac{div_\Sigma(Y)}{m+1}+Q^{(2)}(w,Y) +\rho^2 L(w,Y) d\mu_{\Sigma}\Big).
\end{align*}
Therefore, we easily get
\begin{equation}
\resizebox{0.92\hsize}{!}{ $ \vol( Sec(\Sigma_{(p,\s),\rho}(w,Y)))-\vol( Sec(\Sigma_{(p,\s),\rho}))=-\rho^{m+1} \int_{\Sigma} w-R\tfrac{div_\Sigma(Y)}{m+1}+Q^{(2)}(w,Y) +\rho^2 L(w,Y) d\mu_{\Sigma},$}
\end{equation}
or more generally, restoring the index $\sigma=0,1,2$:
\begin{equation}
\resizebox{0.92\hsize}{!}{ $
\vol( Sec(\Sigma^\sigma_{(p,\s),\rho}(w,Y)))-\vol( Sec(\Sigma^\sigma_{(p,\s),\rho}))=-\rho^{m+1} \int_{\Sigma^\sigma} w_\sigma-R\tfrac{div_{\Sigma^\sigma}(Y_\sigma)}{m+1}+Q^{(2)}(w,Y) +\rho^2 L(w,Y) d\mu_{\Sigma^\sigma}.$}
\end{equation}
From the  above equations we deduce
\begin{align}
\rho^{-(m+1)}(V_1)_{(p,\s),\rho}(w,Y)=&\rho^{-(m+1)}(V_1)_{(p,\s),\rho} - \int_{\Sigma^1} w_1-R_1\tfrac{div_{\Sigma^1}(Y_1)}{m+1}+Q^{(2)}(w,Y) +\rho^2 L(w,Y) d\mu_{\Sigma^1}\nonumber\\
&- \int_{\Sigma^0} w_0-R_0\tfrac{div_{\Sigma^0}(Y_0)}{m+1}+Q^{(2)}(w,Y) +\rho^2 L(w,Y) d\mu_{\Sigma^0}+\mathcal{O}(\rho^3),\label{eq.perturbedvolume1}
\end{align}
and
\begin{align}
\rho^{-(m+1)}(V_2)_{(p,\s),\rho}(w,Y)=&\rho^{-(m+1)}(V_2)_{(p,\s),\rho} - \int_{\Sigma^2} w_2-R_2\tfrac{div_{\Sigma^2}(Y_2)}{m+1}+Q^{(2)}(w,Y) +\rho^2 L(w,Y) d\mu_{\Sigma^2}\nonumber\\
&+ \int_{\Sigma^0} w_0-R_0\tfrac{div_{\Sigma^0}(Y_0)}{m+1}+Q^{(2)}(w,Y) +\rho^2 L(w,Y) d\mu_{\Sigma^0}+\mathcal{O}(\rho^3).\label{eq.perturbedvolume2}
\end{align}
In the symmetric case, instead of computing the difference between the sectors, we can directly integrate over the regions $(B_1)_{(p,\s),\rho}(w,Y)$ and $(B_2)_{(p,\s),\rho}(w,Y)$ enclosed by the double bubble considered. Analogous calculations to the ones above yield (we set $(V)_{(p,\s),\rho}:=(V_1)_{(p,\s),\rho}=(V_2)_{(p,\s),\rho}$)
\begin{align}
\rho^{-(m+1)}(V_1)_{(p,\s),\rho}(w,Y)=&\rho^{-(m+1)}(V)_{(p,\s),\rho}- \int_{\Sigma^1} w_1-R\tfrac{div_{\Sigma^1}(Y_1)}{m+1}+Q^{(2)}(w,Y) +\rho^2 L(w,Y) d\mu_{\Sigma^1}\nonumber\\
&- \int_{\Sigma^0} w_0-R\tfrac{div_{\Sigma^0}(Y_0)}{m+1}+Q^{(2)}(w,Y) +\rho^2 L(w,Y) d\mu_{\Sigma^0} +\mathcal{O}(\rho^3), \label{eq.perturbedvolume1sym}
\end{align}
and
\begin{align}
\rho^{-(m+1)}(V_2)_{(p,\s),\rho}(w,Y)=&\rho^{-(m+1)}(V)_{(p,\s),\rho}- \int_{\Sigma^2} w_2-R\tfrac{div_{\Sigma^2}(Y_2)}{m+1}+Q^{(2)}(w,Y) +\rho^2 L(w,Y) d\mu_{\Sigma^2}\nonumber\\
&+ \int_{\Sigma^0} w_0-R\tfrac{div_{\Sigma^0}(Y_0)}{m+1}+Q^{(2)}(w,Y) +\rho^2 L(w,Y) d\mu_{\Sigma^0}+\mathcal{O}(\rho^3).  \label{eq.perturbedvolume2sym}
\end{align} 
\subsection{Areas of Perturbed Geodesic Double Bubbles}\label{sub.perturbedarea}
In this subsection we  expand  the $m-$dimensional volumes of the perturbed sheets $\Sigma^\sigma_{(p,\s),\rho}(w,Y)$. As already done before, we focus our attention to the case $\sigma=1$ and we drop this index. Using the expansion for the metric $\mathring{g}_{i,j}$ in \eqref{eq.perturbedfirst} and Taylor's expansion of the square root of the determinant
\begin{equation}
\sqrt{det(Id+H)}=1+\tfrac{1}{2} tr(H)+\mathcal{O}(|H|^2).
\end{equation}
we deduce that the volume element $\sqrt{\mathring{g}}$ satisfies
 \begin{align*}
 \rho^{-m} \sqrt{\mathring{g}} =&(1-\tfrac{w}{R})^m[\sqrt{g}+div_{\Sigma}(Y)+\tfrac{\rho^2}{6} \Ric(\Theta,\Theta)+ \rho^2 L(w,Y) + Q^{(2)}(w,Y)+ \mathcal{O}(\rho^3)]. 
 \end{align*}
Therefore, we can integrate this expression over $\Sigma$ to get
\begin{small}
\begin{align*}
\area(\Sigma_{(p,\s),\rho}(w,Y))&=\rho^{m} \int_{\Sigma} 1-m\tfrac{w}{R}+div_{\Sigma}(Y)+\tfrac{\rho^2}{6} \Ric(\Theta,\Theta)+ \rho^2 L(w,Y) + Q^{(2)}(w,Y)d\mu_{\Sigma} + \mathcal{O}(\rho^{m+3})\\
 &=\area(\Sigma_{(p,\s),\rho})-\rho^{m} \int_{\Sigma} m\tfrac{w}{R}-div_{\Sigma}(Y)+ \rho^2 L(w,Y) + Q^{(2)}(w,Y)d\mu_{\Sigma} + \mathcal{O}(\rho^{m+3}),
\end{align*}
\end{small}
Restoring the index for the boundary component, we get for every $\sigma=0,1,2$ in the asymmetric case, or for any $\sigma=1,2$ in the symmetric case
\begin{equation}\label{eq.perturbedarea}
\resizebox{0.92\hsize}{!}{ $
\area(\Sigma^\sigma_{(p,\s),\rho}(w,Y))=\area(\Sigma^\sigma_{(p,\s),\rho})-\rho^{m} \int_{\Sigma^\sigma} m\tfrac{w_\sigma}{R_\sigma}-div_{\Sigma^\sigma}(Y_\sigma)+ \rho^2 L(w,Y) + Q^{(2)}(w,Y)d\mu_{\Sigma^\sigma} + \mathcal{O}(\rho^{m+3}).$}
\end{equation}
In the symmetric case, equation \eqref{eq.perturbedfirstdisk}, together with Taylor's expansion for the square root of the determinant, ensure that the $m-$dimensional measure of the perturbed geodesic disk satisfies
\begin{equation}\label{eq.perturbedareadisk}
\resizebox{0.92\hsize}{!}{ $
\area(\Sigma^0_{(p,\s),\rho}(w,Y))=\area(\Sigma^0_{(p,\s),\rho})+\rho^{m} \int_{\Sigma^0}div_{\Sigma^0}(Y_0)+\rho^2 L(w,Y)+ Q^{(2)}(w,Y) d\mu_{\Sigma^0}+ \mathcal{O}(\rho^{m+3}).$}
\end{equation}
\subsection{Approximate Equi-angularity of Perturbed Geodesic Double Bubbles}\label{sub.perturbedequiangularity}
Concluding this section, we present a  brief calculation for the sum of the inner conormal vector fields $\mathring{\nu}^\sigma$' s of the surfaces $\Sigma_{(p,\s),\rho}(w,Y)$'s at points in $\Gamma_{(p,\s),\rho}(w,Y)$, which will play a \emph{crucial role} in annihilating a boundary term in the proof of the main Proposition \ref{prop.criticalpointCMC}.
Similarly to what we did in Subsection \ref{sub.equiangularity}, we introduce auxiliary vectors $\tilde{\nu}_{(p,\s)}(w,Y)^\sigma:=\Esp (\rho \nu_{w,Y}^\sigma)$, where $\nu_{w,Y}^\sigma$ denotes the inner conormal vector field of the Euclidean perturbed surface $\Sigma^\sigma(w,Y)$. Refining the argument at the beginning of Subsection \ref{sub.second}, one can show that for every $\sigma$ we have $\mathring{\nu}^\sigma=\tilde{\nu}_{(p,\s)}(w,Y)^\sigma/\norm{\tilde{\nu}_{(p,\s)}(w,Y)^\sigma}_G+\mathcal{O}(\rho^2)+\rho^2 \mathfrak{L}(w,Y)+\rho^2 \mathfrak{Q}^{(2)}(w,Y)$. Here $\mathfrak{L}$ and $\mathfrak{Q}^{(2)}$ denote a $L$-term and a $Q^{(2)}$-term respectively, depending only on $w,Y$ and their first derivatives with respect to the vectors $\nu^\sigma$'s, see the notation in Subsection \ref{sub.class}. Moreover, from \eqref{eq.metricnormalcoord} we can deduce
\begin{equation}
\tilde{\nu}_{(p,\s)}(w,Y)^0+\tilde{\nu}_{(p,\s)}(w,Y)^1+\tilde{\nu}_{(p,\s)}(w,Y)^2=\nu^0_{w,Y}+\nu^1_{w,Y}+\nu^2_{w,Y}+ \mathcal{O}(\rho^2)+\rho^2 \mathfrak{L}(w,Y)+\rho^2 \mathfrak{Q}^{(2)}(w,Y),
\end{equation}
so we arrive at
\begin{equation}
\mathring{\nu}^0+\mathring{\nu}^1+\mathring{\nu}^2=\nu^0_{w,Y}+\nu^1_{w,Y}+\nu^2_{w,Y}+ \mathcal{O}(\rho^2)+\rho^2 \mathfrak{L}(w,Y)+\rho^2 \mathfrak{Q}^{(2)}(w,Y).
\end{equation}
From \cite{hut}, we know that for perturbations of Euclidean double bubbles
\begin{equation}
\nu^0_{w,Y}+\nu^1_{w,Y}+\nu^2_{w,Y}=\big(\tfrac{\partial w_0}{\partial \nu^0}+q_0 w_0\big)N^0+\big(\tfrac{\partial w_1}{\partial \nu^1}+q_1 w_1\big)N^1+\big(\tfrac{\partial w_2}{\partial \nu^2}+q_2 w_2\big)N^2+\mathfrak{Q}^{(2)}(w,Y),
\end{equation}
Here the $0^{th}$-order term is $\nu^0+\nu^1+\nu^2=0$ by the geometric balance equations satisfied by the standard double bubble.
Combining the above expressions, we obtain an expansion for the sum of the inner conormal vectors $\mathring{\nu}^\sigma$' s
\begin{equation}\label{eq.suminnernormals}
\begin{aligned}
\mathring{\nu}^0+\mathring{\nu}^1+\mathring{\nu}^2=&\big(\tfrac{\partial w_0}{\partial \nu^0}+q_0 w_0\big)N^0+\big(\tfrac{\partial w_1}{\partial \nu^1}+q_1 w_1\big)N^1+\big(\tfrac{\partial w_2}{\partial \nu^2}+q_2 w_2\big)N^2\\
&+ \mathcal{O}(\rho^2)+\rho^2 \mathfrak{L}(w,Y)+\mathfrak{Q}^{(2)}(w,Y).
\end{aligned}
\end{equation}
We are finally ready to state the equi-angularity we will impose, namely $\mathring{\nu}^0+\mathring{\nu}^1+\mathring{\nu}^2=0$.
From the equi-angularity of the standard double bubble $\Sigma$, we know that $N_1=N_0+N_2$, therefore after projecting the expansion in \eqref{eq.suminnernormals} on $N_0$ and $N_2$, we obtain the equivalent system
\begin{equation}\label{eq.perturbedequiangularity}
\begin{cases}
\tfrac{\partial w_0}{\partial \nu^0} +q_0 w_0 +\tfrac{\partial w_1}{\partial \nu^1} +q_1 w_1 =\mathcal{O}(\rho^2)+\rho^2 \mathfrak{L}(w,Y)+\mathfrak{Q}^{(2)}(w,Y)=:(e_0)_{(p,\s),\rho}(w,Y) \quad \text{on} \ \Gamma;\\
\tfrac{\partial w_1}{\partial \nu^1} +q_1 w_1 +\tfrac{\partial w_2}{\partial \nu^2} +q_2 w_2 =\mathcal{O}(\rho^2)+\rho^2 \mathfrak{L}(w,Y)+\mathfrak{Q}^{(2)}(w,Y)=:(e_2)_{(p,\s),\rho}(w,Y) \quad \text{on} \ \Gamma.
\end{cases}
\end{equation}
The particular structure of $\mathfrak{L}$ and $\mathfrak{Q}^{(2)}$ guarantees that the boundary data we are imposing in the system \eqref{eq.perturbedequiangularity} belongs to $C^{0,\alpha}(\Gamma) \times C^{0,\alpha}(\Gamma)$.
\section{Fixed Point Argument and Pseudo-Double Bubbles}\label{sec.pseudodb}
Throughout this section, we will discuss how one can choose the perturbation $\phi_{w,Y}$ as in \eqref{eq.definitionperturbation} in order to get the mean curvature vector $(\mathring{H_0},\mathring{H_1},\mathring{H_2})$ of the perturbed geodesic double bubble $\Sigma^\sigma_{(p,\s),\rho}(w,Y)$ as close as possible to a constant three-vector $(H_0,H_1,H_2)$. As we already pointed out in the introduction, one cannot expect to find \emph{at every point} $(p,\s)\in UTM$ a perturbation for which this vector of mean curvatures is constant, because the linearization of this condition is induced by an elliptic operator with non-trivial kernel. This issue will give rise to the concept of \emph{pseudo-double bubbles}, see Definition \ref{def.pseudobubble} below, inspired by the analogous concept considered by Nardulli (Definition $1.2$ in \cite{nar}). 

From an analytic perspective, we will need to find a (unique) solution  $(w_{(p,\s),\rho}, Y_{(p,\s),\rho})\in \mathcal{C}_{amm}$ to a coupled system of PDE's under mixed boundary conditions. In particular, we will have to deal with three quasi-linear second-order elliptic equations for the $w_\sigma$'s, in presence of a non-trivial kernel, and three quasi-linear non-elliptic first-order equations for the $Y_\sigma$'s, under nine Dirichlet- and Robin-type boundary conditions. To solve these, we will apply a Lyapunov-Schmidt reduction, and impose some compatibility conditions to tackle respectively the non-trivial kernel and this excess of boundary conditions.

To begin, let us consider for data $(e_0,e_2) \in C^{k,\alpha}(\Gamma) \times C^{k,\alpha}(\Gamma)$ of small norm the generalized equi-angularity conditions (we are projecting on $N^0$ and $N^2$)
\begin{equation}\label{eq.generalequiangularity}
\begin{cases}
\tfrac{\partial w_0}{\partial \nu^0} +q_0 w_0 +\tfrac{\partial w_1}{\partial \nu^1} +q_1 w_1 =e_0 \quad \text{on} \ \Gamma;\\
\tfrac{\partial w_1}{\partial \nu^1} +q_1 w_1 +\tfrac{\partial w_2}{\partial \nu^2} +q_2 w_2 =e_2 \quad \text{on} \ \Gamma.
\end{cases}
\end{equation}
Notice that the Euclidean condition \eqref{eq.linearisedequiangularity} is equivalent to the above with $(e_0,e_2)=(0,0)$, whereas the perturbed condition \eqref{eq.perturbedequiangularity} is recovered by choosing $(e_0,e_2)=((e_0)_{(p,\s),\rho}(w,Y),(e_2)_{(p,\s),\rho}(w,Y))=(\mathcal{O}(\rho^2)+\rho^2 \mathfrak{L}(w,Y)+\mathfrak{Q}^{(2)}(w,Y),\mathcal{O}(\rho^2)+\rho^2 \mathfrak{L}(w,Y)+\mathfrak{Q}^{(2)}(w,Y))$, so the introduction of a 
non-zero right-hand side in \eqref{eq.generalequiangularity} will allow to pass from a linearized equiangularity condition 
to a genuine one. In what follows, we will show how to produce a unique solution to an \emph{augmented} problem depending Lipschitz-continuously on the data $(e_0,e_1)$ in \eqref{eq.generalequiangularity} and then obtain a solution of the problem under the desired condition \eqref{eq.perturbedequiangularity}.
As we have done in the previous section, we consider a fixed pair $(p,\s)\in UTM$ along which a geodesic double bubble $\Sigma_{(p,\s),\rho}$ is centered and a fixed scale $\rho \in (0,\rho_0)$. We adopt the same conventions and notation of the previous sections. Suppose first that we are in the asymmetric case, so each sheet $\Sigma^\sigma$ is a cap inside a sphere $S(C^\sigma,R_\sigma)$. In view of 
Theorem \ref{th.Hperturbed}, we would like to find a solution $(w,Y)$ to an equation of the form  
\begin{equation*}
\begin{aligned}
m&+(R\Delta_{\Sigma}w+\tfrac{m}{R}w)+\tfrac{\rho^2}{3}\Ric(\Theta,\Theta) - \tfrac{2 \rho^2}{3} \Ric(C,\Theta) + \tfrac{\rho^2 m}{6 R^2} \scal{\riem{\Theta}{C}{\Theta}}{C}\\
&+ \rho^2 L(w,Y) + Q^{(2)}(w,Y) + \mathcal{O}(\rho^3)=\rho R \mathring{H}=\rho R H(\Sigma_{(p,\s),\rho}(w,Y))=m,
\end{aligned}
\end{equation*}
under the boundary conditions \eqref{eq.junctionGamma}, \eqref{eq.boundaryconditionw},  \eqref{eq.boundaryconditionu} and \eqref{eq.generalequiangularity}, where it is understood that $w=(w_0,w_1,w_2)$, $Y=(Y_0,Y_1,Y_2)$ and every quantity  depends on the index $\sigma=0,1,2$ considered. It turns out that one cannot always solve this boundary value problem, since its linearization is induced by the operator
\begin{equation*}
(R_0 \Delta_{\Sigma^0}+\tfrac{m}{R_0},R_1 \Delta_{\Sigma^1}+\tfrac{m}{R_1}, R_2 \Delta_{\Sigma^2}+\tfrac{m}{R_2}):C^{2,\alpha}(\Sigma) \longrightarrow C^{0,\alpha}(\Sigma)
\end{equation*}
which has non-trivial kernel under the linearized boundary conditions \eqref{eq.boundaryconditionw} and \eqref{eq.linearisedequiangularity} imposed on $w$. For this reason, we choose to perform a Lyapunov-Schmidt Reduction, which will simplify the problem to a finite dimensional one. In the following, set for all $\sigma=0,1,2$
\begin{equation}
b_\sigma:=\tfrac{\rho^2}{3}\Ric(\Theta,\Theta) - \tfrac{2 \rho^2}{3} \Ric(C^\sigma,\Theta) + \tfrac{\rho^2 m}{6 R^2} \scal{\riem{\Theta}{C^\sigma}{\Theta}}{C^\sigma} \mid_{\Sigma^\sigma}.
\end{equation}
\subsection{Lyapunov-Schmidt Reduction}
Decompose the space $\mathbb{L}^2(\Sigma)$ in an orthogonal sum $V \oplus V^{\perp}$, where $V:=Ker(R_0 \Delta_{\Sigma^0}+\tfrac{m}{R_0}) \times Ker(R_1 \Delta_{\Sigma^1}+\tfrac{m}{R_1}) \times Ker(R_2 \Delta_{\Sigma^2}+\tfrac{m}{R_2})$ under the linearized boundary conditions \eqref{eq.boundaryconditionw} and \eqref{eq.linearisedequiangularity}, and notice that $C^{2,\alpha}(\Sigma)$ inherits the decomposition as an affine subspace of $\mathbb{L}^2(\Sigma)$. Therefore, when performing the reduction, we will impose \eqref{eq.generalequiangularity} only on $V^\perp$.
The space $V\cap C^{2,\alpha}(\Sigma)$ was studied, under the assumed boundary conditions \eqref{eq.boundaryconditionw} and \eqref{eq.linearisedequiangularity}, by the first-named author in \cite{dim}, where it was shown that this space is generated by infinitesimal translations and rotations of the standard double bubble, and has dimension $2m+1$.
Moreover, by Fredholm's alternative, we know that the following operator is invertible
\begin{equation}
(R_0 \Delta_{\Sigma^0}+\tfrac{m}{R_0},R_1 \Delta_{\Sigma^1}+\tfrac{m}{R_1}, R_2 \Delta_{\Sigma^2}+\tfrac{m}{R_2}):V^{\perp}\cap C^{k+2,\alpha}(\Sigma) \longrightarrow V^{\perp} \cap C^{k,\alpha}(\Sigma),
\end{equation}
and the inverse operator $(R_0 \Delta_{\Sigma^0}+\tfrac{m}{R_0},R_1 \Delta_{\Sigma^1}+\tfrac{m}{R_1}, R_2 \Delta_{\Sigma^2}+\tfrac{m}{R_2})^{-1}$ induces a continuous linear endomorphism on $C^{k,\alpha}(\Sigma)$ if $(e_0,e_2) \in C^{k,\alpha}(\Gamma) \times C^{k,\alpha}(\Gamma)$. For later use, we remark that this operator is Lipschitz-continuous with respect to the data $(e_0,e_2)\in C^{k,\alpha}(\Gamma) \times C^{k,\alpha}(\Gamma)$. 
Using this information, we decompose $w:=\omega+v$, and we would like to find $(\omega,v) \in V^{\perp} \times V$, satisfying the weaker system
\begin{equation}\label{eq.systemw}
\begin{aligned}
(R_\sigma\Delta_{\Sigma^\sigma}\omega_\sigma+\tfrac{m}{R_\sigma}\omega_\sigma)+\rho^2 b_\sigma+ \rho^2 L(w,Y) + Q^{(2)}(w,Y) + \mathcal{O}(\rho^3) + v_\sigma=0, \quad \sigma=0,1,2
\end{aligned}
\end{equation}
under the boundary conditions described above. In order to obtain a solution, we project the system \eqref{eq.systemw} on $V$ and $V^{\perp}$ through projectors $\Pi_V$ and  $\Pi_{V^{\perp}}$ respectively: 
\begin{equation}\label{eq.systemwprojected}
\begin{cases}
v_\sigma=-\Pi_V \Big( \rho^2 b_\sigma+ \rho^2 L(w,Y) + Q^{(2)}(w,Y) + \mathcal{O}(\rho^3) \Big),\\
(R_\sigma \Delta_{\Sigma^\sigma}\omega_\sigma+\tfrac{m}{R_\sigma}\omega_\sigma)=-\Pi_{V^{\perp}}\Big(\rho^2 b_\sigma+ \rho^2 L(w,Y) + Q^{(2)}(w,Y) + \mathcal{O}(\rho^3) \Big),
\end{cases}
 \quad \sigma=0,1,2. 
\end{equation}
For later purposes, it is convenient to rewrite the second equation of \eqref{eq.systemwprojected} in the following form 
\begin{equation}\label{eq.systemwrewritten}
\begin{cases}
\omega_\sigma=\mathfrak{F}_\sigma(\omega_\sigma,v_\sigma,Y_\sigma):=(R_\sigma \Delta_{\Sigma^\sigma}+\tfrac{m}{R_\sigma})^{-1} \Pi_{V^{\perp}}\Big( \rho^2 L(w,Y)
+ Q^{(2)}(w,Y) -\rho^2 b_\sigma+ \mathcal{O}(\rho^3)\Big),\\
v_\sigma=\mathfrak{G}_\sigma(\omega_\sigma,v_\sigma,Y_\sigma):=-\Pi_V \Big( \rho^2 b_\sigma+ \rho^2 L(w,Y) + Q^{(2)}(w,Y) + \mathcal{O}(\rho^3) \Big).
\end{cases}
\end{equation}
Notice that \eqref{eq.systemwprojected} clearly depends also on $Y$ (through $L(w,Y)$ and $Q^{(2)}(w,Y)$), so we need to make a suitable choice of it. In principle, we could extend the boundary data $u_\sigma \nu^\sigma+(Y_\sigma)^{\Gamma}$'s arbitrarily inside the hypersurface $\Sigma^\sigma$, but we encounter two major problems. First of all, most of the extension theorems we are aware of  (see \cite{cic,mcs,ste}), allow to \emph{preserve} the regularity so, since by \eqref{eq.boundaryconditionu} and $w\mid_{\Gamma} \in C^{0,\alpha}(\Gamma)$ we deduce  $u\mid_{\Gamma} \in C^{0,\alpha}(\Gamma)$, we expect to get $Y \in C^{0,\alpha}(\Sigma)$; however, these results do \emph{not improve} the regularity and hence we cannot get anything better than $Y \in C^{0,\alpha}(\Sigma)$, which is not enough to have $(w,Y)$ admissible. Secondly, we would like the choice $Y$ to depend smoothly on the point $(p,\s) \in UTM$ along which we are setting up our analysis, and this is difficult to achieve without a selection principle for these possible extensions. We therefore choose to impose additional equations on the tangential component $Y$ which will solve both problems at once; more precisely, we will prescribe its divergence under the assumed boundary conditions. It is worth mentioning that this problem is \emph{overdetermined} and \emph{non-elliptic} (in an analytical sense), but it still has well-posedness and regularity theory by the results in \cite{sch}, under proper compatibility conditions between the prescribed data, which can be met in our case since we have freedom in choosing the prescribing functions. Finally, we remark that one cannot find $Y$ in the form $Y:=\nabla y$ where $y$ satisfies an elliptic equation: indeed, under the conditions \eqref{eq.boundaryconditionu} and \eqref{eq.junctionGamma}, such a problem is overdetermined.
\subsection{A Fluid Dynamics Approach}
The method we are about to adopt is inspired by analogous problems in fluid dynamics, where one prescribes the so-called \emph{compressibility} of the flow, that is the divergence of its velocity. For a mathematical point of view see \cite{aco,gir} and the exhaustive book \cite{sch}, where Schwartz gives a beautiful treatment of the regularity theory for the more general case of Hodge decomposition.
Before introducing new equations on $Y$, let us give a brief summary of some results from \cite{sch} we intend to use. Given a smooth Riemannian manifold with boundary $(X,\partial X, g_X)$, let us consider the following boundary value problem: for any data $f \in C^{k,\alpha}(X)$ and $V \in C^{k,\alpha}(\partial X;TX)$ we look for a solution $Y \in C^{k+1,\alpha}(X;TX)$ of
\begin{equation}\label{eq.divergenceproblemgeneral}
\begin{cases}
\Div_X (Y)=f \quad \text{on } X;\\
Y=V \quad \text{on } \partial X.
\end{cases}
\end{equation}
From Lemma $3.5.5$ in \cite{sch}, we know that this problem has a solution if and only if the following compatibility condition is met
\begin{equation}\label{eq.compatibilitygeneral}
\int_X f d\mu_X=-\int_{\partial X} g_X(\nu_X,V) d\mu_{\partial X},
\end{equation}
where $\nu_X$ is the unit vector normal to $\partial X$ and pointing inwards. Clearly, this condition is necessary by the divergence theorem, but the sufficiency is far from being trivial, in fact, we repeat, this problem is not elliptic and over-determined. Furthermore, by  Corollary $3.3.4$ in \cite{sch} one can always choose a solution to \eqref{eq.divergenceproblemgeneral} verifying, for some constants $D_k$,
\begin{equation}\label{eq.ellipticestimatediv}
\norm{Y}_{C^{k+1,\alpha}(X;TX)} \le D_k \big(\norm{f}_{C^{k,\alpha}(X)}+\norm{V}_{C^{k,\alpha}(\partial X;TX)} \big).
\end{equation}
Here we have used Morrey's embedding theorem in order to substitute the Sobolev norms considered in \cite{sch} with the above H\"older norms. For what it regards the uniqueness, the solution is unique up to what the author calls \emph{Dirichlet} and \emph{Neumann} fields (Theorem $3.2.5$ and Corollary $3.2.6$ in \cite{sch}); however, when the manifold $(X,\partial X)$ is contractible, both these classes collapse on the trivial one, so we get that the solution to \eqref{eq.divergenceproblemgeneral} is unique, see the Remark after Theorem $2.2.2$ in \cite{sch}. Therefore in this case, we have a well-defined and continuous operator $\Div^{-1}:C^{k,\alpha}(X;TX) \rightarrow C^{k+1,\alpha}(X;TX)$, inducing continuous endomorphisms of $C^{k,\alpha}(X;TX)$ for every $k$.
Altogether, since in our case $(\Sigma^\sigma,\Gamma)$ is contractible, we have existence, uniqueness and $C^{1,\alpha}$-norm bounds (i.e. well-posedness) for the following problem
\begin{equation}\label{eq.systemY}
\begin{cases}
\Div_{\Sigma^\sigma} (Y_\sigma)=f_\sigma \quad \text{on } \Sigma;\\
g(Y_\sigma,\nu^\sigma)=u_\sigma \quad \text{on } \Gamma,\\
(Y_\sigma)^\Gamma=0 \quad \text{on } \Gamma,
\end{cases}
\end{equation}
as long as the $f_\sigma$'s are chosen so that
\begin{equation}\label{eq.compatibilitycondition}
\int_{\Sigma^\sigma} f_\sigma d\mu_{\Sigma^\sigma}=-\int_{\Gamma} u_\sigma d\mu_{\Gamma}.
\end{equation}
Remark that the conditions imposed on the fields $(Y_\sigma)^\Gamma$'s imply the condition \eqref{eq.junctionGamma}, and from now on we will consider this more strict condition:
\begin{equation}\label{eq.junctionGamma1}
(Y_0)^\Gamma=(Y_1)^\Gamma=(Y_2)^\Gamma=0.
\end{equation}
We now proceed to construct suitable data $f_\sigma$. In order to shorten our notation, we will drop the volume
elements in the integrals below, since they are already implicitly defined from the domains of integration
considered. Using equations \eqref{eq.boundaryconditionu} we obtain
\begin{equation}\label{eq.integraluGammaw}
\int_{\Gamma} u_0= \tfrac{1}{\sqrt{3}} \int_{\Gamma} w_0+2 w_2, \quad \int_{\Gamma} u_1= \tfrac{1}{\sqrt{3}} \int_{\Gamma} w_0-w_2, \quad \int_{\Gamma} u_2= \tfrac{1}{\sqrt{3}} \int_{\Gamma} -2 w_0-w_2.
\end{equation}
Therefore we are naturally led to consider $\int_\Gamma w_\sigma$; applying the system \eqref{eq.generalequiangularity} and the divergence theorem we get
\begin{align*}
\int_\Gamma w_0&= \tfrac{1}{q_0} \int_\Gamma -\tfrac{\partial w_0}{\partial \nu^0}-\tfrac{\partial w_1}{\partial \nu^1}-q_1 w_1+e_0=\tfrac{1}{q_0} \Big(\int_{\Sigma^0} \Delta w_0+\int_{\Sigma^1} \Delta w_1 \Big)-\tfrac{q_1}{q_0} \int_\Gamma w_1+\tfrac{1}{q_0} \int_\Gamma e_0;\\
\int_\Gamma w_2&= \tfrac{1}{q_2} \int_\Gamma -\tfrac{\partial w_2}{\partial \nu^2}-\tfrac{\partial w_1}{\partial \nu^1}-q_1 w_1+e_2=\tfrac{1}{q_2} \Big(\int_{\Sigma^2} \Delta w_2+\int_{\Sigma^1} \Delta w_1 \Big)-\tfrac{q_1}{q_2} \int_\Gamma w_1+\tfrac{1}{q_2} \int_\Gamma e_2.
\end{align*}
Summing these two equations and using \eqref{eq.boundaryconditionw} we also find
\begin{equation*}
\resizebox{0.99\hsize}{!}{ $
\int_\Gamma w_1=\int_\Gamma w_0+w_2= \tfrac{1}{q_0} \Big(\int_{\Sigma^0} \Delta w_0+\int_{\Sigma^1} \Delta w_1 \Big)+\tfrac{1}{q_2} \Big(\int_{\Sigma^2} \Delta w_2+\int_{\Sigma^1} \Delta w_1 \Big)- \big( \tfrac{q_1}{q_0} + \tfrac{q_1}{q_2}\big) \int_\Gamma w_1+\tfrac{1}{q_0} \int_\Gamma e_0+\tfrac{1}{q_2} \int_\Gamma e_2.$}
\end{equation*}
Firstly, let us consider the easier case $(e_0,e_2)=(0,0)$. Solving for $\int_\Gamma w_1$ and substituting back in the previous identities we arrive at
\begin{equation}\label{eq.integralwGamma}
\begin{aligned}
\int_\Gamma w_0=& P \Big[ q_2 \Big(\int_{\Sigma^0} \Delta w_0+\int_{\Sigma^1} \Delta w_1 \Big)+q_1 \Big(\int_{\Sigma^0} \Delta w_0-\int_{\Sigma^2} \Delta w_2 \Big) \Big],\\
\int_\Gamma w_1=& P \Big[ q_2 \Big(\int_{\Sigma^0} \Delta w_0+\int_{\Sigma^1} \Delta w_1 \Big)+q_0 \Big(\int_{\Sigma^1} \Delta w_1+\int_{\Sigma^2} \Delta w_2 \Big) \Big],\\
\int_\Gamma w_2=& P \Big[ q_1 \Big(-\int_{\Sigma^0} \Delta w_0+\int_{\Sigma^2} \Delta w_2 \Big)+q_0 \Big(\int_{\Sigma^1} \Delta w_1+\int_{\Sigma^2} \Delta w_2 \Big) \Big].
\end{aligned}
\end{equation}
Here we have set $P:=(q_0 q_1+q_1 q_2 + q_0 q_2)^{-1}$. Together with \eqref{eq.definitionq} and \eqref{eq.integraluGammaw}, the identities \eqref{eq.integralwGamma} imply
\begin{equation}\label{eq.integraluGammaLaplace}
\begin{aligned}
\int_\Gamma u_0=& P H_1 H_2 \Big[ -\tfrac{H_0}{H_1 H_2} \int_{\Sigma^0} \Delta w_0+\int_{\Sigma^1} R_1 \Delta w_1 +\int_{\Sigma^2} R_2 \Delta w_2 \Big],\\
\int_\Gamma u_1=& P H_0 H_2 \Big[ - \int_{\Sigma^0} R_0 \Delta w_0-\tfrac{H_1}{H_0 H_2}\int_{\Sigma^1} \Delta w_1+\int_{\Sigma^2} R_2 \Delta w_2  \Big],\\
\int_\Gamma u_2=& P H_0 H_1 \Big[ \int_{\Sigma^0} R_0 \Delta w_0+\int_{\Sigma^1} R_1 \Delta w_1-\tfrac{H_2}{H_0 H_1}\int_{\Sigma^2} \Delta w_2 \Big].
\end{aligned}
\end{equation}
Let us focus on the case $\sigma=1$. Since we will need to bootstrap the regularity of the solution, we need to get rid of the Laplace operators appearing in the formula above. Thus we appeal to the equations \eqref{eq.systemw} solved by the functions $w_\sigma$'s, and obtain
\begin{equation}\label{eq.integraluGammadifferent}
\begin{aligned}
\int_\Gamma u_1=& P H_0 H_2 \Big[ \int_{\Sigma^0} \Big( v_0+\tfrac{m}{R_0}w_0+\rho^2 b_0 + \rho^2 L(w_0,Y_0) + Q^{(2)}(w_0,Y_0) + \mathcal{O}(\rho^3) \Big)\\
+&\tfrac{H_1^2}{H_0 H_2} \int_{\Sigma^1} \Big( v_1+\tfrac{m}{R_1}w_1+\rho^2 b_1 + \rho^2 L(w_1,Y_1) + Q^{(2)}(w_1,Y_1) + \mathcal{O}(\rho^3) \Big)\\
-&\int_{\Sigma^2} \Big( v_2+\tfrac{m}{R_2}w_2+\rho^2 b_2 + \rho^2 L(w_2,Y_2) + Q^{(2)}(w_2,Y_2) + \mathcal{O}(\rho^3) \Big) \Big].
\end{aligned}
\end{equation}
As by \eqref{eq.compatibilitycondition}, we would like to express the right-hand side of \eqref{eq.integraluGammadifferent} as the integral of a function $-f_1:\Sigma^1 \rightarrow \R$. In principle, we may consider the right-hand side as a constant function, so the function $-f_1$ is given by an average; however, this would create an extremely troublesome non-local term in the system \eqref{eq.systemY} we are planning to solve. We therefore use the following trick, based on the change of variables in the integrals. Let $R_{\s} \in O(m+1)$ be the reflection along the hyperplane $\s^{\perp}$, and define for every $\sigma, \ \tau =0,1,2$, $\sigma \neq \tau$, functions $h_\sigma^\tau:\Sigma^\sigma \rightarrow \Sigma^\tau$ by
\begin{equation}
\begin{aligned}
h_0^1(x):= \tfrac{R_1}{R_0} (R_{\s}(x-C^0))&+C^1, \quad h_0^2(x):= \tfrac{R_2}{R_0} (x-C^0)+C^2, \quad h_2^1(x):= \tfrac{R_1}{R_2} (R_{\s}(x-C^2))+C^1,\\
h_1^0&:=(h_0^1)^{-1}, \quad h_2^0:=(h_0^2)^{-1}, \quad h_1^2:=(h_2^1)^{-1}.
\end{aligned}
\end{equation}
Notice that the $C^k$-norms (and the Jacobians) of the functions $h_\sigma^\tau$'s can be bounded in terms of the radii $R_\sigma$, $R_\tau$ and the dimension $m$. We can now change  variables in the above formula \eqref{eq.integraluGammadifferent}  to obtain a more suitable expression
\begin{equation*}
\begin{aligned}
\int_{\Sigma^1} f_1=-\int_\Gamma u_1=&- P H_0 H_2 \Big[ \big(\tfrac{R_0}{R_1}\big)^m \int_{\Sigma^1} \Big( v_0+\tfrac{m}{R_0}w_0+\rho^2 b_0 + \rho^2 L(w_0,Y_0) + Q^{(2)}(w_0,Y_0) + \mathcal{O}(\rho^3) \Big)\circ h_1^0\\
+&\tfrac{H_1^2}{H_0 H_2} \int_{\Sigma^1} \Big( v_1+\tfrac{m}{R_1}w_1+\rho^2 b_1 + \rho^2 L(w_1,Y_1) + Q^{(2)}(w_1,Y_1) + \mathcal{O}(\rho^3) \Big)\\
-&\big(\tfrac{R_2}{R_1}\big)^m \int_{\Sigma^1} \Big( v_2+\tfrac{m}{R_2}w_2+\rho^2 b_2 + \rho^2 L(w_2,Y_2) + Q^{(2)}(w_2,Y_2) + \mathcal{O}(\rho^3) \Big) \circ h_1^2 \Big],
\end{aligned}
\end{equation*}
from which we deduce a local expression for the data $f_1$
\begin{equation}\label{eq.definitionf1}
\begin{aligned}
f_1:=&- P H_0 H_2 \Big[ \big(\tfrac{R_0}{R_1}\big)^m \Big( v_0+\tfrac{m}{R_0}w_0+\rho^2 b_0 + \rho^2 L(w_0,Y_0) + Q^{(2)}(w_0,Y_0) + \mathcal{O}(\rho^3) \Big)\circ h_1^0\\
+&\tfrac{H_1^2}{H_0 H_2} \Big( v_1+\tfrac{m}{R_1}w_1+\rho^2 b_1 + \rho^2 L(w_1,Y_1) + Q^{(2)}(w_1,Y_1) + \mathcal{O}(\rho^3) \Big)\\
-&\big(\tfrac{R_2}{R_1}\big)^m \Big( v_2+\tfrac{m}{R_2}w_2+\rho^2 b_2 + \rho^2 L(w_2,Y_2) + Q^{(2)}(w_2,Y_2) + \mathcal{O}(\rho^3) \Big) \circ h_1^2 \Big]. 
\end{aligned}
\end{equation}
Similarly, from \eqref{eq.integraluGammaLaplace} and \eqref{eq.systemw} we are naturally led to set
\begin{equation}\label{eq.definitionf0}
\begin{aligned}
f_0:=&- P H_1 H_2 \Big[ \tfrac{H_0^2}{H_1 H_2} \Big( v_0+\tfrac{m}{R_0}w_0+\rho^2 b_0 + \rho^2 L(w_0,Y_0) + Q^{(2)}(w_0,Y_0) + \mathcal{O}(\rho^3) \Big)\\
-&\big(\tfrac{R_1}{R_0}\big)^m \Big( v_1+\tfrac{m}{R_1}w_1+\rho^2 b_1 + \rho^2 L(w_1,Y_1) + Q^{(2)}(w_1,Y_1) + \mathcal{O}(\rho^3) \Big)\circ h_0^1\\
-&\big(\tfrac{R_2}{R_0}\big)^m \Big( v_2+\tfrac{m}{R_2}w_2+\rho^2 b_2 + \rho^2 L(w_2,Y_2) + Q^{(2)}(w_2,Y_2) + \mathcal{O}(\rho^3) \Big) \circ h_0^2 \Big],
\end{aligned}
\end{equation}
and
\begin{equation}\label{eq.definitionf2}
\begin{aligned}
f_2:=&- P H_0 H_1 \Big[ -\big(\tfrac{R_0}{R_2}\big)^m \Big( v_0+\tfrac{m}{R_0}w_0+\rho^2 b_0 + \rho^2 L(w_0,Y_0) + Q^{(2)}(w_0,Y_0) + \mathcal{O}(\rho^3) \Big)\circ h_2^0\\
-&\big(\tfrac{R_1}{R_2}\big)^m \Big( v_1+\tfrac{m}{R_1}w_1+\rho^2 b_1 + \rho^2 L(w_1,Y_1) + Q^{(2)}(w_1,Y_1) + \mathcal{O}(\rho^3) \Big)\circ h_2^1\\
+&\tfrac{H_2^2}{H_0 H_1} \Big( v_2+\tfrac{m}{R_2}w_2+\rho^2 b_2 + \rho^2 L(w_2,Y_2) + Q^{(2)}(w_2,Y_2) + \mathcal{O}(\rho^3) \Big) \circ h_1^2 \Big].
\end{aligned}
\end{equation}

Restoring now the dependence on the data $(e_0,e_2)$, and arguing as above, one arrives to equations of the form
\begin{equation}
\int_{\Sigma^\sigma} f_\sigma=\int_{\Sigma^\sigma} f_\sigma \mid_{(e_0,e_2)=(0,0)}+ \int_\Gamma P'_0 e_0+P'_2 e_2,
\end{equation}
for some constants $P'_0$ and $P'_2$ depending only on the dimension $m$ and the radii $R_\sigma$'s. For each $\sigma=0,1,2$ we solve uniquely the system
\begin{equation}\label{eq.systemforE}
\begin{cases}
\Delta E_\sigma-E_\sigma=0 \quad \text{on } \Gamma,\\
\partial_{\nu^\sigma} E_\sigma=-P'_0 e_0-P'_2 e_2 \quad \text{on } \Gamma,
\end{cases}
\end{equation}
so that we obtain
\begin{equation}
\int_{\Sigma^\sigma} f_\sigma=\int_{\Sigma^\sigma} f_\sigma\mid_{(e_0,e_2)=(0,0)}+ E_\sigma,
\end{equation}
and we can finally set $f_\sigma= f_\sigma\mid_{(e_0,e_2)=(0,0)}+ E_\sigma$. Notice that this choice of $f_\sigma$ depends Lipschitz-continuously on the data $(e_0,e_2)$. In fact, for every $k\in \mathbb{N}$ there exist a constant $D_k$ such that the solution to \eqref{eq.systemforE} satisfies $\norm{E}_{C^{k+1,\alpha}(\Sigma)} \le D_k\norm{(e_0,e_2)}_{C^{k,\alpha}(\Gamma) \times C^{k,\alpha}(\Gamma)}$ and, given two pairs of data $(e_0,e_2)$ and $(e_0',e_2')$, then $\norm{E-E'}_{C^{k+1,\alpha}(\Sigma)} \le D_k\norm{(e_0-e_0',e_2-e_2')}_{C^{k,\alpha}(\Gamma) \times C^{k,\alpha}(\Gamma)}$, where we denoted by $E$ and $E'$ the solutions relatives to the respective data $(e_0,e_2)$ and $(e_0',e_2')$.
With the scope of approaching the problem through a fixed point argument, we rewrite the first equation in \eqref{eq.systemY} as
\begin{equation}\label{eq.definitionfunctionalH}
Y_\sigma=\mathfrak{H}_\sigma(\omega_\sigma,v_\sigma,Y_\sigma):=\Div^{-1}(f_\sigma), \quad \sigma=0,1,2.
\end{equation}
We stress  that this is a system of equations coupling the $(w_\sigma,Y_\sigma)$'s with each other, in contrast to \eqref{eq.systemwrewritten}, which could in principle be solved in each sheet separately.
\subsection{Fixed Point Argument}
We are finally ready to set up the fixed point argument in order to solve our problem. Slightly abusing notation, we will drop the index $\sigma$ and consider all the quantities involved as three-vectors. We will always implicitly assume the boundary conditions \eqref{eq.junctionGamma1}, \eqref{eq.boundaryconditionw}, \eqref{eq.boundaryconditionu} and \eqref{eq.generalequiangularity}. Consider the operator
\begin{equation*}
\resizebox{0.95\hsize}{!}{ $(\mathfrak{F},\mathfrak{G},\mathfrak{H}):\big(C^{0,\alpha}(\Sigma)\cap V^{\perp} \big)\times \big(C^{0,\alpha}(\Sigma)\cap V \big)\times C^{0,\alpha}(\Sigma;T\Sigma) \longrightarrow \big(C^{0,\alpha}(\Sigma)\cap V^{\perp} \big)\times \big(C^{0,\alpha}(\Sigma)\cap V \big)\times C^{0,\alpha}(\Sigma;T\Sigma) $}
\end{equation*}
induced by the functions $\mathfrak{F}_\sigma$'s, $\mathfrak{G}_\sigma$'s and $\mathfrak{H}_\sigma$'s previously introduced in \eqref{eq.systemwrewritten} and \eqref{eq.definitionfunctionalH}. 
Using the properties of $L(w,Y)$ and $Q^{(2)}(w,Y)$, we are going to show that the operator $(\mathfrak{F},\mathfrak{G},\mathfrak{H})$ has a unique fixed point in a ball around the origin, of $C^{0,\alpha}-$radius $\rho^2$ for any $\rho<\rho_3$ and $(e_0,e_2)\in C^{0,\alpha}(\Gamma) \times C^{0,\alpha}(\Gamma)$ small enough. Here the threshold $\rho_3$ depends on $(p,\s)$ and hence, by compactness, ultimately only on $UTM$. Furthermore, this fixed point will depend Lipschitz-continuously on the data $(e_0,e_2)$; this property will reveal crucial in solving the original problem, where $(e_0,e_2)=((e_0)_{(p,\s),\rho}(w,Y),(e_2)_{(p,\s),\rho}(w,Y))$. It is equivalent to consider the direct product of three balls $\mathcal{U}:=B(0,c_1 \rho^2)\times B(0,c_2 \rho^2)\times B(0,c_3 \rho^2)\subseteq \big(C^{0,\alpha}(\Sigma)\cap V^{\perp} \big)\times \big(C^{0,\alpha}(\Sigma)\cap V \big)\times C^{0,\alpha}(\Sigma;T\Sigma )$ for some constants $c_1,c_2,c_3>0$ instead of an actual ball in the product space. Also, let us consider data $(e_0,e_2) \in B(0,c_4 \rho^2) \times B(0,c_5 \rho^2) \subset C^{0,\alpha}(\Gamma) \times C^{0,\alpha}(\Gamma)$. A natural approach would be to show that $(\mathfrak{F},\mathfrak{G},\mathfrak{H})$ is a contraction in $\mathcal{U}$; unfortunately, the operator $\mathfrak{H}$ is not a contraction, if seen as a function of the triple $(\omega,v,Y)$, see Remark \ref{rem.Hnoncontraction} below. Therefore we outline the following plan: 
\begin{itemize}
\item we first show that $\mathfrak{H}(\omega,v,\cdot)$ is a contraction in a small enough ball $B(0,c_3 \rho^2)\subseteq C^{0,\alpha}(\Sigma;T\Sigma )$ as above for any fixed $(\omega,v)\in B(0,c_1 \rho^2)\times B(0,c_2 \rho^2)\subseteq \big(C^{0,\alpha}(\Sigma)\cap V^{\perp} \big)\times \big(C^{0,\alpha}(\Sigma)\cap V \big)$ and any fixed data $(e_0,e_2) \in B(0,c_4 \rho^2) \times B(0,c_5 \rho^2) \subset C^{0,\alpha}(\Gamma) \times C^{0,\alpha}(\Gamma)$;
\item we show the Lipschitz-continuity of the unique solution $Y=Y(\omega,v)$ found as fixed point of the 
previous contraction as a function of $(\omega,v)$;
\item we show the Lipschitz-continuity of the unique solution $Y=Y(e_0,e_2)$ found above as a function of the data $(e_0,e_2)$;
\item we prove that $(\mathfrak{F}(\cdot,\cdot,Y(\cdot,\cdot)),\mathfrak{G}(\cdot,\cdot,Y(\cdot,\cdot)))$ is a contraction of the set $B(0,c_1 \rho^2)\times B(0,c_2 \rho^2)$ for any fixed data $(e_0,e_2) \in B(0,c_4 \rho^2) \times B(0,c_5 \rho^2) \subset C^{0,\alpha}(\Gamma) \times C^{0,\alpha}(\Gamma)$;
\item finally, combining the results obtained in the previous points, we find a unique  solution $(\omega,v)$ associated to the data $(e_0,e_2)=((e_0)_{(p,\s),\rho}(w,Y),(e_2)_{(p,\s),\rho}(w,Y))$.
\end{itemize}

When carrying out this plan, we will produce several estimates depending on constants $c,c_1,c_2,c_3,c_4$ and $c_5>0$, whose values may change from line to line, but remain always finite and depend only on the dimension $m$, and the geometries of $\Sigma \subset \R^{m+1}$ and $UTM$. It is important to keep in mind that the normal component $w$ is given by the sum of the kernel component $v$ and the orthogonal component $\omega$. Let us start with the proof of the first point above. 

$\bullet$ First of all, we show that $\mathfrak{H}$ sends $B(0,c_3 \rho^2)$ in itself , if we fix data $(e_0,e_2) \in B(0,c_4 \rho^2) \times B(0,c_5 \rho^2)$ and normal components $(\omega,v)\in B(0,c_1 \rho^2)\times B(0,c_2 \rho^2)$ for suitable constants $c_1,c_2,c_3,c_4$ and $c_5>0$. Here the values of these constants depend only on the radii $R_\sigma$'s, the dimension $m$ and on $UTM$. Equations \eqref{eq.definitionfunctionalH}, \eqref{eq.definitionf1}, \eqref{eq.definitionf0} and \eqref{eq.definitionf2}, together with the convention on the $L$-$Q^{(2)}$-$\mathcal{O}$-terms, yield
\begin{align*}
&\norm{\mathfrak{H}(\omega,v,Y)}_{C^{0,\alpha}(\Sigma)}\le \norm{\mathfrak{H}(\omega,v,Y)}_{C^{1,\alpha}(\Sigma)} \\ &\le c\norm{f}_{C^{0,\alpha}(\Sigma)}\le c \Big( \norm{(\omega,v)}_{C^{0,\alpha}(\Sigma)}+\rho^2 \norm{b}_{C^{0,\alpha}(\Sigma)}+\norm{\rho^2 L(\omega+v,Y)}_{C^{0,\alpha}(\Sigma)}\\
&+\norm{Q^{(2)}(\omega+v,Y)}_{C^{0,\alpha}(\Sigma)} +\norm{\mathcal{O}(\rho^3)}_{C^{0,\alpha}(\Sigma)}  \Big)+\norm{E}_{C^{0,\alpha}(\Sigma)}\le c \rho^2 +c \norm{(e_0,e_2)}_{C^{0,\alpha}(\Gamma)\times C^{0,\alpha}(\Gamma)}\le c \rho^2.
\end{align*}
We now prove the contraction property for $\mathfrak{H}(\omega,v, \cdot)$. Consider a fixed pair $(\omega,v) \in B(0,c_1 \rho^2)\times B(0,c_2 \rho^2)$, two points $Y_1,Y_2 \in B(0,c_3 \rho^2)$, and fixed data $(e_0,e_2) \in B(0,c_4 \rho^2) \times B(0,c_5 \rho^2)$, and compute (recall that by Remark \ref{rem.Ogeometric} there is cancellation of the $\mathcal{O}(\rho^3)$ terms, whereas the terms involving $E$ cancel since the data $(e_0,e_2)$ is fixed)
\begin{align*}
\| \mathfrak{H}(\omega,v,Y_1)&-\mathfrak{H}(\omega,v,Y_2)\|_{C^{0,\alpha}(\Sigma)}\le \norm{\mathfrak{H}(\omega,v,Y_1)-\mathfrak{H}(\omega,v,Y_2)}_{C^{1,\alpha}(\Sigma)} \le c \norm{f(\omega,v,Y_1)-f(\omega,v,Y_2)}_{C^{0,\alpha}(\Sigma)}\\
\le& c \rho^2 \norm{L((\omega+v,Y_1)-(\omega+v,Y_2))}_{C^{0,\alpha}(\Sigma)}+ c\norm{Q^{(2)}(\omega+v,Y_1)-Q^{(2)}(\omega+v,Y_2)}_{C^{0,\alpha}(\Sigma)}\\
\le& c \rho^2 \norm{(\omega+v,Y_1)-(\omega+v,Y_2)}_{C^{0,\alpha}(\Sigma)}+c\big(\norm{(\omega+v,Y_1)}_{C^{0,\alpha}(\Sigma)}+\norm{(\omega+v,Y_2)}_{C^{0,\alpha}(\Sigma)}\big) \\
&\cdot\norm{(\omega+v,Y_1)-(\omega+v,Y_2)}_{C^{0,\alpha}(\Sigma)} \le c \rho^2 \norm{Y_1-Y_2}_{C^{0,\alpha}(\Sigma)}.
\end{align*}
Thus, we can always chose $\rho_3$ small enough so that $c\rho^2<1$, ensuring that $\mathfrak{H}$ is a contraction of $B(0,c_3 \rho^2)$. Therefore, for $(\omega,v) \in B(0,c_1 \rho^2)\times B(0,c_2 \rho^2)$ and $(e_0,e_2) \in B(0,c_4 \rho^2) \times B(0,c_5 \rho^2)$ there exists a unique element $Y=Y(\omega,v) \in B(0,c_3 \rho^2)$ which is a fixed point for $\mathfrak{H}$. 

$\bullet$ We now show that this solution $Y(\omega,v)$ depends Lipschitz continuosly on $(\omega,v)$; this will play a crucial role in establishing the fourth point. Fix data $(e_0,e_2) \in B(0,c_4 \rho^2) \times B(0,c_5 \rho^2)$. Take two pairs $(\omega_1,v_1),(\omega_2,v_2) \in  B(0,c_1 \rho^2)\times B(0,c_2 \rho^2)$ and denote their associated fixed points by $Y_1:=Y(\omega_1,v_1)$ and $Y_2:=Y(\omega_2,v_2)$. Exploting once again the properties of the $L$-$Q^{(2)}$-$\mathcal{O}$-terms as above, and recalling the expansions \eqref{eq.definitionf1}, \eqref{eq.definitionf0} and \eqref{eq.definitionf2}, we obtain
\begin{small}
\begin{align*}
&\norm{Y(\omega_1,v_1)-Y(\omega_2,v_2)}_{C^{0,\alpha}(\Sigma)}=\norm{\mathfrak{H}(\omega_1,v_1,Y_1)-\mathfrak{H}(\omega_2,v_2,Y_2)}_{C^{0,\alpha}(\Sigma)}\le c \norm{v_1-v_2}_{C^{0,\alpha}(\Sigma)}+c\norm{\omega_1-\omega_2}_{C^{0,\alpha}(\Sigma)} \\
&+ c \rho^2 \norm{L((\omega_1+v_1,Y_1)-(\omega_2+v_2,Y_2))}_{C^{0,\alpha}(\Sigma)}+ C\norm{Q^{(2)}(\omega_1+v_1,Y_1)-Q^{(2)}(\omega_2+v_2,Y_2)}_{C^{0,\alpha}(\Sigma)}\\
&\le (c+c \rho^2) \norm{(\omega_1,v_1)-(\omega_2,v_2)}_{C^{0,\alpha}(\Sigma)}+ c \rho^2 \norm{Y_1-Y_2}_{C^{0,\alpha}(\Sigma)}.
\end{align*}
\end{small}
For $\rho_3$ small enough, we can always absorb the last summand to the left-hand-side, and get the claimed Lipschitz continuity.
\begin{remark}\label{rem.Hnoncontraction}
From the calculation above, it is clear that the operator $\mathfrak{H}$ is not a contraction in $(\omega,v,Y)$ due of the presence of the linear summands $v_\sigma +\frac{w_\sigma}{R_\sigma}$ at the $0^{th}$-order in $\rho$ in the expansions \eqref{eq.definitionf1},\eqref{eq.definitionf0} and \eqref{eq.definitionf2}, exactly as anticipated. This indeed produce the constant $c$ not multiplied by any $\rho^2$ factor in the formula above.
\end{remark}
$\bullet$ For what regards the third point, consider fixed normal components $(\omega,v) \in  B(0,c_1 \rho^2)\times B(0,c_2 \rho^2)$ and two pairs of boundary data $(e_0,e_2),(e_0',e_2')\in C^{0,\alpha}(\Gamma) \times C^{0,\alpha}(\Gamma)$ for \eqref{eq.generalequiangularity}. Denoting by $Y:=Y(e_0,e_2)$ and $Y':=Y(e_0',e_2')$ their associated \emph{fixed points}, we compute
\begin{small}
\begin{align*}
&\norm{Y(e_0,e_2)-Y(e_0',e_2')}_{C^{0,\alpha}(\Sigma)}=\norm{\mathfrak{H}(Y(e_0,e_2))-\mathfrak{H}(Y(e_0',e_2'))}_{C^{0,\alpha}(\Sigma)} \le c \rho^2 \norm{L((\omega+v,Y)-(\omega+v,Y'))}_{C^{0,\alpha}(\Sigma)}\\
&+ c\norm{Q^{(2)}(\omega+v,Y)-Q^{(2)}(\omega+v,Y')}_{C^{0,\alpha}(\Sigma)}+ c \norm{E-E'}_{C^{0,\alpha}(\Sigma)} \le c \rho^2 \norm{Y-Y'}_{C^{0,\alpha}(\Sigma)}+ c \norm{E-E'}_{C^{1,\alpha}(\Sigma)} \\
&\le c \rho^2 \norm{Y(e_0,e_2)-Y(e_0',e_2')}_{C^{0,\alpha}(\Sigma)}+c \norm{(e_0,e_2)-(e_0',e_2')}_{C^{0,\alpha}(\Gamma) \times C^{0,\alpha}(\Gamma)}.
\end{align*}
\end{small}
Here we have used the continuity of the solutions $E_\sigma$'s of \eqref{eq.systemforE} from the boundary data $(e_0,e_2)$. Once again, for $\rho_3>0$ small enough we can absorb the first summand on the right-hand-side, and obtain the claimed Lipschitz continuity.

$\bullet$ Let us prove that $(\mathfrak{F}(\cdot,\cdot,Y(\cdot,\cdot)),\mathfrak{G}(\cdot,\cdot,Y(\cdot,\cdot)))$ is a contraction of the set $B(0,c_1 \rho^2)\times B(0,c_2 \rho^2)$ for any fixed data $(e_0,e_2) \in B(0,c_4 \rho^2) \times B(0,c_5 \rho^2) \subset C^{0,\alpha}(\Gamma) \times C^{0,\alpha}(\Gamma)$. Firsly, we show the stronger statement that $(\mathfrak{F},\mathfrak{G})$ sends $B(0,c_1 \rho^2)\times B(0,c_2 \rho^2)$ into itself \emph{for any} $Y \in B(0,c_3 \rho^2)$, so in particular the same property is verified when restricted to the solutions $Y(\omega,v)$ constructed above. Recalling the system \eqref{eq.systemwrewritten}, we can argue as above to get
\begin{align*}
&\norm{\mathfrak{F}(\omega,v,Y)}_{C^{0,\alpha}(\Sigma)}\le \norm{\mathfrak{F}(\omega,v,Y)}_{C^{2,\alpha}(\Sigma)} \\
&\le c\norm{ \Pi_{V^{\perp}}\Big( \rho^2 L(\omega+v,Y)
+ Q^{(2)}(\omega+v,Y) -\rho^2 b_\sigma+ \mathcal{O}(\rho^3)\Big)}_{C^{0,\alpha}(\Sigma)}+c \norm{(e_0,e_2)}_{C^{0,\alpha}(\Gamma)\times C^{0,\alpha}(\Gamma)}\\
&\le c\norm{\rho^2 L(\omega+v,Y)+ Q^{(2)}(\omega+v,Y) -\rho^2 b+ \mathcal{O}(\rho^3)}_{C^{0,\alpha}(\Sigma)}+c \norm{(e_0,e_2)}_{C^{0,\alpha}(\Gamma)\times C^{0,\alpha}(\Gamma)}\\
&\le c\rho^2 \norm{(\omega,v,Y)}_{C^{0,\alpha}(\Sigma)}+ c(\norm{(\omega,v,Y)}_{C^{0,\alpha}(\Sigma)})^2 +\rho^2 \norm{b}_{C^{0,\alpha}(\Sigma)}+\norm{ \mathcal{O}(\rho^3)}_{C^{0,\alpha}(\Sigma)}\\
&+c \norm{(e_0,e_2)}_{C^{0,\alpha}(\Gamma)\times C^{0,\alpha}(\Gamma)}\le c \rho^2;\\
&\norm{\mathfrak{G}(\omega,v,Y)}_{C^{0,\alpha}}\le \norm{\rho^2 L(\omega+v,Y)
+ Q^{(2)}(\omega+v,Y) +\rho^2 b+ \mathcal{O}(\rho^3)}_{C^{0,\alpha}} \le c \rho^2.
\end{align*}
We are now ready to prove that $(\mathfrak{F}(\cdot,\cdot,Y(\cdot,\cdot)),\mathfrak{G}(\cdot,\cdot,Y(\cdot,\cdot)))$ is a contraction of $B(0,c_1 \rho^2)\times B(0,c_2 \rho^2)$. Recall once again our convention from Remark \ref{rem.Ogeometric} on purely geometric terms $\mathcal{O}$. Consider two triplets $(\omega_1,v_1,Y_1), (\omega_2,v_2,Y_2) \in \Omega$, where we have set $Y_1:=Y(\omega_1,v_1)$ and $Y_2=Y(\omega_2,v_2)$, and compute
\begin{align*} 
&\norm{\mathfrak{F}(\omega_1,v_1,Y_1)-\mathfrak{F}(\omega_2,v_2,Y_2)}_{C^{0,\alpha}(\Sigma)}\le \norm{\mathfrak{F}(\omega_1,v_1,Y_1)-\mathfrak{F}(\omega_2,v_2,Y_2)}_{C^{2,\alpha}(\Sigma)} \\ &\le c\rho^2 \norm{L((\omega_1+v_1,Y_1)-(\omega_2+v_2,Y_2))}_{C^{0,\alpha}(\Sigma)}\\
&+ c\norm{Q^{(2)}(\omega_1+v_1,Y_1)-Q^{(2)}(\omega_2+v_2,Y_2)}_{C^{0,\alpha}(\Sigma)} \\ &\le c \rho^2 \norm{(\omega_1,v_1,Y_1)-(\omega_2,v_2,Y_2)}_{C^{0,\alpha}(\Sigma)}\le c \rho^2 \norm{(\omega_1,v_1)-(\omega_2,v_2)}_{C^{0,\alpha}(\Sigma)}.
\end{align*}
In the last inequality we have used the Lipschitz continuity from the second point proved above. For the functional $\mathfrak{G}$ the situation is similar
\begin{align*} 
&\norm{\mathfrak{G}(\omega_1,v_1,Y_1)-\mathfrak{G}(\omega_2,v_2,Y_2)}_{C^{0,\alpha}(\Sigma)}\le c\rho^2 \norm{L((\omega_1+v_1,Y_1)-(\omega_2+v_2,Y_2))}_{C^{0,\alpha}(\Sigma)}\\
&+ c\norm{Q^{(2)}(\omega_1+v_1,Y_1)-Q^{(2)}(\omega_2+v_2,Y_2)}_{C^{0,\alpha}(\Sigma)} \\ & \le c \rho^2 \norm{(\omega_1,v_1,Y_1)-(\omega_2,v_2,Y_2)}_{C^{0,\alpha}(\Sigma)}\le c \rho^2 \norm{(\omega_1,v_1)-(\omega_2,v_2)}_{C^{0,\alpha}(\Sigma)}.
\end{align*}
Therefore, for $\rho_3$ small enough, the operator $(\mathfrak{F}(\cdot,\cdot,Y(\cdot,\cdot)),\mathfrak{G}(\cdot,\cdot,Y(\cdot,\cdot)))$ is a contraction as required, and hence there exists a unique fixed point $(\omega_{(e_0,e_2)},v_{(e_0,e_2)},Y_{(e_0,e_2)})\in \mathcal{U}$ of the operator $(\mathfrak{F},\mathfrak{G},\mathfrak{H})$ \textbf{for any boundary data} $(e_0,e_2) \in B(0,c_4 \rho^2) \times B(0,c_5 \rho^2)$.

$\bullet$ Finally, we combine the results obtained in the previous points to find a unique fixed solution $(\omega,v)$ associated to the data $(e_0,e_2)=((e_0)_{(p,\s),\rho}(\omega+v,Y),(e_2)_{(p,\s),\rho}(\omega+v,Y))$, coming from the equi-angularity condition \eqref{eq.perturbedequiangularity}. Choosing suitably the constants $c_1,c_2$ and $c_3$, we see that for $(\omega,v,Y) \in B(0,c_1 \rho^2)\times B(0,c_2 \rho^2) \times B(0,c_3 \rho^2)$, the data $((e_0)_{(p,\s),\rho}(\omega+v,Y),(e_2)_{(p,\s),\rho}(\omega+v,Y))$ belongs to $B(0,c_4 \rho^2) \times B(0,c_5 \rho^2)$.

A completely analogouos argument to the one described above allows us to prove that the data $((e_0)_{(p,\s),\rho}(w,Y),(e_2)_{(p,\s),\rho}(w,Y))$ depends Lipschitz continuously on $(\omega,v,Y)$ with Lispchitz constant as small as we wish after possibly reducing the threshold $\rho_3>0$, see \eqref{eq.perturbedequiangularity}. Using this fact, and the result in the third point (continuity from the boundary data $(e_0,e_2)$), we can carry out the same calculations done in point one to obtain a unique fixed point $Y=Y(\omega,v)$ satisfying the equi-angularity condition \eqref{eq.perturbedequiangularity}. The dependence of the solution $Y$ from $(\omega,v)$ remains Lipschitz continuous, and in particular we can repeat the argument of point four to obtain that the operator $(\mathfrak{F}(\cdot,\cdot,Y(\cdot,\cdot)),\mathfrak{G}(\cdot,\cdot,Y(\cdot,\cdot)))$ is a contraction, and hence the existence and uniqueness of a fixed point $(\omega_{(p,\s),\rho},v_{(p,\s),\rho},Y_{(p,\s),\rho})\in \mathcal{U}$ of the operator $(\mathfrak{F},\mathfrak{G},\mathfrak{H})$ under the boundary conditions \eqref{eq.perturbedequiangularity}, as required.
\vspace{0.5cm}

From  elliptic regularity theory and the bounds in \eqref{eq.ellipticestimatediv}, and since it is constructed as a fixed point, we can bootstrap the regularity of $(\omega_{(p,\s),\rho},v_{(p,\s),\rho},Y_{(p,\s),\rho})\in \mathcal{U}$, and get its smoothness; indeed, the boundary data are still of class $C^{0,\alpha}$, since they depend only of $(w,Y)$ and their first derivatives. Notice that the  smoothness of $v_{(p,\s),\rho}$ could have been deduced a-priori since the space $V$ is classified in \cite{dim}. Moreover, for every $k$ there are constants $c_k$ with
\begin{equation}
\norm{(\omega_{(p,\s),\rho},v_{(p,\s),\rho},Y_{(p,\s),\rho})}_{C^{k,\alpha}(\Sigma)} \le c_k \rho^2.
\end{equation}

In the symmetric case, we set $V:=Ker(\Delta_{\Sigma^0}) \times Ker(R_1 \Delta_{\Sigma^1}+\tfrac{m}{R_1}) \times Ker(R_2 \Delta_{\Sigma^2}+\tfrac{m}{R_2})$, and argue similarly as above to obtain once again a unique smooth solution $(\omega_{(p,\s),\rho},v_{(p,\s),\rho},Y_{(p,\s),\rho})$ with $C^{k,\alpha}(\Sigma)$-norms of order $\mathcal{O}(\rho^2)$ to the system
\begin{equation}
\begin{cases}
\rho R \mathring{H}_\sigma-m + v^\sigma=0, \quad \sigma=1,2;\\
\rho \tfrac{\sqrt{3}}{2} R \mathring{H}_0+ v^0=0;\\
\Div_{\Sigma^\sigma} (Y_\sigma)=f_\sigma, \quad \sigma=0,1,2,
\end{cases}
\end{equation}
under boundary conditions \eqref{eq.junctionGamma}, \eqref{eq.boundaryconditionu}, \eqref{eq.boundaryconditionu} and \eqref{eq.perturbedequiangularity}. We leave the details to the reader. 

In both the asymmetric and symmetric cases, the dependence of the solution $(\omega_{(p,\s),\rho},v_{(p,\s),\rho},Y_{(p,\s),\rho})$ with respect to $(p,\s) \in UTM$ is clearly smooth, as the operators involved depend smoothly on $(p,\s)$, and one can even show that 
\begin{equation}\label{eq.nablaboundsolution}
\norm{\nabla^{k}_{p,\s} \omega_{(p,\s),\rho}}_{C^{2,\alpha}}+\norm{\nabla^{k}_{p,\s} v_{(p,\s),\rho}}_{C^{2,\alpha}}+\norm{\nabla^{k}_{p,\s} Y_{(p,\s),\rho}}_{C^{1,\alpha}} \le d_k \rho^2,
\end{equation}
for some constant $d_k>0$. Inspired by the analogous concept in the case of hypersurfaces enclosing a single volume, studied in \cite{nar}, we give the following definition, corresponding to Definition $1.2$ in \cite{nar}.
\begin{definition}[Pseudo-Double Bubbles]\label{def.pseudobubble}
For any $\rho>0$ and $(p,\s)\in UTM$, we call the double bubble $\Sigma_{(p,\s),\rho}^{\flat}:=\Sigma_{(p,\s),\rho}(\omega_{(p,\s),\rho}+v_{(p,\s),\rho},Y_{(p,\s),\rho})$ the \emph{pseudo-double bubble} at the point $(p,\s)$ and scale $\rho$, where $\omega_{(p,\s),\rho}$, $v_{(p,\s),\rho}$ and $Y_{(p,\s),\rho}$ are constructed  above.
\end{definition}
For the convenience of the reader we recall the equations satisfied by pseudo-double bubbles in the following proposition.
\begin{proposition}\label{prop.pseudoproperty}
For fixed $\rho>0$ and $(p,\s)\in UTM$, the pseudo-bubble $\Sigma_{(p,\s),\rho}^{\flat}$ and the functions $(\omega_{(p,\s),\rho},v_{(p,\s),\rho},Y_{(p,\s),\rho})$ defining it satisfy:
\begin{itemize}
\item The perturbation is admissible, i.e. $(\omega_{(p,\s),\rho}+v_{(p,\s),\rho},Y_{(p,\s),\rho}) \in \mathcal{C}_{amm}$;
\item The perturbation satisfies the strong equi-angularity condition $\mathring{\nu}^0+\mathring{\nu}^1+\mathring{\nu}^2=0$;
\item In the asymmetric case, the mean curvatures $\mathring{H}_\sigma$'s of the pseudo-double bubble $\Sigma_{(p,\s),\rho}^{\flat}$ satisfy 
\begin{equation}
\begin{cases}
\rho R_\sigma \mathring{H}_\sigma-m + v_{(p,\s),\rho}^\sigma=0, \quad \sigma=0,1,2;\\
\Div_{\Sigma^\sigma} ((Y_{(p,\s),\rho})_\sigma)=f_\sigma, \quad \sigma=0,1,2.
\end{cases}
\end{equation}
In the symmetric case they verify
\begin{equation}
\begin{cases}
\rho R \mathring{H}_\sigma-m + v_{(p,\s),\rho}^\sigma=0, \quad \sigma=1,2;\\
\rho \tfrac{\sqrt{3}}{2} R \mathring{H}_0+ v_{(p,\s),\rho}^0=0;\\
\Div_{\Sigma^\sigma} ((Y_{(p,\s),\rho})_\sigma)=f_\sigma, \quad \sigma=0,1,2.
\end{cases}
\end{equation}
\end{itemize}

\end{proposition}

\section{Existence of Constant Mean Curvature Double Bubbles}\label{sec.existence}
In this section we aim to prove Theorems \ref{th.mainasymmetric} and \ref{th.mainsymmetric}, that is the existence of a constant mean curvature double bubble at any critical point of a suitable auxiliary function. We start generalising an idea of Kapouleas (for single enclosed volume) to smooth Euclidean double bubbles, see \cite{kap}. Consider an arbitrary smooth double bubble $\hat{\Sigma}\subset \R^{m+1}$ as in Section \ref{sec.preliminary}; in particular, we keep the same notation for the several associated geometric objects, with a hat $\hat{\cdot}$, and the same convention on the normal vector fields. Given any two constants $H_1>H_2>0$, we  define for such a double bubble a function $\mathcal{E}$ given by
\begin{equation}
\mathcal{E}(\hat{\Sigma}):= |\hat{\Sigma}^0|_m+|\hat{\Sigma}^1|_m+|\hat{\Sigma}^2|_m-H_1 |\hat{B}_1|_{m+1}-H_2 |\hat{B}_2|_{m+1}.
\end{equation}
For any vector field $Z$ in $\R^{m+1}$, we can consider the family of double bubbles $\hat{\Sigma}_t$ generated by the flow of $Z$, and consider the first variation of $\mathcal{E}$ along this family which, once set $H_0:=H_1-H_2$, has the following expression (we drop the volume elements since already implicit in the domains of integration considered)
\begin{equation}
\resizebox{0.93\hsize}{!}{ $\partial_t \mathcal{E}(\hat{\Sigma}_t) \mid_{t=0}= \int_{\hat{\Sigma}^0} (H_0-\hat{H}) \scal{Z}{\hat{N}^0} + \int_{\hat{\Sigma}^1} (H_1-\hat{H}) \scal{Z}{\hat{N}^1} + \int_{\hat{\Sigma}^2} (H_2-\hat{H}) \scal{Z}{\hat{N}^2} -\int_{\hat{\Gamma}} \scal{Z}{\hat{\nu}_0+\hat{\nu}_1+\hat{\nu}_2}.$}
\end{equation}
In case the vector field $Z$ is a Killing vector field, the flow generated by $Z$ acts by isometries on $\R^{m+1}$, and in particular $\mathcal{E}(\hat{\Sigma}_t)$ is constant in $t$, so we deduce
\begin{equation}
\int_{\hat{\Sigma}^0} (H_0-\hat{H}) \scal{Z}{\hat{N}^0} + \int_{\hat{\Sigma}^1} (H_1-\hat{H}) \scal{Z}{\hat{N}^1} + \int_{\hat{\Sigma}^2} (H_2-\hat{H}) \scal{Z}{\hat{N}^2} -\int_{\hat{\Gamma}} \scal{Z}{\hat{\nu}_0+\hat{\nu}_1+\hat{\nu}_2}=0.
\end{equation}
Let us now assume that the smooth double bubble $\hat{\Sigma}$ has mean curvature functions of the form $\hat{H}_\sigma=H_\sigma+\scal{\hat{Z}}{\hat{N}^\sigma}$ for some constants $H_\sigma>0$ verifying $H_1=H_0+H_2$ and a Killing vector field $\hat{Z}$, and also that the sheets $\hat{\Sigma}^\sigma$ meet in an equi-angular way $\hat{\nu}_0+\hat{\nu}_1+\hat{\nu}_2=0$. Notice that these conditions are verified by a pseudo-double bubble in $\R^{m+1}$. Then choosing $Z=\hat{Z}$ we obtain
\begin{equation}
\int_{\hat{\Sigma}^0} (\scal{\hat{Z}}{\hat{N}^0})^2 + \int_{\hat{\Sigma}^1}(\scal{\hat{Z}}{\hat{N}^1})^2 + \int_{\hat{\Sigma}^2} (\scal{\hat{Z}}{\hat{N}^2})^2=0,
\end{equation}
hence $\hat{Z}\equiv 0$ and the three sheets must have constant mean curvatures $\hat{H}_\sigma=H_\sigma$.

We would like to transpose this property from the Euclidean space to a Riemannian context thanks to the use of normal coordinates, similarly to what Pacard and Xu did in the case of a single enclosed volume, see \cite{pac}. { Given a standard double bubble $\Sigma$ whose mean curvature vector is $(H_0,H_1,H_2)$, with $H_1=H_0+H_2$, there exists a threshold $\rho_3$ such that for any $\rho<\rho_3$ and $(p,\s)\in UTM$ we can construct an associated pseudo-double bubble $\Sigma_{(p,\s),\rho}^{\flat}$ as in Section \ref{sec.pseudodb}.} We define the following function $\Psi: UTM \longrightarrow \mathbb{R}$
\begin{equation}
\begin{aligned}
\resizebox{0.93\hsize}{!}{ $\Psi_\rho(p,\s):= \mathcal{E}(\Sigma_{(p,\s),\rho}^{\flat}) 
=\area((\Sigma_{(p,\s),\rho}^{\flat})^0)+\area((\Sigma_{(p,\s),\rho}^{\flat})^1)+\area((\Sigma_{(p,\s),\rho}^{\flat})^2)-\tfrac{H_1}{\rho} \vol((B_{(p,\s),\rho}^{\flat})^1)-\tfrac{H_2}{\rho} \vol((B_{(p,\s),\rho}^{\flat})^2).$}
\end{aligned}
\end{equation}

Coherently to the model flat case just described, we expect the pseudo-double bubble $\Sigma_{(p,\s),\rho}^{\flat}$ to have constant mean curvature vector { $(\tfrac{1}{\rho}H_0,\tfrac{1}{\rho}H_1,\tfrac{1}{\rho}H_2)$ at any critical point $(p,\s)$ of the function $\Psi_\rho$}. This heuristic is confirmed by the following Proposition. 
\begin{proposition}\label{prop.criticalpointCMC}
For any $\rho \in (0,\rho_3)$ and any critical point $(p,\s)$ of $\Psi_\rho$, the pseudo-double bubble $\Sigma_{(p,\s),\rho}^{\flat}$ has constant mean curvatures $(\mathring{H}_0,\mathring{H}_1,\mathring{H}_2) = (\tfrac{1}{\rho}H_0,\tfrac{1}{\rho}H_1,\tfrac{1}{\rho}H_2)$, with $\mathring{H}_1 = \mathring{H}_0 + \mathring{H}_2$, {and satisfies the equi-angularity condition $\mathring{\nu}_0+\mathring{\nu}_1+\mathring{\nu}_2=0$}. 
\end{proposition}
Before proceeding with the proof of this result, we need to recall a deep technical result from \cite{cic}, presented here in a weaker version. Roughly speaking, this result says that given one reference hypersurface with boundary and its image through a diffeomorphism, one can bound effectively the size of the tangential displacement in terms of the Hausdorff distance $d_H$ between the two hypersurfaces. We will use this lemma to compare two pseudo-bubbles when close enough, and write one as a perturbation of the other, while \emph{controlling quantitatively} the tangential component of this perturbation. Given a hypersurface $S \subset \R^{m+1}$ with boundary, and a value $\theta>0$, set $[S]_\theta:=\set{x \in S \mid d_{\R^{m+1}}(x,\partial S)>\theta}$.
\begin{theorem}[Theorem $3.1$, Remark $3.4$ in \cite{cic}]\label{th.technical}
For any natural number $m \ge 2$, real numbers $\alpha \in (0,1]$ and $L>0$, there exist $\mu_0 \in (0,1)$ and $C_0>0$ depending only on $m$, $\alpha$ and $L$ with the following property.
Let $S_0$ be a compact connected $m-$dimensional $C^{2,1}-$manifold with boundary in $\R^{m+1}$ such that its normal vector field $N_{S_0}$ is of class $C^{1,1}$.
Let $S$ be a compact connected $m-$dimensional $C^{1,\alpha}-$manifold with boundary such that its normal vector field $N_S$ satisfies for every $x,y \in S$
\begin{equation}\label{eq.conditionS}
|N_S(x)-N_S(y)|\le L |x-y|^{\alpha}, \quad \text{and} \quad  |\scal{N_S(x)}{(y-x)}| \le L|y-x|^{1+\alpha},
\end{equation}
and for some $\theta \in (0,\mu_0^2)$ we have $d_{H}(S,S_0)\le \theta$. Assume in addition that there exists a $C^{1,\alpha}-$diffeomorphism $f$ between $\partial S_0$ and $\partial S$ with
\begin{equation}\label{eq.conditionf}
\begin{aligned}
\norm{f}_{C^{1,\alpha}(\partial S_0)}&\le L,\\
\norm{f-Id}_{C^{1}(\partial S_0)}&\le \theta,\\
\norm{N_S(f)-N_{S_0}}_{C^{0}(\partial S_0)}&\le \theta,\\
\norm{\nu_S(f)-\nu_{S_0}}_{C^{0}(\partial S_0)}&\le \theta.
\end{aligned}
\end{equation}
Finally, suppose there exists $\psi \in C^{1,\alpha}([S_0]_{\theta})$ such that
\begin{equation}\label{eq.conditionpsi}
[S]_{3 \theta} \subset (Id+\psi N_{S_0})([S_0]_{\theta})\subset S, \quad \norm{\psi N_{S_0}}_{C^{1,\alpha}([S_0]_{\theta})}\le L, \quad \norm{\psi N_{S_0}}_{C^{1}([S_0]_{\theta})}\le \theta.
\end{equation}
Then there exists a $C^{1,\alpha}-$diffeomorphism $F$ between $S_0$ and $S$ such that
\begin{align}
F&=f \quad \text{on} \quad \partial S_0, \label{eq.Fboundary}\\
F&=Id+\psi N_{S_0} \quad \text{on} \quad [S_0]_{\sqrt{\theta}}, \label{eq.Fnormal}\\
\norm{F}_{C^{1,\alpha}(S_0)}&\le C_0,\label{eq.Fholder}\\
\norm{F-Id}_{C^{0}(S_0)}&\le C_0 (d_H(S,S_0)+\norm{f-Id}_{C^{1}(\partial S_0)}+\norm{\psi N_{S_0}}_{C^{0}([S_0]_{\theta})}).\label{eq.Finfty}
\end{align}
\end{theorem}
Let us remark that Theorem $3.1$ in \cite{cic} gives a stronger result under  more general assumptions, but we  preferred to state here this weakened version for the reader's convenience.

The idea of the proof of Proposition \ref{prop.criticalpointCMC} is to apply this theorem using the pseudo double bubble $\Sigma_{(p,\s),\rho}^{\flat}$ as the reference hypersurface $S_0$, and another pseudo double bubble $\Sigma_{(p',\s'),\rho}^{\flat}$ as $S$, for $(p',\s') \in UTM$ arbitrarily close to $(p,\s)$. However, we will first need to reduce everything to an analysis in a fixed ambient Euclidean space, the tangent space $T_p M$; subsequently, this theorem will allow us to patch together the diffeomorphism $f$ of the boundary with the normal perturbation $\psi$, \emph{ensuring uniform bounds} on the tangential component of the diffeomorphism $F$. The price to pay for this gluing, is to get farther from the boundary by a factor $\theta^{-\frac{1}{2}}$; this is due to the annihilation of the tangential component expressed in equation \eqref{eq.Fnormal}. These uniform bounds will be linear in the distance between $(p,\s)$ and $(p',\s')$, so that we will deduce $C^0$-bounds on the vector field generated when $(p',\s')$ approaches $(p,\s)$, thanks to \eqref{eq.Finfty}. Without this theorem, we could have degeneration of the tangential component as $(p',\s')$ tends  to $(p,\s)$ when writing $\Sigma_{(p',\s'),\rho}^{\flat}$ as a perturbation of $\Sigma_{(p,\s),\rho}^{\flat}$, {as in Figure \ref{fig.wild} below.}
\begin{figure}[h]\label{fig.wild}
\includegraphics[scale=2.3]{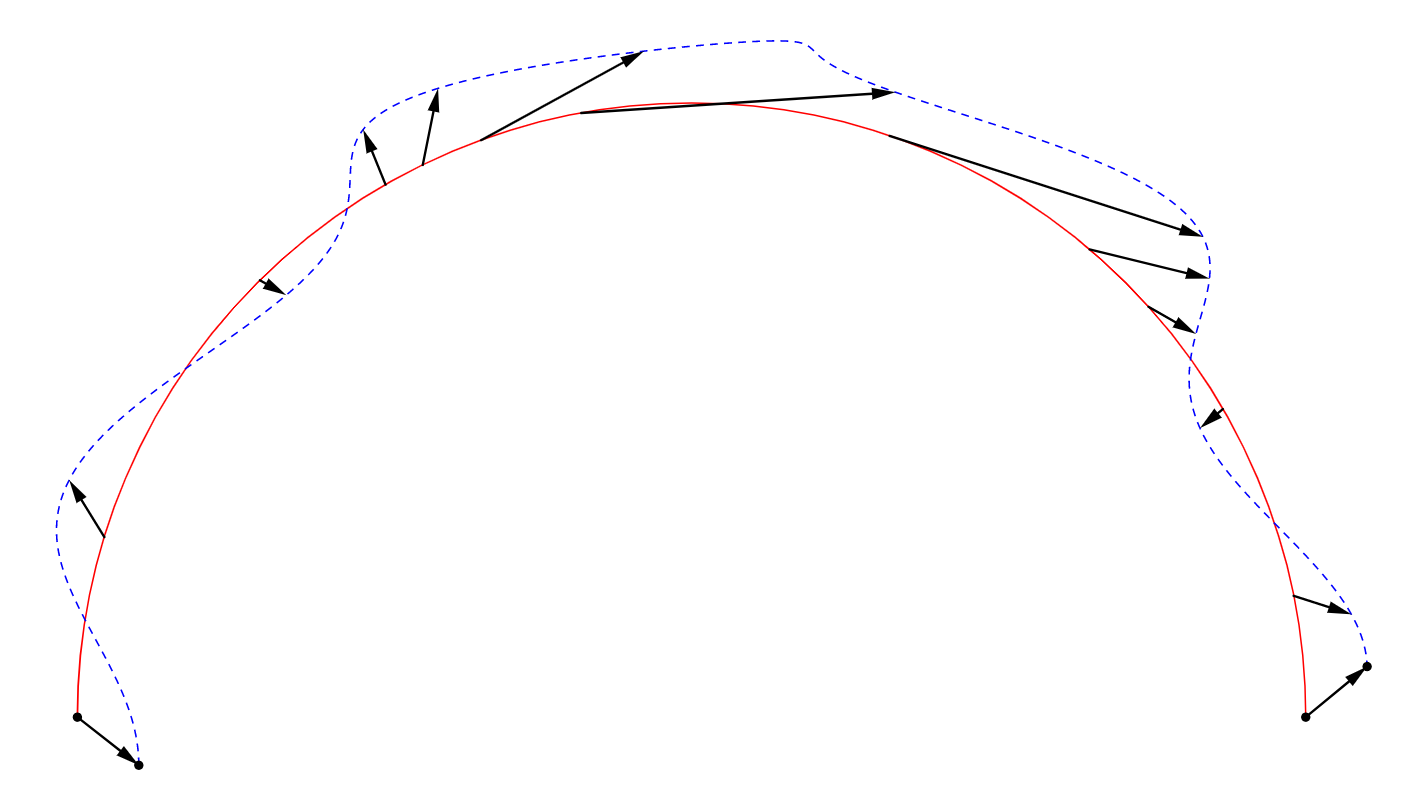}
\caption{Possible wild behaviour of the perturbation}
\end{figure}

We notice that the double bubbles considered are smooth, as remarked in the previous section, so we can find a parameter $L$ as in the statement above, which moreover can be chosen independently of the parameter $\rho<\rho_3$ and the point $(p,\s)\in UTM$ by compactness of the ambient manifold $M$.
\begin{proof}[Proof of Proposition \ref{prop.criticalpointCMC}]
Without loss of generality we will focus our attention on the asymmetric case, since the proof in the symmetric case follows the exact same lines. We drop the index $\sigma=0,1,2$ and treat every quantity as a three-vector. In the course of the proof a constant $c>0$ will appear, whose value may change from line to line, but still remains finite.
Consider a fixed critical point $(p,\s) \in UTM$ of $\Psi_\rho$; since $\Sigma_{(p,\s),\rho}^{\flat}$ is a pseudo-double bubble, by Proposition \ref{prop.pseudoproperty} it has mean curvature vector verifying $\rho R \mathring{H}=m-v_{(p,\s),\rho}$, so it is enough to show that the kernel component $v_{(p,\s),\rho} \in V$ vanishes. Following the argument in \cite{pac}, we will reconnect the differential of $\Psi_\rho$ at $(p,\s)$ to the first variation of the functional $\mathcal{E}$ defined above. In order to do that, consider an arbitrary curve $\gamma:(-\e,\e)\longrightarrow UTM$, which we write as $\gamma(t)=(p(t),\s(t))$, such that $\gamma(0)=(p,\s)$ and set $\tilde{\Xi}:=\gamma'(0) \in T_{(p,\s)}UTM$. It is convenient to identify $\tilde{\Xi}$ with a vector field $\Xi \in T_p M =\R^{m+1}$, so that it induces a family $K_t$ of isometries of $T_p M$; more precisely, we can pick $(K_t)_{t \in (-\e,\e)}$ to be a convex combination of the identity map of $T_p M$ and a composition of a fixed translation and a fixed rotation. Notice that we have $K_t(x)=x+t \Xi +\mathcal{O}(t^2)$. Set $S_0:=\Esp^{-1}(\Sigma_{(p,\s),\rho}^{\flat})$ and consider for $t$ small enough the family of double bubbles $S_t:=K_t^{-1}(\Esp^{-1}(\Sigma_{(p(t),\s(t)),\rho}^{\flat}))$. The composition with the roto-translation $K_t^{-1}$ is due to the annihilation effect described above, see Remark \ref{rem.rotationchoice} after the proof for more details on this choice.

We claim that for every $\theta>0$, the conditions in Theorem \ref{th.technical} above are satisfied for all $t$ small enough. We will then get rid of the parameter $\theta$ considering bounds of infinitesimal nature. Since the distance between $(p,\s)$ and $(p(t),\s(t))$ is approximately $|t| \norm{\Xi}$, by the norm bounds on the covariant derivatives $\nabla_{(p,\s)}\omega_{(p,\s),\rho}$, $\nabla_{(p,\s)}v_{(p,\s),\rho}$ and $\nabla_{(p,\s)}Y_{(p,\s),\rho}$ in equation \eqref{eq.nablaboundsolution}, the Hausdorff distance $d_H(S_t,S_0)$ can easily be bounded by $C \rho^2 t+\mathcal{O}(t^2)$. In fact, the exponential map distorts the distance by an order $\rho^2$, the distance in the Euclidean space would clearly be of order $t$, and the $C^0$-distance between the components $\omega_{(p,\s),\rho}$, $v_{(p,\s),\rho}$ and $Y_{(p,\s),\rho}$ and their $(p',\s')$ analogues is bounded by
\begin{equation}\label{eq.distanceperturbation}
\begin{aligned}
\norm{\omega_{(p',\s'),\rho}-\omega_{(p,\s),\rho}}_{C^0} &\le C |t| \norm{\Xi}\max_t \norm{\nabla_{(p,\s)}\omega_{(p(t),\s(t)),\rho}}_{C^0}\le C |t| \norm{\Xi} \rho^2\\
\norm{v_{(p',\s'),\rho}-v_{(p,\s),\rho}}_{C^0}&\le C |t| \norm{\Xi} \max_t \norm{\nabla_{(p,\s)}v_{(p(t),\s(t)),\rho}}_{C^0}\le C |t| \norm{\Xi} \rho^2,\\
\norm{Y_{(p',\s'),\rho}-Y_{(p,\s),\rho}}_{C^0}&\le C |t| \norm{\Xi} \max_t \norm{\nabla_{(p,\s)}Y_{(p(t),\s(t)),\rho}}_{C^0} \le C |t| \norm{\Xi} \rho^2.
\end{aligned}
\end{equation}
Since $\partial S_0$ and $\partial S_t$ are two smooth $S^{m-1}$, exactly as in \cite{pac}, we deduce the existence of normal diffeomorphisms $f_t: \partial S_0 \longrightarrow \partial S_t$; here the word ''normal'' means that the image $f_t(x) \in \partial S_t$ of any point $x \in \partial S_0$ lies in the normal space $N_x \partial S_0$ seen as a subset of $T_p M$. Moreover, arguing as above, the inequalities in \eqref{eq.distanceperturbation} implies that $\norm{f_t-Id}_{C^1(\partial S_0)} \le C |t| \norm{\Xi} \rho^2+\mathcal{O}(t^2)$, $\norm{N_{S_t} (f_t)-N_{S_0}}_{C^0(\partial S_0)} \le C |t| \norm{\Xi} \rho^2+\mathcal{O}(t^2)$ and $\norm{\nu_{S_t} (f_t)-\nu_{S_0}}_{C^0(\partial S_0)} \le C |t| \norm{\Xi} \rho^2+\mathcal{O}(t^2)$, so that \eqref{eq.conditionf} holds whenever $C |t| \norm{\Xi} \rho^2 +\mathcal{O}(t^2) \le \theta$. We can use the same argument to show \eqref{eq.conditionpsi}, again under the condition  $C |t| \norm{\Xi} \rho^2 +\mathcal{O}(t^2) \le \theta$.

Therefore, we can appeal to Theorem \ref{th.technical} and obtain for $|t|$ small enough the existence of a family of $C^{1,\alpha}$-diffeomorphisms $F_t$ between $S_0$ and $S_t$ satisfying all the conditions \eqref{eq.Fboundary}, \eqref{eq.Fnormal}, \eqref{eq.Fholder} and \eqref{eq.Finfty}. In particular, from \eqref{eq.Finfty} we deduce after applying a $t$-derivative and setting $Z:=\tfrac{d}{dt} \mid_{t=0} F_t \in T_p M$
\begin{equation}\label{eq.boundonZ}
\norm{F_t-Id}_{C^0(S_0)} \le C |t| \norm{\Xi} \rho^2 +\mathcal{O}(t^2) \Rightarrow \norm{Z}_{C^0(S_0)} \le C \norm{\Xi} \rho^2.
\end{equation}
Compose the diffeomorphism $F_t$ with $K_t$ and $\Esp$ to construct a family of diffeomorphisms $(\mathcal{F}_t)_{|t|\in (-\e,\e)}$ between the two pseudo-double bubbles $\Sigma_{p,\s}^{\flat}$ and $\Sigma_{(p(t),\s(t)),\rho}^{\flat}$, defined by $\mathcal{F}_t:=\Esp \circ K_t \circ F_t \circ \Esp^{-1}$. Once set $\mathcal{Z}:=\tfrac{d}{dt} \mid_{t=0} \mathcal{F}_t$ the vector field induced by the family on $\Sigma_{p,\s}^{\flat}$, and $X$ the parallel transport of the vector { $\tilde{\Xi}$ along geodesic in $UTM$ emanating from $(p,\s)$ (to be precise, its identification in $TM$),} equation \eqref{eq.boundonZ} implies the bound (the composition with $K_t$ gives a summand $-\Xi$)
\begin{equation}\label{eq.pxbound}
\norm{\mathcal{Z}-X}_{G} \le C \rho^2 \norm{\Xi}.
\end{equation}

We can finally rewrite the differential of $\Psi_\rho$ as a first variation of $\mathcal{E}$ and obtain
\begin{equation}
\begin{aligned}
0=&d \Psi_{p,\s}(\tilde{\Xi})= \partial \mathcal{E}(\Sigma_{(p(t),\s(t)),\rho}^{\flat})\mid_{t=0}=\tfrac{d}{dt} \mid_{t=0} \mathcal{E}(\mathcal{F}_t(\Sigma_{p,\s}^{\flat}))= \int_{(\Sigma_{(p,\s),\rho}^{\flat})^0} (\tfrac{m}{\rho R_0}-\mathring{H}_0) G(\mathcal{Z},\mathring{N}^0) \\
&+ \int_{(\Sigma_{(p,\s),\rho}^{\flat})^1} (\tfrac{m}{\rho R_1}-\mathring{H}_1) G(\mathcal{Z},\mathring{N}^1) + \int_{(\Sigma_{(p,\s),\rho}^{\flat})^2} (\tfrac{m}{\rho R_2}-\mathring{H}_2) G(\mathcal{Z},\mathring{N}^2) -\int_{\Gamma_{(p,\s),\rho}^{\flat}} G(\mathcal{Z},\mathring{\nu_0}+\mathring{\nu_1}+\mathring{\nu_2}).
\end{aligned}
\end{equation}
First of all, we use \eqref{eq.pxbound} to get
\begin{equation}\label{eq.XsubstituteZ}
\resizebox{0.99\hsize}{!}{ $
\sum_{\sigma=0,1,2} \int_{(\Sigma_{(p,\s),\rho}^{\flat})^\sigma} (\tfrac{m}{\rho R_\sigma}-\mathring{H}_\sigma) G(X,\mathring{N}^\sigma) -\int_{\Gamma_{(p,\s),\rho}^{\flat}} G(\mathcal{Z},\mathring{\nu_0}+\mathring{\nu_1}+\mathring{\nu_2})\le \sum_{\sigma=0,1,2} \int_{(\Sigma_{(p,\s),\rho}^{\flat})^\sigma} (\tfrac{m}{\rho R_\sigma}-\mathring{H}_\sigma) G(X-\mathcal{Z},\mathring{N}^\sigma).$}
\end{equation}
We would like to rewrite the above identity  in terms of integrals over the standard double bubble $\Sigma:=\Esp^{-1} (\Sigma_{(p,\s),\rho}^{\flat}) \subset T_p M$. To this aim,  recall that from the expansions \eqref{eq.perturbedfirst} and \eqref{eq.perturbedarea} obtained in the previous section, { as well as from  Proposition \ref{prop.pseudoproperty}}, we can approximate the quantities involved as follows
\begin{small}
\begin{equation*}
\begin{aligned}
G(X,\mathring{N}^\sigma)=\scal{\Xi}{N^\sigma}+\mathcal{O}(\rho^2)\norm{\Xi}; \qquad  \mathring{\nu_0}+\mathring{\nu_1}+\mathring{\nu_2}=0; \\ \tfrac{m}{\rho R_\sigma}-\mathring{H}_\sigma= \tfrac{1}{\rho R_\sigma} v_{(p,\s),\rho}^\sigma; \qquad  \area((\Sigma_{(p,\s),\rho}^{\flat})^\sigma)=\rho^m |\Sigma^\sigma|_m+\mathcal{O}(\rho^{m+2}).
\end{aligned}
\end{equation*}
\end{small}
Substituting in the expression above we arrive at
\begin{equation}\label{eq.firstvariationeuclidean}
\sum_{\sigma=0,1,2} \rho^{m}\int_{\Sigma^\sigma} \tfrac{1}{\rho R_\sigma} v_{(p,\s),\rho}^\sigma \scal{\Xi}{N^\sigma}\le c \rho^2 \norm{\Xi}\sum_{\sigma=0,1,2} \rho^{m}\int_{\Sigma^\sigma} \tfrac{1}{\rho R_\sigma} |v_{(p,\s),\rho}^\sigma|. 
\end{equation}
{ Notice that we got rid of the troublesome boundary term in \eqref{eq.XsubstituteZ}, see Remark \ref{rem.boundaryterm}.}
We have finally reconducted the problem to a situation similar to the one in \cite{pac}. { Inspired by \cite{pac}, we will deduce the triviality of the kernel component $v_{(p,\s),\rho}$ by exploiting the multiplicative factor $\norm{\Xi}$ in \eqref{eq.firstvariationeuclidean}.}

{This identity holds for a general $\tilde{\Xi}$, so for any vector field $\Xi$ on $T_p M$ generated by translations and rotations. } Arguing by scaling and compactness, as in \cite{pac}, we notice that for every such vector field we must have 
\begin{equation}
\int_{\Sigma^\sigma} \tfrac{1}{R_\sigma}  (\scal{\Xi}{N^\sigma})^2=c(\phi^\sigma,R_\sigma,m) \norm{\Xi}^2,
\end{equation}
unless $\Xi=\Xi_\sigma$ is inducing a rotation around the center $C^\sigma$ of $\Sigma^\sigma$, in which case the integral above annihilates. On the other hand, in this case we must have
\begin{equation}
\int_{\Sigma^\tau} \tfrac{1}{R_\tau}(\scal{\Xi}{N^\tau})^2=c(\phi^\tau,R_\tau,m) \norm{\Xi}^2 \quad \forall \tau \neq \sigma.
\end{equation}
In any case, we can deduce that for every vector field $\Xi$ as above
\begin{equation}
\sum_{\sigma=0,1,2} \int_{\Sigma^\sigma} \tfrac{1}{R_\sigma} (\scal{\Xi}{N^\sigma})^2=c(\phi^0,\phi^1,\phi^2,R_0,R_1, R_2,m) \norm{\Xi}^2,
\end{equation}
and hence, by the volume element expansion induced by \eqref{eq.metricnormalcoord}, we the following inequality
\begin{equation}
\sum_{\sigma=0,1,2} \int_{\Sigma^\sigma} \tfrac{1}{R_\sigma} (\scal{\Xi}{N^\sigma})^2 \ge \frac{c}{2} \norm{\Xi}^2.
\end{equation}
In particular, equation \eqref{eq.firstvariationeuclidean} becomes
\begin{equation}
\sum_{\sigma=0,1,2} \int_{\Sigma^\sigma} \tfrac{1}{\rho R_\sigma} v_{(p,\s),\rho}^\sigma \scal{\Xi}{N^\sigma}\le c \rho^2 \bigg(\sum_{\sigma=0,1,2} \int_{\Sigma^\sigma} \tfrac{1}{R_\sigma}(\scal{\Xi}{N^\sigma})^2 \bigg)^{\frac{1}{2}}\bigg( \sum_{\sigma=0,1,2} \int_{\Sigma^\sigma} \tfrac{1}{\rho R_\sigma} |v_{(p,\s),\rho}^\sigma| \bigg). 
\end{equation}
However, recall that $v_{(p,\s),\rho}$ must itself be obtained by infinitesimal translations and rotations by the main result in \cite{dim}, and therefore there exists such a vector field $\Xi_{(p,\s),\rho}$ verifying the identity $\scal{\Xi_{(p,\s),\rho}}{N^\sigma}=v_{(p,\s),\rho}^\sigma$ for every $\sigma=0,1,2$. With this choice of $\Xi$ (hence $\tilde{\Xi}$ chosen accordingly) in the equation above, we obtain
\begin{eqnarray}
\bigg(  \sum_{\sigma=0,1,2} \int_{\Sigma^\sigma} \tfrac{1}{R_\sigma} (\scal{\Xi_{(p,\s),\rho}}{N^\sigma})^2 \bigg)^{\frac{1}{2}} \le c \rho^2 \bigg( \sum_{\sigma=0,1,2} \int_{\Sigma^\sigma} \tfrac{1}{R_\sigma} |\scal{\Xi_{(p,\s),\rho}}{N^\sigma}| \bigg) \nonumber \\ \le c \rho^2 \bigg(  \sum_{\sigma=0,1,2} \int_{\Sigma^\sigma} \tfrac{1}{R_\sigma} (\scal{\Xi_{(p,\s),\rho}}{N^\sigma})^2 \bigg)^{\frac{1}{2}}.
\end{eqnarray}
For $\rho$ small enough this implies that the integrals are zero. Since the integrands are non-negative, we must have $v_{(p,\s),\rho}^\sigma=\scal{\Xi_{(p,\s),\rho}}{N^\sigma}\equiv 0$ at any point of $\Sigma^\sigma$ and for any $\sigma=0,1,2$, proving the statement about the CMC property. 
	
To conclude the proof of  Proposition, we notice that the condition  $\mathring{H}_1 = \mathring{H}_0 + \mathring{H}_2$  as well as the equi-angularity condition $\mathring{\nu}_0+\mathring{\nu}_1+\mathring{\nu}_2=0$ here imposed, would be standard consequences of the formula of first variation of the area under the volume-constraint of the 2-cluster.
\end{proof}

{
\begin{remark}\label{rem.boundaryterm}
In the course of the proof, the strong equi-angularity \eqref{eq.perturbedequiangularity} imposed had the crucial role of deleting the problematic integral on the boundary $\Gamma_{(p,\s),\rho}^{\flat}$, which cannot be proven to be proportional to $\Xi$ for the solution $(\omega_{(p,\s),\rho},v_{(p,\s),\rho},Y_{(p,\s),\rho})$.
\end{remark}
\begin{remark}\label{rem.rotationchoice}
In the proof above, we have opted to consider the family of roto-translated double bubbles $S_t$ instead of the simpler ones $S_t':=\Esp^{-1}(\Sigma_{(p,\s),\rho}^{\flat})$. The major reason for doing this is that, as briefly mentioned after Theorem \ref{th.technical}, the same theorem would produce an almost normal diffeomorphism between the reference double bubble $S_0$ and $S_t'$; heuristically, the effect of the roto-translation would tend to concentrate more and more on the boundary of $S_0$, making it much harder to show an estimate of the form \eqref{eq.pxbound}, with the crucial multiplicative term $\norm{\Xi}$ on the right hand side.
\end{remark}
}
We are finally ready to prove the our main Theorems  \ref{th.nondeg},  \ref{th.mainasymmetric} and \ref{th.mainsymmetric}.

\begin{proof}[Proof of Theorem \ref{th.nondeg}]
We begin by proving \eqref{eq.approximateasym}: in order to do this, we will exploit the expansions obtained in Subsections \ref{sub.volume}, \ref{sub.area}, \ref{sub.perturbedvolume} and \ref{sub.perturbedarea}.

Using \eqref{eq.perturbedvolume1} and \eqref{eq.perturbedvolume2}, as well as the balance equation $R_0^{-1}=R_1^{-1}-R_2^{-1}$ we get
\begin{equation*}
	\begin{aligned}
		\tfrac{m}{\rho R_1} \vol((B_{(p,\s),\rho}^{\flat})^1)&+\tfrac{m}{\rho R_2} \vol((B_{(p,\s),\rho}^{\flat})^2)= \tfrac{m}{\rho R_1} \vol(B_{(p,\s),\rho}^1)+\tfrac{m}{\rho R_2} \vol(B_{(p,\s),\rho}^2)\\
		&-\rho^m \sum_{\sigma=0,1,2} \int_{\Sigma^\sigma} \tfrac{m w_\sigma}{R_\sigma}-\tfrac{m}{m+1}div_{\Sigma^\sigma}(Y_\sigma)+Q^{(2)}(w,Y) +\rho^2 L(w,Y) d\mu_{\Sigma^\sigma}. 
	\end{aligned}
\end{equation*}
Notice that from \eqref{eq.perturbedarea}
\begin{small}
	\begin{equation*}
		\area(\Sigma_{(p,\s),\rho}^{\flat})=\sum_{\sigma=0,1,2} \area(\Sigma^\sigma_{(p,\s),\rho})-\rho^{m} \int_{\Sigma^\sigma} m\tfrac{w_\sigma}{R_\sigma}-div_{\Sigma^\sigma}(Y_\sigma)+ \rho^2 L(w,Y) + Q^{(2)}(w,Y)d\mu_{\Sigma^\sigma} + \mathcal{O}(\rho^{m+3}),
	\end{equation*}
\end{small}
hence we deduce (recall that both $w_{(p,\s),\rho}$ and $Y_{(p,\s),\rho}$ are of order $\mathcal{O}(\rho^2)$, so $\rho^2 L(w,Y)$ and $Q^{(2)}(w,Y)$ can be absorbed)
\begin{equation}\label{eq.perturbedPsi}
	\begin{aligned}
		\Psi_\rho(\Sigma_{(p,\s),\rho}^{\flat})&=\Psi_\rho(\Sigma_{(p,\s),\rho})+\rho^m \sum_{\sigma=0,1,2} \int_{\Sigma^\sigma} \tfrac{1}{m+1}div_{\Sigma^\sigma}(Y_\sigma)+\mathcal{O}(\rho^{m+3})\\
		&=\Psi_\rho(\Sigma_{(p,\s),\rho})-\rho^m  \tfrac{1}{m+1} \int_{\Gamma} \sum_{\sigma=0,1,2}u_\sigma+\mathcal{O}(\rho^{m+3})=\Psi_\rho(\Sigma_{(p,\s),\rho})+\mathcal{O}(\rho^{m+3}),
	\end{aligned}
\end{equation}
where we have used the divergence theorem and the boundary condition \eqref{eq.boundaryconditionu} to delete the sum of the $u_\sigma$'s. We now derive an expansion for $\Psi_\rho(\Sigma_{(p,\s),\rho})$. Using \eqref{eq.volumedecomposition} we see 
\begin{equation}
	\tfrac{m}{\rho R_1} (V_1)_{(p,\s),\rho}+\tfrac{m}{\rho R_2}(V_2)_{(p,\s),\rho}= \sum_{\sigma=0,1,2} \tfrac{m}{\rho R_\sigma}  \vol((P_{(p,\s),\rho})^\sigma),
\end{equation}
so that
\begin{equation}
	\Psi_\rho(\Sigma_{(p,\s),\rho})=\sum_{\sigma=0,1,2}\area(\Sigma^\sigma_{(p,\s),\rho})- \tfrac{m}{\rho R_\sigma}  \vol((P_{(p,\s),\rho})^\sigma).
\end{equation}
Let us consider the case $\sigma=1$, since the other cases can be treated nearly identically. Clearly, the critical points remain fixed if we multiply $\Psi_\rho$ by $\rho^{-m}$. From equations \eqref{eq.volumeP1} and \eqref{eq.areaasymmetric1} we obtain
\begin{equation}
	\begin{aligned}
		&\rho^{-m}\Big(\area(\Sigma^1_{(p,\s),\rho})- \tfrac{m}{\rho R_1}  \vol(P_{(p,\s),\rho}^1) \Big)= (|\Sigma^1|- \tfrac{m}{\rho R_1} |P^1|)+\tfrac{\rho^{2}}{6} R_1^{m+2} \omega_m \Big[ I_{m+1}(\phi^1)\Sc(p)\\
		&+(m \cos^2(\phi^1)I_{m-1}(\phi^1)-\sin^m(\phi^1)\cos(\phi^1))\Ric(\s,\s) \Big]+ \mathcal{O}(\rho^{3})+\tfrac{\rho^2}{6}\omega_m R_1^{m+2} \bigg[ \tfrac{m I_{m+3}(\phi^1)}{m+2} \Sc(p)\\
		&+ m\Big(\tfrac{m+3}{m+2} I_{m+3}(\phi^1)-I_{m+1}(\phi^1) \sin^2(\phi^1) \Big)\Ric(\s,\s)\bigg]+\mathcal{O}(\rho^3).
	\end{aligned}
\end{equation}
The first summand $|\Sigma^1|- \tfrac{m}{\rho R_1} |P^1|$ does not depend on $(p,\s)$, so we can get rid of it and further divide everything by $\tfrac{\rho^2}{6} \omega_m$ to get the function $\Phi_\rho$ given by (restore the sum over $\sigma=0,1,2$)
\begin{equation}\label{eq.definitionPhi}
	\Phi_\rho((p,\s)):= \tfrac{6}{\omega_m \rho^2} \Big(\rho^{-m} \Psi_\rho(p,\s)-\sum_{\sigma=0,1,2} \big( |\Sigma^\sigma|- \tfrac{m}{\rho R_\sigma} |P^\sigma|\big) \Big), 
\end{equation}
which verifies
\begin{small}
\begin{equation*}
	\begin{aligned}
		&\Phi_\rho((p,\s))\underset{C^k(UTM)}{\sim}\sum_{\sigma=0,1,2} R_\sigma^{m+2} \Big[ \Big(I_{m+1}(\phi^\sigma)+\tfrac{m I_{m+3}(\phi^\sigma)}{m+2} \Big)\Sc(p)\\
		+\Big((m \cos^2(\phi^\sigma)I_{m-1}(\phi^\sigma)&-\sin^m(\phi^\sigma)\cos(\phi^\sigma))+m(\tfrac{m+3}{m+2} I_{m+3}(\phi^\sigma)-I_{m+1}(\phi^\sigma) \sin^2(\phi^\sigma)) \Big)\Ric(\s,\s) \Big]\\
		=\sum_{\sigma=0,1,2} R_\sigma^{m+2} \Big[ \Big(I_{m+1}(\phi^\sigma)+&\tfrac{m I_{m+3}(\phi^\sigma)}{m+2} \Big)\Sc(p)+\Big((2m+1)I_{m+1}(\phi^\sigma) -\tfrac{2m+2}{m+2}\sin^{m+2}(\phi^\sigma) \cos(\phi^\sigma)\Big)\Ric(\s,\s) \Big]\\
		=:\Sc(p) A(m,H_0,H_1,H_2) -&\Ric(\s,\s) B(m,H_0,H_1,H_2)
	\end{aligned}
\end{equation*}
\end{small}
up to an order $\mathcal{O}(\rho)$, as required. The dependency of the constants $A$ and $B$ come from the fact that the standard double bubble $\Sigma$ is uniquely determined by the mean curvatures $H_\sigma'$s, after a rigid motion.

We also notice that if a contraction depends smoothly on a set of parameters, also its fixed point 
does, see e.g. Section 2.6 in \cite{bre}. This applies in particular to the construction of pseudo-bubbles 
in  Section \ref{sec.pseudodb} depending on their location and orientation, allowing to extend the 
estimate \eqref{eq.approximateasym} up to any number of derivatives in $p$ and $\s$.

\

Let now $p$, $\mu$ be as in the statement of  Theorem \ref{th.nondeg},  and let $n$  denote the multiplicity of $\mu$ 
as an eigenvalue of $\Ric_p$. Let us assume first that $n > 1$. 
By continuity of the eigenvalues of $\Ric$, we can find a neighborhood $\mathcal{V}$ of $p$  and  numbers $0  < \delta_0 \ll \delta_1$ 
such that every eigenvalue $\eta$ of $\Ric_q$ satisfies either $|\eta - \mu| < \delta_0$ or $|\eta - \mu| > \delta_1$ for all 
$q \in \mathcal{V}$. We call $\mathcal{H}_q$ the direct sum of the eigenspaces corresponding to eigenvalues of the 
first kind, and $\tilde{\mathcal{H}}_q$  the direct sum of the eigenspaces corresponding to those of the 
second kind. Notice that these two subspaces are orthogonal, that $\mathcal{H}_q$ depends smoothly on $q$ and 
that $\mathcal{H}_q$ has dimension $n$ for all $q \in \mathcal{V}$. 

For each $q \in \mathcal{V}$ we denote by $UTM_q$ the fiber of $UTM$ over $q$, and by 
$S_q$ the $(n-1)$-dimensional sphere given by $\mathcal{H}_q \cap UTM_q$.  
By the expansion \eqref{eq.approximateasym} (and its differential counterpart), $S_q$ is a 
manifold of approximate critical points for $\Phi_\rho|_{UTM_q}$, and  $\Phi_\rho|_{UTM_q}$ 
is orthogonally non-degenerate on $S_q$, and hence near $S_q$ one can control the normal 
component of the spherical gradient $\nabla (\Phi_\rho|_{UTM_q})$ using displacements normal 
to $S_q$. Similarly, since $p$ is non-degenerate for the scalar curvature, if $\Pi$ denotes the 
natural projection $UTM \to M$, one can control $\Pi_* \nabla \Phi_\rho$ by properly varying the 
base point $q$, see Figure \ref{fig.6}.

\begin{figure}[h] \label{fig.6}
	\begin{minipage}[l]{0.48\linewidth}
		\includegraphics[scale=1.1]{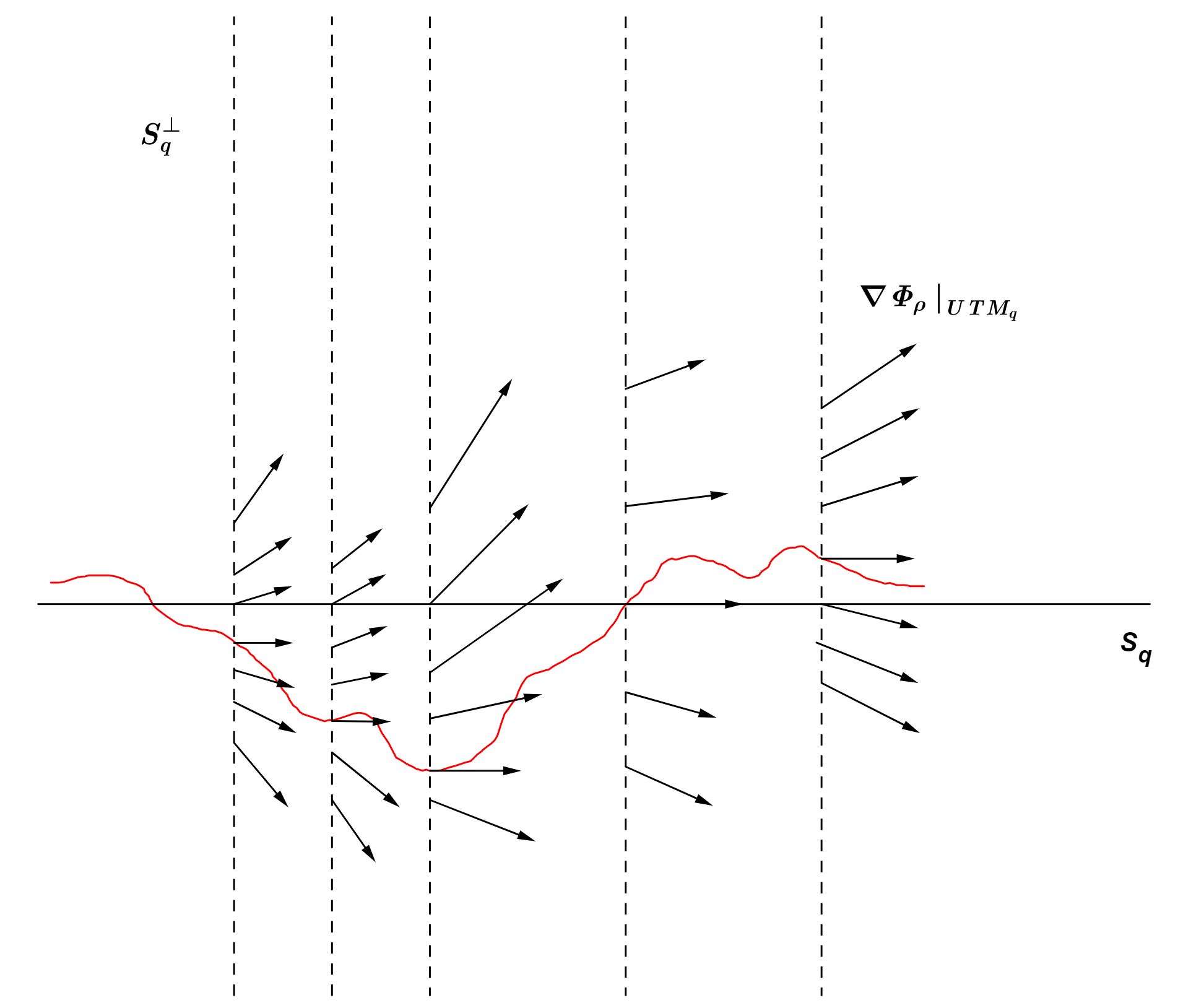}
	\end{minipage} \hspace{0.5cm}
	\begin{minipage}[r]{0.3\linewidth}
		\includegraphics[scale=0.3]{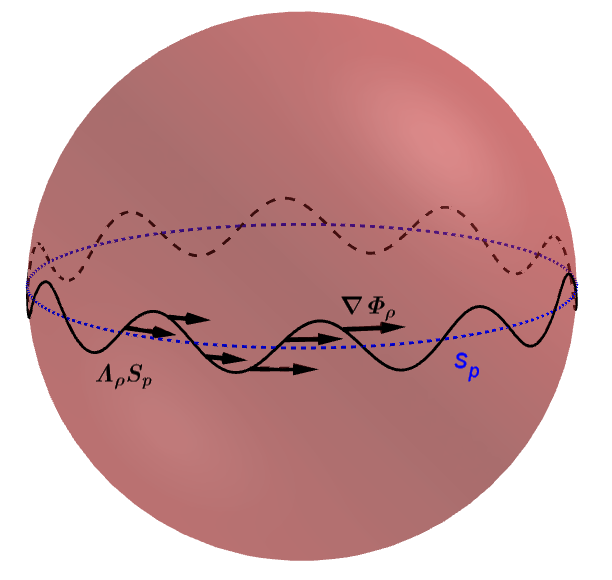} 
	\end{minipage}
	\caption{The behaviour of $\nabla  \Phi_\rho|_{UTM_q}$ near $S_p$ and $\nabla \Phi_\rho|_{UTM_q}$ 
	on $\Lambda_\rho S_p$}
\end{figure}

For these reasons, for $\rho$ small one can find an embedding $\Lambda_\rho : S_p \to UTM$ 
with the following properties 
\begin{equation}\label{eq.red-sphere}
	  \begin{cases}
  	\Pi_* \nabla \Phi_\rho (\Lambda_\rho \vartheta) = 0 & \hbox{ for every } \vartheta \in S_p; \\
  	 \nabla \Phi_\rho (\Lambda_\rho \vartheta) \hbox{ is parallel to } S_{\Pi \Lambda_\rho \vartheta} 
  	& \hbox{ for every } \vartheta \in S_p, 
  \end{cases}
\end{equation}
with the image of $\Lambda_\rho$ converging smoothly to the identity map of $S_p$ as 
 $\rho \to 0$. 
Concerning the second condition in \eqref{eq.red-sphere}, since $\Pi_* \nabla \Phi_\rho (\Lambda_\rho \vartheta) = 0$, 
$\nabla \Phi_\rho (\Lambda_\rho \vartheta)$ is tangent to the fiber over $\Pi \Lambda_\rho \vartheta$, 
so we are viewing the gradient as an element of $T_{\Pi \Lambda_\rho \vartheta} M$.

Reasoning as for the proof of Proposition \ref{prop.criticalpointCMC} and using \eqref{eq.red-sphere}, one can show that a 
critical point of the restriction $\Phi_\rho|_{\Lambda_\rho(S_p)}$ is indeed a free critical point of 
$\Phi_\rho$. 
Applying then Proposition \ref{prop.criticalpointCMC}, the existence of a CMC double-bubble 
aligned along $S_p$ follows, proving the theorem. 

The case of multiplicity $n = 1$ can be directly proved via the implicit function theorem on $\Phi_\rho$, 
which also yields uniqueness once an orientation in the simple eigenspace is chosen, 
by the non-degeneracy of  critical point for the reduced functional. In case 
$H_0 = 0$, the pseudo-bubbles stay invariant by revesing their axis, as noticed in the next proof, 
which gives then uniqueness of critical configurations oriented in the given eigenspace.  
\end{proof}

\begin{proof}[Proof of Theorem \ref{th.mainasymmetric}]
The first result simply follows from Proposition \ref{prop.criticalpointCMC}, the fact that 
the domain of $\Phi_\rho$ is $UTM$, and from Lusternik-Schnirelman's theory. 
	
If $p$ is a non-degenerate critical point of the scalar curvature and $\mu$ an eigenvalue 
of the Ricci tensor, 	we apply the same final argument in the proof of Theorem 
\ref{th.nondeg}, noticing that the category of $\Lambda_\rho S_p$ is equal to $2$. 
The fact that $H_1 \neq H_2$ avoids the possible geometric equivalence 
of multiple critical points of $\Phi_\rho|_{\Lambda_\rho S_p}$, as it might indeed 
happen for antipodal critical points in the symmetric case. 
\end{proof}

\begin{proof}[Proof of Theorem \ref{th.mainsymmetric}]
We begin by proving the  bounds in equation  \eqref{eq.approximatesym}. 
In the symmetric case, we still define the function $\Phi_\rho$ through equation \eqref{eq.definitionPhi}. This time, we use equations \eqref{eq.perturbedvolume1sym}, \eqref{eq.perturbedvolume2sym}, \eqref{eq.perturbedarea}, \eqref{eq.perturbedareadisk}, \eqref{eq.geodesicvolumesym} and \eqref{eq.areasymmetric} to finally arrive to
the expression
\begin{equation*}
\begin{aligned}
R^{m+2}\Phi^{sym}_\rho((p,\s)):=&\Phi_\rho((p,\s))\underset{C^k(\mathbb{P}TM)}{\sim} R^{m+2} \Big(\tfrac{1}{m+2} \big(\tfrac{\sqrt{3}}{2} \big)^{m+2}+2 I_{m+1}(\tfrac{2}{3}\pi)+\tfrac{m I_{m+3}(\tfrac{2}{3}\pi)}{m+2} \Big)\Sc(p)\\
&+R^{m+2}\Big(\tfrac{m}{2}I_{m-1}(\tfrac{2}{3}\pi)+\tfrac{2m+1}{2m+4}\big(\tfrac{\sqrt{3}}{2} \big)^m-\tfrac{3m}{2}I_{m+1}(\tfrac{2}{3}\pi) +\tfrac{m(m+3)}{m+2}I_{m+3}(\tfrac{2}{3}\pi)\Big)\Ric(\s,\s)\\
=:&R^{m+2} \big( \Sc(p) A^{sym}(m) -\Ric(\s,\s) B^{sym}(m) \big)
\end{aligned}
\end{equation*}
up to an order $\mathcal{O}(\rho^2)$, as required. This better approximation comes from Remarks \ref{rem.volumesymmetricbetter} and \ref{rem.areasymmetricbetter}, valid for symmetric geodesic double bubbles, and the fact that $\rho^2 L(w,Y)$ and $Q^{(2)}(w,Y)$ are of order $\mathcal{O}(\rho^4)$, leading together to a better approximation in \eqref{eq.perturbedPsi}. This yields \eqref{eq.approximatesym}. 

We next discuss the validity of \eqref{eq.symmetry-Phi-rho}: this follows from the fact that, by uniqueness 
of the fixed point, the perturbation produced in Section \ref{sec.pseudodb} for a couple $(p,-\s)$ 
is the antipodal of that for $(p,\s)$. The equality in \eqref{eq.symmetry-Phi-rho} then follows, 
and  a straightforward application of Proposition \ref{prop.criticalpointCMC} and critical point theory gives the existence of at least $cat(\mathbb{P}TM)$ constant mean curvature double bubbles respectively in the   symmetric case.

When $p$ is a non-degenerate critical point of the scalar curvature, if $\Lambda_\rho$ is as in the 
proof of Theorem \ref{th.nondeg}, the latter reasoning yields 
also antipodal symmetry of $\Phi_\rho|_{\Lambda_\rho}$: one can then use Krasnoselski's 
genus to find the desired multiplicity result, 
see e.g. Chapter 10 in \cite{ambmal2}. 
\end{proof}

Concluding this paper, we recall that some of the double bubbles we constructed are candidates solutions for the isoperimetric problem under two small volume constraints. It would be interesting to characterize these minima as some of the configurations we constructed here. This has already been done for ambient flat tori in \cite{alv} and in two-dimensional general ambient surfaces in \cite{morwic}. The study of further cases will be the scope of a forthcoming paper.

%
%
%

\end{document}